\documentclass{article}
\usepackage{graphicx} 
\usepackage[margin=2cm]{geometry}
\usepackage{amssymb}
\usepackage{amsmath}
\usepackage{amsfonts}
\usepackage{amsthm}
\usepackage{mathtools}
\usepackage{listings}
\usepackage{latexsym}
\usepackage{booktabs}
\usepackage{multirow}
\usepackage{graphicx}
\usepackage{xcolor}
\usepackage{epsfig}
\usepackage{epstopdf}
\usepackage{url}
\usepackage{relsize}
\usepackage{hyperref}
\usepackage{textcomp}
\usepackage{indentfirst}
\usepackage{tipa}
\usepackage{algorithm}
\usepackage{algpseudocode}
\usepackage{subcaption}
\usepackage{dsfont}
\usepackage{float}
\usepackage[mathscr]{euscript}
\usepackage{enumitem}
\usepackage[title]{appendix}
\usepackage{bm}

\usepackage[backend=biber,style=numeric,citestyle=numeric,giveninits=true,sortcites=true,maxbibnames=4]{biblatex}
\renewbibmacro{in:}{}
\DeclareSymbolFont{yhlargesymbols}{OMX}{yhex}{m}{n} 
\DeclareMathAccent{\yhwidehat}{\mathord}{yhlargesymbols}{"62}

\usepackage{silence}
\usepackage{color}
\definecolor{deepblue}{rgb}{0,0,0.5}
\definecolor{deepred}{rgb}{0.6,0,0}
\definecolor{deepgreen}{rgb}{0,0.5,0}

\definecolor{DarkBlue}{rgb}{0.00,0.00,0.55}
\definecolor{Black}{rgb}{0.00,0.00,0.00}
\hypersetup{
	linkcolor = DarkBlue,
	anchorcolor = DarkBlue,
	citecolor = DarkBlue,
	filecolor = DarkBlue,
	colorlinks  = true,
	urlcolor = Black
}

\graphicspath{ {./Figures/},{./} }
\DeclareGraphicsExtensions{.eps,.pdf,.png,.jpg}

\newcommand{\X}{{\check{\bm{x}}}}
\newcommand{\x}{{\bm{x}}}
\newcommand{\dx}{\text{ d}\x}
\newcommand{\dX}{\text{ d}\X}
\newcommand{\K}{{\check{K}}}

\newtheorem{theorem}{Theorem}[section]
\newtheorem{lemma}[theorem]{Lemma}

\newtheorem{corollary}[theorem]{Corollary}
\theoremstyle{definition}
\newtheorem{definition}{Definition}[section]
\newtheorem{example}{Example}[section]
\theoremstyle{remark}
\newtheorem{remark}{Remark}[section]

\newtheorem{assumption}{Assumption}[section]

\usepackage{todonotes}
\definecolor{myblue}{RGB}{135, 206, 250}

\providecommand{\keywords}[1]{\textbf{\textit{Keywords:}} #1}

\title{Mixed-precision finite element kernels and assembly:\\Rounding error
analysis and hardware acceleration}
\author{Matteo Croci\footnote{BCAM, Basque Center for Applied Mathematics,
Bilbao, Spain \& Ikerbasque, Basque Foundation for Science, Bilbao, Spain
Email: \texttt{mcroci@bcamath.org}.}\hspace{3pt}  \and Garth N.
Wells\footnote{Department of Engineering, University of Cambridge, UK. Email:
\texttt{gnw20@cam.ac.uk}.}}

\date{}

\begin{document}

\maketitle

\begin{abstract}
    In this paper we develop the first fine-grained rounding error analysis of
    finite element (FE) cell kernels and assembly. The theory includes
    mixed-precision implementations and accounts for hardware-acceleration via
    matrix multiplication units, thus providing theoretical guidance for
    designing reduced- and mixed-precision FE algorithms on CPUs and GPUs.
    Guided by this analysis, we introduce hardware-accelerated mixed-precision
    implementation strategies which are provably robust to low-precision
    computations. Indeed, these algorithms are accurate to the lower-precision
    unit roundoff with an error constant that is independent from: the
    conditioning of FE basis function evaluations, the ill-posedness of the
    cell, the polynomial degree, and the number of quadrature nodes.
    Consequently, we present the first AMX-accelerated FE kernel implementations
    on Intel Sapphire Rapids CPUs.  Numerical experiments demonstrate that the
    proposed mixed- (single/half-) precision algorithms are up to 60 times faster
    than their double precision equivalent while being orders of magnitude more
    accurate than their fully half-precision counterparts.
\end{abstract}


\section{Introduction}
\label{sec:intro}
Hardware accelerators such as graphical processing units (GPUs) have been
ubiquitously employed to accelerate finite element (FE) computations, leading to
orders-of-magnitude speedups \cite{fu2014architecting, reguly2015finite,
dziekonski2013generation, trotter2023targeting, banas2014numerical}. These
improvements arise from the parallelization capabilities of GPUs and their
efficiency at performing computations in reduced- and mixed-precision
arithmetic.  However, there is no comprehensive rounding error analysis of
the FE method that can guide reduced-precision implementations and ensure
their accuracy. In this paper we develop the first fine-grained rounding
error analysis of FE cell kernels and assembly which includes
mixed-precision implementations and hardware accelerators.  This theory
establishes the accuracy of mixed-precision hardware-accelerated FE kernels
and assembly and can act as a guide for algorithmic design. Furthermore, we propose
a strategy for mixed-precision hardware accelerated FE kernel
implementations. We show that for high polynomial degrees these kernels are
not just faster than double precision kernels, but are also
orders of magnitude more accurate than their reduced-precision equivalent.
Therefore, this work is a stepping stone towards comprehensive
reduced-/mixed-precision FE implementations and their complete rounding
error analysis.

Artificial intelligence and machine learning applications have been driving the
development of computer chips that are particularly efficient at performing
reduced- and mixed-precision computations\footnote{Parallelization is also an important
aspect, but it is not the focus of this work.}. The list includes GPUs,
tensor/intelligence/neural processing units (TPUs/IPUs/NPUs), field-programmable
gate arrays (FPGAs), and others. Key components of these chips are
matrix multiplication units%
\footnote{The naming convention of accelerators and matrix
    units is not standardized and dictated by marketing choices. For instance,
    we have tensor cores (NVIDIA), matrix cores (AMD), matrix multiply units
(Google TPUs), advanced matrix extensions units (Intel), etc.}
which perform fused-matrix-multiply-add instructions in reduced- or mixed-precision
with a higher computational throughput than standard (e.g.,
vectorized) operations. For instance, NVIDIA H100 GPU half-precision tensor
cores are over 32 times faster than standard double precision computations%
\footnote{See \url{https://developer.nvidia.com/blog/nvidia-hopper-architecture-in-depth/}.}.
However, these chips are specifically tailored to graphics and machine learning
workloads and exploiting these accelerators in broader scientific computing
applications is an active field of investigation \cite{abdelfattah2021survey,
higham2022mixed}.

Hardware-accelerated reduced- and mixed-precision scientific computing poses two
challenges: 1) Computational bottlenecks must be offloaded, if possible, onto
matrix-matrix multiplications (which are relatively inexpensive).  2) Rounding
errors must be kept under control to ensure the accuracy of computations.
Algorithmic advances thus require both a careful implementation and a thorough
rounding error analysis.

The advantages of accelerating matrix-matrix multiplications via reduced- and
mixed-precision matrix multiplication units have been widely investigated in the
numerical linear algebra literature \cite{fasi2021numerical,
abdelfattah2021survey, higham2022mixed, lopez2023mixed, haidar2020mixed,
lewis2022large, fasi2023matrix, blanchard2020mixed}.  However, these
implementation strategies have not been fully incorporated within FE
computations. Indeed, multiple studies \cite{fu2014architecting,
reguly2015finite, pazner2023end, dziekonski2013generation, beams2020high,
andrej2024high, trotter2023targeting, banas2014numerical, cecka2011assembly,
remacle2016gpu, knepley2013finite} (the reference lists are in no way
exhaustive) have presented GPU-accelerated FE algorithms, but these mainly
exploit GPUs for parallelization and focus less on rounding error implications.
For instance, the cited FE works do not employ half-precision floating-point
arithmetic and do not mention tensor cores\footnote{Some references
\cite{reguly2015finite, knepley2013finite, trotter2023targeting,
beams2020high} mention or use cuBLAS which, among other functionalities,
implements tensor-core-accelerated matrix-matrix multiplications. It is not
stated in these papers whether these features have been exploited.}.

Similarly, the rounding error analysis of the FE method is limited.  Most of the
available studies \cite{maryvska2000schur, fried1971discretization,
alvarez2012round, babuvska2018roundoff} exclusively focus on rounding errors
arising from the linear solver and from the conditioning of FE linear systems.
The only exceptions are two studies (\cite{utku1984solution} and Section 8.5 in
\cite{croci2022stochastic}) that focus on
timestepping for parabolic problems and two more by Melosh and collaborators
which analyze rounding errors in one of the first instances of the FE method:
the FE analysis of structures \cite{melosh1969manipulation,
melosh1973inherited}.  There is no existing work on any other component of the
FE method, including FE kernels and assembly.  Therefore, it is
theoretically unclear whether FE computations in reduced- and mixed-precision can
attain enough accuracy for practical purposes.  On the other hand, there is a
wide breadth of work on reduced- and mixed-precision numerical linear algebra
algorithms (see \cite{abdelfattah2021survey, higham2022mixed} and the references
therein) as well as multiple studies on reduced-/mixed-precision timestepping
\cite{burnett2024stability, croci2023effects,grant2022perturbed,
burnett2024stability, balos2023leveraging, croci2022mixed, klower2022fluid} and
Newton's method \cite{tisseur2001newton, kelley2022newton}. These results,
together with classical rounding error theory \cite{higham2002accuracy,
wilkinson2023rounding}, can be used as building blocks to derive a comprehensive
rounding error analysis of the FE method. Such an analysis can then guide
reduced- and mixed-precision hardware-accelerated implementations and the design
of mixed-precision iterative solvers and related preconditioning strategies.

In this paper we take the first step in this direction by developing a
fine-grained rounding error analysis of reduced- and mixed-precision FE cell
kernels and assembly. In particular, we make the following new contributions:
\begin{itemize}
    \item We derive the first rounding error analysis of FE kernels and
        assembly. This theory includes basis function tabulation, geometry
        computations, FE function and coefficient evaluations, and the
        evaluation of quadrature rules. Furthermore, it accounts for
        mixed-precision implementations and mixed-precision hardware
        accelerators, thus providing theoretical guidance for designing reduced-
        and mixed-precision FE algorithms. For the sake of brevity, the analysis
        is presented for Lagrange FE and mass and Poisson forms (both bilinear
        forms and actions) on simplices and boxed-cells (e.g., hexahedra), but
        should qualitatively apply to more general settings.

    \item We present a new mixed-precision FE kernel implementation strategy
        which employs: a) higher precision for geometry computations and FE
        tabulation, b) lower precision for storage, and c) mixed-precision for
        FE function evaluations and for summing together quadrature node
        contributions. We prove that these kernels are accurate to the machine
        precision of the lower precision format with a small error constant that
        is independent to leading order from: the conditioning of FE basis
        function evaluations, the ill-posedness of the cell, the polynomial
        degree, and the number of quadrature nodes.  Conversely, we prove that
        the error constant is not independent from these quantities in fully
        reduced-precision implementations.

    \item We provide the first hardware-accelerated FE kernel implementations
        using mixed-precision fp32/bf16 Intel advanced matrix extension (AMX)
        units on Intel CPUs. For high-degree polynomials, AMX acceleration leads
        to a speedup factor of up to $60\times$ with respect to double-precision
        vectorized computations for mass forms and up to $15\times$ for Poisson
        forms. In contrast, we show that vectorized reduced- (fp16) and mixed-
        (fp32/bf16) precision kernels are never more than $4$-$6$ times
        faster than double precision.  Numerical results on a single CPU core
        support our rounding error analysis and thus provenly demonstrate that
        for high-degree polynomials mixed-precision hardware-accelerated kernels
        are faster than double precision kernels and orders of magnitude more
        accurate than their corresponding low-precision implementations. These
        results also show that it is possible to evaluate FE kernels robustly at
        half-precision accuracy.

    \item As a side theoretical contribution, we provide the first rounding
        error analysis of the evaluation of multivariate Lagrange polynomials
        and their gradients on simplicies (and boxed cells). The analysis shows
        that multivariate Lagrange polynomials and their derivatives can be
        evaluated to high absolute and relative (absolute only for derivatives)
        pointwise accuracy. We highlight that the rounding error analysis of the
        evaluation of 1D Lagrange polynomials and generic multivariate
        polynomials is well-known (cf.~\cite{higham2004numerical} and
        \cite{pena2000multivariate} respectively). 
\end{itemize}

We remark that while our rounding error analysis and implementation strategies
are not specific to a given set of floating-point formats, matrix multiplication
units often employ half precision which may limit the ultimate accuracy of the
FE method. However, this is not necessarily problematic: Often in computational
sciences and engineering a few digits of accuracy is all is needed.
Furthermore, half-precision-accurate FE kernels and assembly can still
accelerate computations in applications that warrant higher accuracy. Indeed,
there is a wide range of mixed-precision techniques which obtain high-precision
accuracy at low-precision costs by exploiting fast reduced-precision accurate
computations.  The list includes: mixed-precision iterative refinement methods
\cite{amestoy2024five, haidar2020mixed, carson2023mixed, oktay2022multistage},
mixed-precision Newton \cite{tisseur2001newton, kelley2022newton},
mixed-precision multigrid \cite{mccormick2021algebraic,
tamstorf2021discretization}, and mixed-precision iterative solvers and
preconditioning strategies \cite{carson2023mixed, georgiou2023mixed,
flegar2021adaptive, abdelfattah2021survey}. We leave the investigation of how
these techniques can be combined with mixed-precision FE kernels and assembly
to future work.

The paper is structured as follows: Section \ref{sec:background} contains a
brief introduction to rounding error analysis and hardware accelerators. In
Section \ref{sec:local_kernel_algorithm_and_assembly} we introduce the forms,
cell kernels, and assembly algorithms we consider in the paper. In Section
\ref{sec:rounding_error_analysis} we develop a comprehensive rounding error
analysis of these algorithms which accounts for mixed-precision implementations
and hardware accelerators. Numerical results supporting the theory are
presented in Section \ref{sec:numerical_results}. Finally, Section \ref{sec:conclusions}
contains our concluding remarks.

\section{Background}
\label{sec:background}

We begin this section by describing the notation adopted throughout the paper.
Then, we provide an introduction to rounding error analysis and to the theoretical
results from numerical linear algebra to be used as building blocks
for the theory developed in Section \ref{sec:rounding_error_analysis}. Finally,
we provide some background on mixed-precision hardware accelerators in
Section \ref{subsec:background_on_accelerators}.

\subsection{Notation}
In this paper, we adopt the following notation:
\begin{itemize}
    \item \textbf{Constants.} For simplicity of notation we indicate with $c$ a
        generic constant and we use numbers to denote different constants
        appearing in the same equation, e.g., $c_1$, $c_2$, etc. Constants with
        the same name appearing in different equations do not necessarily share
        the same value. We use $a \lesssim b$ to indicate that $a \leq cb$ for a
        constant $c>0$. We indicate dimension-dependent constants with $c(d)$.

    \item \textbf{Scalars, vectors, matrices, and tensor.} We indicate scalars
        with lower-case letters (e.g., $a$, $b$, etc.), column vectors with
        lower-case bold letters (e.g., $\bm{a}$, $\bm{b}$, etc.), and matrices
        and tensors with capitals. The only exception to this rule is for
        variables which are of higher tensor order in the Poisson kernels than
        in the mass kernels. In this case we use the higher tensor order
        notation for both for simplicity (e.g., if a variable is a matrix for
        Poisson and a vector for the mass form we use an upper case letter).  We
        employ the following partial indexing notation for tensors: if, e.g.,
        $A\in\mathbb{R}^{m\times n \times q \times r}$ is a fourth-order tensor,
        we indicate with $A_i\in \mathbb{R}^{n \times q \times r}$ and
        $A_{ij}\in \mathbb{R}^{q \times r}$ respectively the third- and
        second-order tensors obtained after fixing the first dimensions.

    \item \textbf{Linear algebra.} For a matrix $A$ or a column vector $\bm{b}$
        we indicate with $A^T$ and $\bm{b}^T$ their transposes and with $\lVert
        A \rVert_p$ and $\lVert \bm{b} \rVert_p$ their $p$ norm respectively.
        For a set $\{A^k\}_{k=1}^n$ of square matrices we indicate with
        $\textnormal{diag}_{k=1}^n(A^k)$ the block diagonal matrix whose $k$-th
        block is given by $A^k$. For a set $\{\bm{v}^k\}_{k=1}^n$ of column
        vectors we denote with $\textnormal{vstack}_{k=1}^n(\bm{v}^k)$ the
        column vector obtained by vertically stacking the column vectors. For a
        set of scalars $\{a_k\}_{k=1}^n$ we indicate with $[a_k]_{k=1}^n$ the
        column vector $\bm{a}$ such that its $k$-th entry is given by $a_k$.

    \item \textbf{Rounding errors.} We denote with $u$ a unit roundoff and we
        possibly use subscripts to distinguish between the units roundoff of
        different precisions and floating-point formats (cf.~Table \ref{tab:precision}).
	We indicate with $\delta$ single roundoffs which
        satisfy $|\delta|\leq u$. We denote with $\theta_n$
        generic $O(nu)$ rounding error terms which we define in Lemma
        \ref{lemma:higham_theta_gamma}. Different $\theta_n$ terms need not have
        the same value, but all $\theta_n$ terms are bounded by $|\theta_n|\leq
        \gamma_n\leq cu$ (also defined in Lemma \ref{lemma:higham_theta_gamma}).
        We will use superscripts to denote $\delta$, $\theta_n$, $\gamma_n$
        terms which arise from calculations performed in different precisions.
        Finally, for any quantity $a$ we indicate with $\hat{a}$ the result of
        its finite-precision computation and with $\Delta a$ we denote the
        corresponding rounding error: $\Delta a := \hat{a}-a$.

    \item \textbf{Condition numbers.} We denote condition numbers of numerical
        computations with the Greek letter $\kappa$. Additionally, for a square
        matrix $A$, we indicate with $\kappa_p(A)$ the $p$-norm condition
        number of $A$.

    \item \textbf{Domain and mesh}. Let $D\subset\mathbb{R}^d$ be an open
        polyhedral domain in $\mathbb{R}^d$. For any such $D$, we denote with
        $D_h$ a non-degenerate tessellation of $D$ made of either simplices or
        boxed cells (e.g., quadrilaterals or hexahedra). We denote with $K\in
        D_h$ a cell of $D_h$ and with $\K$ the reference cell.

\end{itemize}

\subsection{Background on rounding error analysis}
\label{subsec:background_on_rounding_error_analysis}

\begin{table}
	\centering
	\begin{tabular}{@{}llllcc@{}}
		\toprule
		\multicolumn{1}{c}{Format} & \multicolumn{1}{c}{$u$} & \multicolumn{1}{c}{$x_{\min}$} & \multicolumn{1}{c}{$x_{\max}$} & $t$ & exponent bits \\ \midrule
		bf16/bfloat16 (half)                   & $3.91\times 10^{-3}$    & $1.18\times 10^{-38}$          & $3.39\times 10^{38}$           & $8$         & $8$      \\
		fp16 (half)                       & $4.88\times 10^{-4}$    & $6.10\times 10^{-5}$           & $6.55\times 10^{4}$            & $11$        & $5$      \\
		fp32 (single)              & $5.96\times 10^{-8}$    & $1.18\times 10^{-38}$          & $3.40\times 10^{38}$           & $24$        & $8$      \\
		fp64 (double)              & $1.11 \times 10^{-16}$  & $2.22\times 10^{-308}$         & $1.80\times 10^{308}$          & $53$        & $11$     \\ \bottomrule
	\end{tabular}
	\caption{\textit{Overview of commonly used floating point systems. Here $u=2^{-t}$ is the roundoff unit, $x_{\min}$ in the smallest normalized positive number, $x_{\max}$ is the largest finite number and $t$ is the precision. While bf16 and fp32 have the same range and exponent bits, fp16 has a smaller roundoff unit (higher precision) at the cost of a smaller range.}}
	\label{tab:precision}
	\vspace{-6pt}
\end{table}

There are multiple floating-point number formats available. In
this paper we only use four, which are listed in Table \ref{tab:precision}
together with their properties. We start by giving some basic definitions:

\begin{definition}
    The terms \emph{reduced precision} and \emph{low precision} denote the use of
    floating-point number formats with larger unit roundoffs (e.g., fp16 or
    bf16), in contrast to \emph{high precision} formats which are more
    accurate (e.g., fp32 or fp64).
\end{definition}

\begin{definition}
    \emph{Reduced- and high-precision algorithms} are algorithms which are run
    entirely in the same precision (low or high respectively) as opposed to
    \emph{mixed-precision algorithms} that employ different formats. Typically a
    mixed-precision algorithm employs both high- and low-precision
    computations.
\end{definition}

Next, we describe the standard floating point error model which applies to
round-to-nearest computations (cf.~Chapter 2 of \cite{higham2002accuracy}):
\begin{align}
\label{eq:std_floating_point_error_model}
\widehat{(a \text{ op } b)} = (a \text{ op } b)(1 + \delta),\quad |\delta| <
u,\quad\text{op}\in\{+,-,\times,\backslash\}.
\end{align}
Here $u$ is the roundoff unit and $\delta$ is called a roundoff error. In the
above and throughout the paper we use hats to denote quantities that are the
result of finite precision computations. By using the model in
\eqref{eq:std_floating_point_error_model} it is possible to derive \emph{a
priori} rounding error bounds for a variety of different algorithms and
operations \cite{higham2002accuracy}.

Products and ratios of multiple $(1+\delta)$ terms often arise in rounding error
analysis. It is thus advantageous to simplify the notation by
invoking the following result:
\begin{lemma}[Lemma 3.1 in \cite{higham2002accuracy}]
    \label{lemma:higham_theta_gamma}
        If $|\delta_i|\leq u$ and $\rho_i=\pm 1$ for $i=1,\dots,n$ and $nu<1$,
        then
        \begin{align}
            \prod_{i=1}^n(1+\delta_i)^{\rho_i}=(1+\theta_n),\quad\text{where}\quad
            |\theta_n|\leq \frac{nu}{1-nu}:=\gamma_n.
        \end{align}
        If $u$ is sufficiently small, then there exists $c \geq 1$ such that
        $\gamma_n\leq cu$.
\end{lemma}
We will use the $\theta_n$ and $\gamma_n$ notations throughout the paper. We
remark that $\theta_n$ indicates a generic error term and two $\theta_n$ terms
need not be the same value unless they arise from exactly the same computation.
Nevertheless, all error terms indicated with $\theta_n$ are always bounded in
absolute value by $\gamma_n$. Manipulation of multiple $\theta_n$ and $\gamma_n$
terms is made easier thanks to the following lemma:
\begin{lemma}[Lemma 3.3 in \cite{higham2002accuracy}]
    \label{lemma:higham_manipulation_theta_gamma}
    For any positive integer $k$ let $\theta_k$ denote a quantity bounded
    according to $|\theta_k|\leq \gamma_k=ku/(1-ku)$. The following relations
    hold:
    \begin{align*}
        (1+\theta_k)(1+\theta_j) = 1 + \theta_{k+j},\quad
        i\gamma_k \leq \gamma_{ik},\quad
        \gamma_k + u \leq \gamma_{k+1},\quad
        \gamma_k + \gamma_j + \gamma_k\gamma_j \leq \gamma_{k+j}.
    \end{align*}
\end{lemma}
\noindent In this paper we will make use of the above relations without
explicitly referring to this lemma. We will also make use of a series of results
of classical rounding error analysis related to inner products, matrix-vector
and matrix-matrix products, matrix inversion, and determinant computations.

\paragraph{a) Inner products.}
The following backward error result holds for the inner products between two
vectors (see equation (3.4) in \cite{higham2002accuracy}): Let $\bm{a},
\bm{b}\in\mathbb{R}^n$, then,
\begin{align}
    \label{eq:inner_prods}
    \widehat{\bm{a}^T\bm{b}}=(\bm{a}+\Delta \bm{a})^T\bm{b},\quad\text{where}\quad |\Delta \bm{a}|\leq \gamma_n|\bm{a}|\leq cnu|\bm{a}|,
\end{align}
where the bound is to be interpreted entrywise. The above bound implies that an
inner product computed in finite-precision arithmetic is equivalent to an exact
inner product between perturbed vectors.

\paragraph{b) Matrix-vector products.}
The following backward error result holds for matrix-vector products (see equation (3.11) in \cite{higham2002accuracy}): Let $A\in\mathbb{R}^{m\times n}$, $\bm{b}\in\mathbb{R}^n$, then,
\begin{align}
    \label{eq:matvec_higham}
    \widehat{A\bm{b}}=(A+\Delta A)\bm{b},\quad\text{where}\quad |\Delta A|\leq \gamma_n|A|\leq cnu|A|,
\end{align}
where the bound is to be interpreted entrywise. The above bound implies that a
matrix-vector product computed in finite-precision arithmetic is equivalent to an exact
matrix-vector product with a perturbed matrix.

\paragraph{c) Matrix-matrix products.}
The following forward error result holds for matrix-vector products (see equation (3.13) in \cite{higham2002accuracy}): Let $A\in\mathbb{R}^{m\times n}$, $B\in\mathbb{R}^{n\times l}$, $C=AB$, then,
\begin{align}
    \label{eq:matmat_higham}
    \hat{C} = C + \Delta C,\quad\text{where}\quad |\Delta C|\leq \gamma_n|A||B|\leq cnu|A||B|.
\end{align}
Again, the bound is to be interpreted entrywise.

\paragraph{d) Matrix inversion.}
Another result which we will be using is a forward error bound for computing the
inverse of a matrix $A\in\mathbb{R}^{n\times n}$. Since there are many
algorithms for matrix inversion available (cf.~\cite{higham2002accuracy}) we
assume that the one we use satisfies the following assumption (see equation
(14.3) in \cite{higham2002accuracy}):

\begin{assumption}
    \label{ass:matrix_inverse_algorithm}
    The matrix inversion algorithm yields an approximation $\widehat{A^{-1}}\approx A^{-1}$
    which satisfies
    \begin{align}
        \label{eq:matrix_inv}
        |\widehat{A^{-1}}-A^{-1}| \leq c(n)u|A^{-1}||A||A^{-1}|,
    \end{align}
    for a positive constant $c(n)$ which grows with $n$.
\end{assumption}

\paragraph{e) Determinant computation.}
The following theorems describe how rounding errors affect the computation of
determinants computed through an LU factorization (e.g., this is what LAPACK
does). 
\begin{theorem}[Theorem 9.3, Section 9, and Section 14.6 in \cite{higham2002accuracy}]
    \label{th:LU_det}
    The LU factorization of a matrix $A\in\mathbb{R}^{n\times n}$ yields LU factors $
    \hat{L}$ and $\hat{U}$ satisfying
    \begin{align}
        \hat{L}\hat{U} = A + \Delta A,\quad\text{where}\quad |\Delta A|\leq c(n)u|\hat{L}||\hat{U}|,
    \end{align}
    where $c(n)$ is a constant growing with $n$. Assuming that no
    underflow/overflow occurs\footnote{For determinants in particular, this is
    a strong assumption, although there are remedies: see Section 14.6 in \cite{higham2002accuracy}.}, computing the
    determinant as $\text{det}(\hat{U})$ yields
    \begin{align}
        \label{eq:_LU_det_theta_n}
        \widehat{\textnormal{det}(A)} = \textnormal{det}(\hat{U})(1 + \theta_n),\quad\text{where}\quad |\theta_n|\leq cnu.
    \end{align}
\end{theorem}
Since $\text{det}(\hat{U})=\text{det}(A+\Delta A)$, this means that in finite
precision we compute the (almost) exact determinant of a slightly perturbed
matrix. To obtain a final bound for the error, we invoke the following
perturbation result for determinants:
\begin{theorem}[Corollary 2.14 in \cite{ipsen2008perturbation}]
    \label{th:ipsen_coroll}
    If $A\in\mathbb{R}^{n\times n}$ is non-singular, then it holds that
    \begin{align}
        |\textnormal{det}(A+E)-\textnormal{det}(A)|\leq \left(\left(\lVert
        A^{-1}\rVert_2\lVert E\rVert_2+1\right)^n-1\right)|\textnormal{det}(A)| \leq
        cn\lVert A^{-1}\rVert_2\lVert E\rVert_2|\textnormal{det}(A)|.
    \end{align}
\end{theorem}

\begin{corollary}
    \label{coroll:detbound_final}
    Let $A\in\mathbb{R}^{n\times n}$ be non-singular and assume that $u$ is
    sufficiently small. If $\det(A)$ is computed via the LU factorization of $A$
    and no underflows/overflows occur, then it holds
    \begin{align}
        \label{eq:detbound_final}
        |\widehat{\textnormal{det}(A)}-\textnormal{det}(A)|\leq
        c(n)u\kappa_2(A)|\textnormal{det}(A)|,
    \end{align}
    where $\kappa_2(A)$ is the $2$-norm condition number of $A$ and $c(n)$
    is a positive constant growing with $n$.
\end{corollary}

\begin{proof}
    The proof is obtained by setting $E=\Delta A$ in Theorem
    \ref{th:ipsen_coroll} and applying Theorem \ref{th:LU_det} to bound the
    perturbation together with the bound $\lVert |\hat{L}||\hat{U}|\rVert_2\leq
    c(n)\lVert A \rVert_2$ which is a consequence of
    \cite[Lemma~9.6]{higham2002accuracy}.
\end{proof}

We will use all the above results as building blocks for deriving rounding
error bounds for the whole assembly process.

\begin{remark}
    All the above bounds (and thus the bounds derived in the next pages) are
    worst-case rounding error bounds. We remark that rounding errors may be more
    lenient in practice: they often behave as independent random variables
    leading to smaller error growth rates with respect to $n$ due to stochastic
    error cancellations \cite{higham2019new}. Under this scenario all the bounds
    listed here would still hold after replacing $\gamma_n$ with the
    smaller $\tilde{\gamma}_n=O(n^{1/2})$. However, it
    is difficult to predict \emph{a priori} whether this will happen for a given
    calculation.
\end{remark}

\subsection{Background on mixed-precision hardware accelerators}
\label{subsec:background_on_accelerators}

Mixed-precision hardware accelerators are specialized units which perform
multiply-add operations with superior performance. They are typically present in
GPUs as well as in chips tailored to artificial intelligence and machine
learning workloads.  Intel Xeon 4th-generation scalable processors also include
mixed-precision hardware accelerators. 

In this paper we aim to exploit these hardware capabilities to either accelerate
FE kernels (and therefore assembly) and/or to increase their accuracy when
reduced-precision formats are employed.  There are two types of accelerators
which are relevant to our work: 1) Mixed-precision vector units, performing
entrywise multiply-add operations between vectors. 2) Mixed-precision matrix
units, performing fused-matrix-multiply-add operations. In both cases the input
arrays are stored in a lower-precision format (e.g., fp16 or bf16) and the
multiply-add operation is performed in a higher precision format (e.g., fp32),
which is also used for storing the result.  The maximum array size is restricted
by the size of the accelerator registers and the multiplication between larger
arrays must necessarily be blocked.

Although our rounding error analysis is independent from the accelerator or chip
used, we only consider CPU implementations in our numerical experiments
(cf.~Section \ref{sec:numerical_results}). Therefore, we only describe the
properties of Intel CPU accelerators in detail here.  Intel vector accelerators
are called ``advanced vector extensions'' (AVX) units and Intel matrix
accelerators are called ``advanced matrix extensions'' (AMX) units.

\paragraph{AVX units.}
AVX is the name used to denote vector accelerators in general, i.e., it also
includes units which operate in the same precision and operations which are not
necessarily entrywise vector-multiply-adds (for a full list, see the Intel
Intrinsics Guide \cite{intel-intrinsics}). AVX units perform computations on
vector registers of different sizes (up to 512 bits on AVX512 units). In particular, in
this work we will use AVX512 operations when working in double, single, or fp16
half precision. For mixed-precision computations we will employ the AVX512-bf16
unit which performs a product between two
bf16 vectors of $32$ entries (i.e., 512 bits), adds the result to a single
precision vector of $16$ entries (again 512 bits), and accumulates it into
another single precision vector of the same size. Namely, given the vectors
\begin{align}
    \bm{a} = [a_1,\dots, a_{32}]^T,\quad 
    \bm{b} = [b_1,\dots, b_{32}]^T,\quad 
    \bm{c} = [c_1,\dots, c_{16}]^T,
\end{align}
with $\bm{a},\bm{b}$ stored contiguously in bf16 and $\bm{c}$ stored
contiguously in single precision, the AVX512-bf16 unit computes
\begin{align}
    \bm{c} \mathrel{+}= \tilde{\bm{a}}^1\circ\tilde{\bm{b}}^1 +
    \tilde{\bm{a}}^2\circ \tilde{\bm{b}}^2,
\end{align}
where $\circ$ denotes the entrywise vector product and all operations are
performed in single precision \cite{intel-intrinsics}. The vectors
$\tilde{\bm{a}}$ and $\tilde{\bm{b}}$, of length $16$, are given by
\begin{align}
    \tilde{\bm{a}}^1=[a_1,a_3,\dots, a_{31}]^T,\quad
    \tilde{\bm{a}}^2=[a_2,a_4,\dots, a_{32}]^T,\quad
    \tilde{\bm{b}}^1=[b_1,b_3,\dots, b_{31}]^T,\quad
    \tilde{\bm{b}}^2=[b_2,b_4,\dots, b_{32}]^T.
\end{align}
AVX512-bf16 fused-multiply-add instructions perform $64$ floating-point
operations per cycle which is $4$ times more than AVX512 fused-multiply-add
double precision units and the same as AVX512 fused-multiply-add fp16
half-precision instructions. Therefore, implementations exploiting these
mixed-precision accelerators can be up to $4$ times faster than double precision
and as fast as half-precision computations (albeit being more accurate than pure
half-precision operations).

\paragraph{AMX units.}
AMX is the name used by Intel to denote matrix
accelerators. The Intel Intrinsics Guide \cite{intel-intrinsics}
includes only load and store operations into matrix registers as well as fused-matrix-multiply-add operations. Here we are only concerned with the mixed-precision
fp32/bf16 instructions in the AMX-bf16 extension set \cite{intel-intrinsics}.
These instructions work with matrices of variable sizes. Here we use the largest matrix
size which fits the AMX registers: $16$ rows of $64$ bits each,
corresponding to $16$-by-$16$ for single precision matrices and $16$-by-$32$
for bf16 matrices (better seen as two packed $16$-by-$16$ bf16 matrices,
see next). The AMX-bf16 mixed-precision matrix multiplication units work as
follows: given the contiguously-stored matrices
\begin{align*}
    A = \left[
        \begin{array}{cccc}
            a_{1,1} & a_{1,2} & \dots & a_{1,32}\\
            a_{2,1} & a_{2,2} & \dots & a_{2,32}\\
            \dots & \dots & \dots & \dots\\
            a_{16,1} & a_{16,2} & \dots & a_{16,32}
        \end{array}
    \right],\quad 
    B = \left[
        \begin{array}{cccc}
            b_{1,1} & b_{1,2} & \dots & b_{1,32}\\
            b_{2,1} & b_{2,2} & \dots & b_{2,32}\\
            \dots & \dots & \dots & \dots\\
            b_{16,1} & b_{16,2} & \dots & b_{16,32}
        \end{array}
    \right],\quad 
    C = \left[
        \begin{array}{cccc}
            c_{1,1} & c_{1,2} & \dots & c_{1,16}\\
            c_{2,1} & c_{2,2} & \dots & c_{2,16}\\
            \dots & \dots & \dots & \dots\\
            c_{16,1} & c_{16,2} & \dots & c_{16,16}
        \end{array}
    \right],
\end{align*}
where $A,B\in\mathbb{R}^{16\times32}$ are stored in bf16 and
$C\in\mathbb{R}^{16\times16}$ in single precision, the AMX units compute
\begin{align}
    \label{eq:AMX_matrix_multiply_add_def}
    C \mathrel{+}= \tilde{A}_1\tilde{B}_1 + \tilde{A}_2\tilde{B}_2,
\end{align}
where all operations are performed in single precision by flushing output
subnormals to zero \cite{intel-dev-manual} and $\tilde{A}_1$, $\tilde{A}_2$
are given by
\begin{align}
    \tilde{A}_1 = \left[
        \begin{array}{cccc}
            a_{1,1} & a_{1,3} & \dots & a_{1,31}\\
            a_{2,1} & a_{2,3} & \dots & a_{2,31}\\
            \dots & \dots & \dots & \dots\\
            a_{16,1} & a_{16,3} & \dots & a_{16,31}
        \end{array}
    \right],\quad 
    \tilde{A}_2 = \left[
        \begin{array}{cccc}
            a_{1,2} & a_{1,4} & \dots & a_{1,32}\\
            a_{2,2} & a_{2,4} & \dots & a_{2,32}\\
            \dots & \dots & \dots & \dots\\
            a_{16,2} & a_{16,4} & \dots & a_{16,32}
        \end{array}
    \right],
\end{align}
i.e., $\tilde{A}_1$ and $\tilde{A}_2$ are the submatrices of $A$ respectively
containing its odd and even columns. $\tilde{B}_1$ and $\tilde{B}_2$ are defined
in the same way.

We remark that when the matrices that need to be multiplied
are larger than what fits in the AMX registers the matrix multiplication must be
decomposed into smaller matrix-multiply-add operations like
\eqref{eq:AMX_matrix_multiply_add_def}. It is also worth mentioning that this
decomposition often involves a non-trivial memory reorganization of the inputs.

\begin{example}
    We clarify this issue with an example: let us assume for simplicity that the input
    bf16 sizes were $3$-by-$6$ rather than $16$-by-$32$ and that we wanted to
    compute a matrix product $AB$ between two matrices $A$ and $B$ of size
    $3$-by-$12$ and $12$-by-$3$ respectively. Let $A$ be stored contiguously as
    \begin{align}
        A = \left[
            \begin{array}{ccc|ccc|ccc|ccc}
                a_{1,1} & a_{1,2} & a_{1,3} & a_{1,4} & a_{1,5} & a_{1,6} & a_{1,7} & a_{1,8} & a_{1,9} & a_{1,10} & a_{1,11} & a_{1,12}\\
                a_{2,1} & a_{2,2} & a_{2,3} & a_{2,4} & a_{2,5} & a_{2,6} & a_{2,7} & a_{2,8} & a_{2,9} & a_{2,10} & a_{2,11} & a_{2,12}\\
                a_{3,1} & a_{3,2} & a_{3,3} & a_{3,4} & a_{3,5} & a_{3,6} & a_{3,7} & a_{3,8} & a_{3,9} & a_{3,10} & a_{3,11} & a_{3,12}
            \end{array}
        \right].
    \end{align}
    In order to compute the product, $A$ would need to first be
    reorganized\footnote{The reshaping is only for the sake of explanation: what
    is important here is that each row is stored contiguously one after the
    other.} contiguously in memory as
    \begin{align}
        A = 
        \left[
            \begin{array}{c|c}
                \tilde{A}_{1,1} & \tilde{A}_{1,2} \\
                \hline \\[-1em]
                \tilde{A}_{2,1} & \tilde{A}_{2,2}
            \end{array}
        \right] = 
        \left[
            \begin{array}{ccc|ccc}
                a_{1,1} & a_{1,4}  & a_{1,2} & a_{1,5}  & a_{1,3} & a_{1,6} \\
                a_{2,1} & a_{2,4}  & a_{2,2} & a_{2,5}  & a_{2,3} & a_{2,6} \\
                a_{3,1} & a_{3,4}  & a_{3,2} & a_{3,5}  & a_{3,3} & a_{3,6} \\
                \hline
                a_{1,7} & a_{1,10} & a_{1,8} & a_{1,11} & a_{1,9} & a_{1,12} \\
                a_{2,7} & a_{2,10} & a_{2,8} & a_{2,11} & a_{2,9} & a_{2,12} \\
                a_{3,7} & a_{3,10} & a_{3,8} & a_{3,11} & a_{3,9} & a_{3,12} \\
            \end{array}
        \right].
    \end{align}
    Only after rearranging $A$ and performing a similar, albeit different,
    rearranging of $B$ (which we omit for brevity) one can finally compute $AB$
    as follows: First, let $C\in\mathbb{R}^{3\times 3}$ be the zero matrix stored in
    single precision. Second, use AMX units twice (loading the entries of $A$ and $B$ in two
    batches of $18$ numbers each) to compute
    \begin{align}
        C \mathrel{+}= \tilde{A}_{1,1}\tilde{B}_{1,1} + \tilde{A}_{1,2}\tilde{B}_{1,2},\\
        C \mathrel{+}= \tilde{A}_{2,1}\tilde{B}_{2,1} + \tilde{A}_{2,2}\tilde{B}_{2,2}.
    \end{align}
    The resulting $C$ is then equal to $AB$.
\end{example}

The appeal of using AMX cores lies in their efficiency: They perform $1024$
floating-point operations per cycle \cite{intel-dev-manual} which are $64$ times
more than AVX512 fused-multiply-add double precision instructions.  Therefore,
codes exploiting AMX-bf16 instructions can in principle be up to $64$ times
faster than their double precision equivalents. Of course this is just a
theoretical upper bound: this speedup is only obtainable on
fused-matrix-multiply-add operations of the form
\eqref{eq:AMX_matrix_multiply_add_def} and not on any other computation.
Nevertheless, these considerations provide a recipe for making full use of AMX
accelerators: large speedups are obtainable if algorithms are implemented so
that their computational bottleneck is comprised of matrix-matrix products or
fused-matrix-multiply-add operations. This is exactly what motivates the FE
kernel implementations presented in the next section.

\section{Local kernel algorithm and global assembly}
\label{sec:local_kernel_algorithm_and_assembly}

In this section we introduce the FE kernels we consider in the paper. First, we
describe the FE forms and then we present the kernel algorithms. For the latter
we select an implementation in which matrix-matrix products are the bottleneck
in order to exploit AMX accelerators. Finally, we briefly mention the assembly
process which is the direct application of our work.

Since rounding error analyses are specific to the actual computations performed,
covering the whole set of possible FE forms and elements would unnecessarily
complicate and lengthen the exposition. Therefore, we adopt the following
simplifications: 
\begin{enumerate}
    \item We only consider the mass and Poisson linear (actions)
        and bilinear (matrix) forms (with coefficients).
    \item We only work with Lagrange elements on simplices and boxed cells (e.g.,
        quadrilaterals and hexahedra).
    \item We assume that all FE basis functions appearing in the forms
        belong to the same FE approximation subspace.
\end{enumerate}
We expect the generalization of our results to other forms and elements
to yield qualitatively similar results.

\subsection{Cell kernels}
\label{subsec:cell_kernels}

Let $D$ be an open domain with compact closure in $\mathbb{R}^d$ and let $D_h$
be a tessellation of $D$ comprised of simplices or boxed cells. The mass and Poisson
forms over a single cell $K\in D_h$ read:
\begin{align}
    a_K(w,v)=\int_{K}z(\x)w v \dx,\quad\text{(mass form)},\quad\quad
    a_K(w,v)=\int_{K}z(\x)\nabla w \cdot \nabla v \dx,\quad\text{(Poisson form)},
\end{align}
where $w$, $v$, and $z$ all belong to the same FE approximation space.
We directly consider the evaluation of the forms over the reference cell and take
$\Phi=[\phi_{k}]_{k=1}^{n_{\phi}}$ to be the reference basis functions.
We consider the computation of
\begin{align}
    A_{ji} = a_{\K}(\phi_i, \phi_j),\quad v_j = a_{\K}(w, \phi_j),
\end{align}
where $w$ is now a known function and $a_{\K}$
is the bilinear form over $K$ mapped onto the reference cell $\K$. Here $A\in
\mathbb{R}^{n_{\phi}\times n_{\phi}}$ and $\bm{v}\in\mathbb{R}^{n_{\phi}}$ are the (local) matrix and
vector that encode the bilinear and linear form over $K$ respectively.

After mapping the mass form to the reference cell we obtain
\begin{align}
    A = \int_{\K}\check{z}(\X) |\textnormal{det}(J(\X))|
    \left(\Phi(\X) \otimes \Phi(\X)\right) \dX,\quad\quad\quad
    \bm{v} = \int_{\K}\left(\check{z}(\X) |\textnormal{det}(J(\X))|
    \check{w}(\X)\right) \Phi(\X) \dX.
\end{align}
Here $\check{z}(\X)$ and $\check{w}(\X)$ are the result of mapping the functions
$z$ and $w$ onto the reference cell and $|\textnormal{det}(J(\X))|$ is the
absolute value of the determinant of the Jacobian $J(\X)$ of the reference map.
For the Poisson forms we instead obtain
\begin{align}
    A = \sum_{s=1}^d\sum_{t=1}^d \bar{A}_{st},\quad\quad \bm{v} =
    \sum_{s=1}^d\bar{\bm{v}}_s,
\end{align}
where $\bar{A}\in\mathbb{R}^{d\times d\times n_{\phi} \times n_{\phi}}$,
$\bar{\bm{v}}_s\in\mathbb{R}^{n_{\phi}}$ are defined as follows: for $s=1,\dots, d$,
$t=1,\dots, d$,
\begin{align}
    \bar{A}_{st} = \int_{\K}\check{C}_{st}(\X) \left(\partial_{\X_t}\Phi(\X) \otimes
    \partial_{\X_s}\Phi(\X)\right) \dX,\quad\quad
    \bar{\bm{v}}_s = \int_{\K} \check{r}_s(\X) \partial_{\X_s}\Phi(\X) \dX.
\end{align}
Here, $\check{C}(\X)\in\mathbb{R}^{d\times d}$ and $\check{\bm{r}}(\X)\in\mathbb{R}^d$ are given by
\begin{align}
    C(\X) = \check{z}(\X)G(\X),\quad\quad
    \check{\bm{r}}(\X) = \check{z}(\X)(G(\X)\nabla \check{w}(\X)),\\
    \text{where}\quad G(\X) =
    |\textnormal{det}(J(\X))|(J^{-1}(\X))(J^{-1}(\X))^T\in\mathbb{R}^{d\times
    d}
\end{align}
is the geometry tensor. We highlight the fact that the functions $\check{C}(\X)$
and $\check{\bm{r}}(\X)$ incorporate both geometry and form coefficient information.

Note that the mass form can also be written in a similar
way by setting $G(\X)=|\textnormal{det}(J(\X))|$,
$\check{\bm{r}}(\X)=\check{z}(\X)G(\X) \check{w}(\X)$, $\bar{\bm{v}}_1=\bm{v}$,
and $\bar{A}_{11}=A$. Therefore, for both forms we can write the local cell matrix
and vector as
\begin{align}
    \label{eq:form_decomposed}
    A = \sum_{s=1}^{n_d}\sum_{t=1}^{n_d} \bar{A}_{st},\quad\quad \bm{v} =
    \sum_{s=1}^{n_d}\bar{\bm{v}}_s,
\end{align}
where $n_d=d$ for Poisson and $n_d=1$ for the mass form and
$\bar{A}\in\mathbb{R}^{n_d\times n_d\times n_{\phi} \times n_{\phi}}$.

\begin{remark}
    Most forms can be expressed as sums similar to
    \eqref{eq:form_decomposed} by considering each term in the sum
    to be the discretization of a monomial term comprised of integrals of
    derivatives of basis function components, cf.~\cite{kirby2007efficient,rognes2010efficient}.
\end{remark}

After exactly integrating the integrals in \eqref{eq:form_decomposed} exactly\footnote{We
use exact quadrature for simplicity. Using inexact rules does not change our
argument.} with suitable quadrature rules $\{(\omega_q,\X_q)\}_{q=1}^{n_q}$ we can
write each term in the sums as
\begin{align}
    \bar{A}_{st} = B_sC_{st}B_t^T,\quad \bar{\bm{v}}_s = B_s\bm{r}_s,
\end{align}
where $C$ is a fourth-order sparse tensor such that 
\begin{align}
    C_{st} = \text{diag}_{q}(\omega_q\check{C}_{st}(\X_q))\in\mathbb{R}^{n_q\times n_q},
\end{align}
$\bm{r}_s\in\mathbb{R}^{n_q}$ is given by
\begin{align}
    \bm{r}_s = \textnormal{vstack}_q(\omega_q\check{\bm{r}}_s(\X_q)),
\end{align}
and $B\in\mathbb{R}^{n_d\times n_{\phi} \times n_q}$ is a third-order tensor with entries given by
\begin{align}
    B_{skq}=\partial_{\X_s}\phi_k(\X_q),\quad\quad B_{skq}=\phi_k(\X_q),
\end{align}
for the Poisson and the mass form respectively. Therefore, for both forms we can
express $A$ as the sum of $n_d^2$ triple matrix products:
\begin{align}
	\label{eq:triple_products}
    A = \sum_{s=1}^{n_d}\sum_{t=1}^{n_d} B_{s}C_{st}B_t^T,
\end{align}
and $\bm{v}$ as the sum of $n_d$ matrix-vector products:
\begin{align}
	\label{eq:matvec_products}
    \bm{v} = \sum_{s=1}^{n_d} B_{s}\bm{r}_s.
\end{align}

The fact that $A$ is expressed as a sum of matrix-matrix products leads to a
kernel implementation that is amenable to AMX acceleration (as well as AVX).
However, the same does not apply to $\bm{v}$, which only includes matrix-vector
products and, as such, can only be accelerated with vector units. As a solution,
we batch kernel computations by groups of $n_{\text{batch}}$ cells
($n_{\text{batch}}$ must be divisible by 16 to get the most out of AMX
computations, cf.~\ref{subsec:background_on_accelerators}), thus obtaining
\begin{align}
     V^{\text{batch}}= \sum_{s=1}^{n_d} B_sR_s^{\text{batch}},\quad\text{where}\quad
     V^{\text{batch}} = \left[\bm{v}^{1} | \dots |
     \bm{v}^{n_{\text{batch}}}\right],\quad R_s^{\text{batch}} =
     \left[\bm{r}^{1}_s | \dots | \bm{r}^{n_{\text{batch}}}_s\right],
\end{align}
where $V_s^{\text{batch}}$ and $R_s^{\text{batch}}$ are matrices of size
$n$-by-$n_\text{batch}$ and $n_q$-by-$n_\text{batch}$ respectively (recall that
$B_s\in\mathbb{R}^{n\times n_q}$ only depends on the reference cell).

\begin{remark}
    \label{rem:cell_batching}
    In what follows, we present our kernel algorithms and rounding error analysis
    without cell-batching for simplicity, but we do use cell-batching in our
    numerical experiments (cf.~Section \ref{sec:numerical_results}). Our theory
    applies to cell-batching as-is since the computation on each column of
    $V^{\text{batch}}$ is performed independently.
\end{remark}

We are now ready to describe the cell kernel algorithms used in this paper.

\subsection{Local kernel algorithms}

In this paper we consider the following local kernel algorithms. Both are
amenable to AMX and AVX acceleration.

\vspace{-3pt}
\paragraph{Algorithm 1. Local kernel algorithm for bilinear forms.}
\begin{enumerate}
    \item Tabulate the basis function values
        forming $B$.
    \item Compute $C$.
    \item Compute $H_{st}=C_{st}B_t^T$ for all $s,t\in\{1,\dots, n_d\}$.
    \item Compute $A=\sum_{s=1}^{n_d}\sum_{t=1}^{n_d}B_sH_{st}$.
\end{enumerate}

\begin{remark}
    Another option is to compute $\tilde{H}_s=\sum_{t=1}^{n_d}H_{st}B_t^T$ so
    that $A=\sum_{s=1}^{n_d}B_s\tilde{H}_s$. This strategy leads to nearly
    the same rounding error analysis and error bounds.
\end{remark}

\paragraph{Algorithm 2. Local kernel algorithm for linear forms.}
\begin{enumerate}
	\item Tabulate the basis function values forming $B$.
    \item Compute $\bm{r}_s$ for all $s$.
    \item Compute $\bm{v}=\sum_{s=1}^{n_d}B_s\bm{r}_s$ (using cell batching).
\end{enumerate}
\vspace{3pt}

\begin{remark}
    For the sake of brevity, we only consider the above algorithms. We also do
    not analyze sum factorization approaches which are often used for tensor
    product FEs (see, e.g., \cite[Section~4.1]{karniadakis2005spectral}). The
    adaptation of our rounding error analysis to different kernel
    implementations should hopefully be straightforward.
\end{remark}

Once the local kernels have been run over all cells, every FE code assembles the
resulting local tensors into a global tensor. This process, called assembly,
is the natural application of faster FE kernels and is briefly described next.

\subsection{Global Assembly}
Assembly is the process through which the local matrices and vectors obtained
from running the kernels are collected into a global tensor.  Let $n_K$ be the
total number of cells and let $n_g$ be the number of global degrees of freedom.
Using linear algebra, we can write the assembled global matrix $A^g$ and vector
$\bm{v}^g$ as
\begin{align}
    \label{eq:global_assembly}
    A^g = LDL^T,\quad \bm{v}^g=L\bm{b},
\end{align}
where
\begin{align}
    \label{eq:assembly_matrixvec_def}
    D = \textnormal{diag}_{k=1}^{n_K}(A^k),\quad \bm{b} = \textnormal{vstack}_{k=1}^{n_K}(\bm{v}^k),
\end{align}
and $A^k$, $\bm{v}^k$ are the local kernel matrix and vector on the $k$-th cell
respectively. Here $L\in \{0,1\}^{n_g\times n_{\phi}n_K}$ is a boolean assembly
matrix with at most one non-zero per column that encodes the local-to-global
map, see e.g.,~\cite{wathen1987realistic}. We take the evaluation of
\eqref{eq:global_assembly} as the assembly algorithm to be investigated in our
rounding error analysis.  For the sake of brevity, we do not consider the
multitude of assembly algorithms available in the literature, but we remark
that parallel implementations only change the order in which
floating-point computations are performed and thus has no influence over the
worst-case error analysis derived next.

\section{Rounding error analysis}
\label{sec:rounding_error_analysis}

In this section we perform a comprehensive rounding error analysis of Algorithms
1 and 2 and of the assembly process. The analysis also includes mixed-precision
implementations and accounts for hardware acceleration. That we are aware of this is the first
rounding error analysis of FE kernels as well as the first rigorous result
guiding the use of mixed-precision hardware accelerators in FE kernels.

Before we begin, we make two simplifying assumptions:

\begin{assumption}
    We ignore floating point range effects, i.e., we assume underflow and
    overflow do not occur.
\end{assumption}

\begin{assumption}
    We ignore rounding errors arising from mesh-generation and the computation
    of quadrature rules.
\end{assumption}

This latter assumption is sensible whenever the mesh and quadrature rules
are computed in higher precision than the target accuracy of our computations.
We remark that order-machine precision perturbations to mesh coordinates
correspond to small relative perturbations of the mesh size and are thus
unlikely to affect the overall FE approximation error provided that the mesh is
of good enough quality.

We proceed in our analysis following the outline of Algorithms 1 and 2. We
consider, in order, rounding errors arising from:
\begin{enumerate}
    \item Basis function and finite element function evaluation (Section \ref{subsec:polyval}).
    \item Geometry computations (Section \ref{subsec:geometry}).
    \item Mixed-precision construction of the matrices and vectors appearing in Algorithms 1 and 2 (Section \ref{subsec:MPmatrix_construction}).
    \item Mixed-precision accumulation of matrix-matrix products into $A$ and $v$ and global assembly (Section \ref{subsec:final_matmats_sum}).
\end{enumerate}

\subsection{Finite element tabulation and function evaluation}
\label{subsec:polyval}

Tabulating FE basis functions amounts to evaluating the reference basis
functions and their derivatives, comprised of multivariate polynomials, at the
quadrature nodes.  Similarly, evaluating functions belonging to FE approximation
subspaces consists of computing linear combinations of FE basis functions.

The rounding error analysis of univariate polynomials is classical (cf.~Section
5 in \cite{higham2002accuracy}). The evaluation of univariate Lagrange
polynomials has been analyzed by Higham in \cite{higham2004numerical} and the
multivariate version of Horner's algorithm has also been studied
\cite{pena2000multivariate}. Nevertheless, there is no specific work on the
evaluation of linear combinations of multivariate Lagrange polynomials and their
derivatives. We present such an analysis in Appendix \ref{appendix_sec:polyval}.
Here, we only describe the specific application of this analysis to Lagrange FE
over simplices or boxed cells.

First, let $V_{\K}=\textnormal{span}\left(\bm{\Phi}\right)$,
where the entries of $\bm{\Phi}$ form the degree-$p$ nodal Lagrange reference
basis over $\K$. Also let $\partial_s V_\K =
\textnormal{span}(\partial_s\bm{\Phi})$ and define the following condition numbers:
\begin{align}
    \kappa(V_{\K}) :=
    \max_{\X\in \K}\
    \lVert\bm{\Phi}(\X)\rVert_1,\quad\kappa(\partial_{s}V_{\K}) := \max_{\X\in
    \K}\
    \lVert\partial_{s}\bm{\Phi}(\X)\rVert_1,\quad\text{and}\quad
    \kappa(\partial_{s}\bm{\Phi}) :=
    \max_{i=1,\dots,n_\phi}\kappa(\partial_{s}\phi_i).
\end{align}
Here $\kappa(V_{\K})$ and $\kappa(\partial_s V_{\K})$ are the condition numbers
(also known as the Lebesgue constants) of the bases spanning $V_{\K}$ and
$\partial_s V_{\K}$ respectively, while $\kappa(\partial_s \phi_i)$ is instead
the condition number of the \emph{evaluation} of the partial derivative of the
$i$-th basis function (we refer to Lemma \ref{lemma:polyval_derivatives} in
Appendix \ref{appendix_sec:polyval} for the formal definition of
$\kappa(\partial_s \phi_i)$).

\begin{remark}
    The quantity $\kappa(V_{\K})$ is the Lebesgue constant of the FE basis. It
    is well known that it can be made to grow polylogarithmically slow with the
    polynomial degree for boxed cells provided that the degrees of freedom are
    suitably chosen.  On simplices the Lebesgue constant is larger, although it
    can be kept small even for moderately high-degree polynomials by
    constructing the FE basis accordingly, see e.g., \cite{isaac2020recursive}.
\end{remark}

We now state the following Lemma which provides
rounding error bounds for the evaluation of FE functions and their derivatives:

\begin{lemma}
    \label{lemma:polyval_lagrange_specific}
    Let $\K$ be a $d$-dimensional reference simplex or hypercube and let
    $V_{\K}=\textnormal{span}\left(\bm{\Phi}\right)$, where the entries of
    $\bm{\Phi}$ form the degree-$p$ nodal Lagrange reference basis over $\K$.
    Then, if $u$ is sufficiently small, the evaluation of $\phi_i(\X)$ and its
    derivatives in precision $u$ satisfy, for all $i$ and for all $\X\in \K$,
    the bounds
    \begin{align}
        \label{eq:coroll_polyval_lagrange_specific_1}
        |\phi_i(\X) - \hat{\phi}_i(\X)| &\lesssim dpu|\phi_i(\X)|,\\[0.25em]
        \max_{\X\in \K}|\partial_{s}\phi_i(\X) -
        \widehat{\partial_{s}\phi_i}(\X)| &\lesssim
        dpu\kappa(\partial_{s}\phi_i)\max_{\X\in
            \K}|\partial_{s}\phi_i(\X)|,
        \label{eq:coroll_polyval_lagrange_specific_2}
    \end{align}
    Furthermore, the evaluation of a function
    $z_h(\X)=\sum_{i=1}^{n_\phi}z_i\phi_i(\X)$ and its derivatives satisfy the
    bounds
    \begin{align}
        \label{eq:coroll_polyval_lagrange_specific_3}
        \max_{\X\in \K}|z_h - \hat{z}_h| &\lesssim p^du\kappa(V_{\K})\lVert
        \bm{z} \rVert_{\infty},\\[0.25em]
        \max_{\X\in \K}|\partial_{s}z_h - \widehat{\partial_{s}z_h}| &\lesssim
        p^du \kappa(\partial_{s}\bm{\Phi})\kappa(\partial_{s}V_{\K})\lVert
        \bm{z} \rVert_{\infty}.
        \label{eq:coroll_polyval_lagrange_specific_4}
    \end{align}
\end{lemma}

\begin{proof}
    See Appendix \ref{appendix_sec:polyval}.
\end{proof}

\subsection{Geometry evaluation}
\label{subsec:geometry}

In this subsection we analyze the propagation of rounding errors in geometry
evaluations. We first recall the expressions of the mass and Poisson geometry
tensors:
\begin{align}
    \begin{array}{ll}
        G(\X) = |\det(J)|, & \text{(mass form)}, \\[1em]
        G_{ij}(\X) = |\text{det}(J)|\sum_{k=1}^d(J^{-1}_{ik}J^{-1}_{jk}), & \text{(Poisson form)}, 
    \end{array}
\end{align}
where $J=J(\X)\in\mathbb{R}^{d\times d}$ is the
Jacobian matrix. Note that on simplices $J$ and $G$ are constant in space.

Our rounding error analysis proceeds by considering, in turn: 1) Evaluations of
the Jacobian (Section \ref{subsubsec:jacobian_eval}). 2) Evaluation of the
Jacobian determinant, i.e., the mass form geometry tensor (Section
\ref{subsubsec:jacobian_det}).  3) Evaluation of the Jacobian inverse (Section
\ref{subsubsec:jacobian_inv}). 4) Evaluation of the Poisson geometry tensor
(Section \ref{subsubsec:poisson_geom}).

\subsubsection{Jacobian evaluations.}
\label{subsubsec:jacobian_eval}
We focus on first-order meshes only for simplicity and we analyse the following
evaluation algorithm: Let $K\in D_h$, let $\{\psi_k(\x)\}_{k=1}^{n_\psi}$ be the degree-$1$
Lagrange basis, and let $\{\phi_k(\X)\}_{k=1}^{n_\psi}$ be the corresponding
reference basis over $\K$, the reference cell. Every point in $K$ then satisfies
\begin{align}
    \x = \sum_{k=1}^{n_\psi}\x^k\psi_k(\x) =
    \sum_{k=1}^{n_\psi}\x^k\phi_k(\X),
\end{align}
where $\x^k$ is the $k$-th vertex of $K$ and $\X\in\K$.  We can use the above to
derive an expression for the Jacobian:
\begin{align}
    \label{eq:Jacobian_def}
    J_{ij}(\X)=\frac{\partial x_i}{\partial X_j} = \sum_{k=1}^{n_\psi} x_i^k
    (\partial_j \phi_k(\X)).
\end{align}

For which the following result holds:
\begin{lemma}
    \label{lemma:jacobian_eval}
    Let $K$ be a non-degenerate cell with vertices $\{\x^k\}_{k=1}^{n_\psi}$ and let
     $\mathcal{X}\in \mathbb{R}^{d\times n_\psi}$ with $\mathcal{X}_{ik} = x_i^k$.
     Let $\{\psi_k(\x)\}_{k=1}^{n_\psi}$ be the degree-$1$ Lagrange
     basis, and let $\{\phi_k(\X)\}_{k=1}^{n_\psi}$ be the corresponding
     reference basis over the reference cell.  Let $J(\X)$ be defined as in
     \eqref{eq:Jacobian_def}. If $u$ is sufficiently small, the evaluation of
     $J$ yields instead $\hat{J}$ satisfying, for all $\X\in \K$,
    \begin{align}
        \label{eq:lemma_jacobian}
        \hat{J} = J + \Delta J,\quad &\text{where}\quad |\Delta J| \leq
        c_1(d)u\lVert \mathcal{X} \rVert_{\infty},\quad \lVert\Delta
        J\rVert_{\infty} \leq c_2(d)u\kappa(K)\lVert J
        \rVert_{\infty},\quad \kappa(K) = n_\psi\frac{\lVert \mathcal{X}
        \rVert_\infty}{\lVert J \rVert_{\infty}}\geq 1.
    \end{align}
    Here $c_1(d)>1$ and $c_2(d)>1$ are constants that grow exponentially
    with~$d$.
\end{lemma}

\begin{proof}
    The derivatives of degree-$1$ Lagrange basis functions have simple
    expressions: on simplicies they are constant, while for boxed
    cells they are products of $d$ monomials and thus satisfy the assumptions of
    Lemma \ref{lemma:FEM_function_eval}. Invoking Lemma
    \ref{lemma:FEM_function_eval} we thus obtain the bound
    \begin{align}
        \label{eq:_lemma_jacobian_1}
        \hat{J} = J + \Delta J,\quad\text{where}\quad |\Delta J| \leq
        \gamma_{n_\psi+d(r+2)}C(\nabla \bm{\Phi})\lVert \mathcal{X} \rVert_{\infty},\quad 
        \text{where}\quad C(\nabla\bm{\Phi}) =
        \max\limits_{j,\X}\sum_{k=1}^{n_\psi}\left|\partial_j \phi_k(\X)\right|.
    \end{align}
    We note that for all types of cells, we have $r=O(1)$. Additionally, it
    holds that $C(\nabla\bm{\Phi})\leq n_\psi\leq 2^d$ since each derivative is bounded in
    absolute value by $1$ for all points on the reference cell. We can thus
    bound
    \begin{align}
        \label{eq:_lemma_jacobian_2}
        \gamma_{n_\psi+d(r+2)}C(\nabla\bm{\Phi}) \leq \gamma_{n_\psi^2+d(r+2)}\leq c(d)u,
    \end{align}
    for some constant $c(d)>1$ which grows exponentially with $d$. Combining
    \eqref{eq:_lemma_jacobian_1} with \eqref{eq:_lemma_jacobian_2} yields the
    first bound in \eqref{eq:lemma_jacobian}. To obtain the second bound it is
    sufficient to take the infinity norm on both sides of the first bound and
    multiply and divide by $\lVert J \rVert_{\infty}$ (since the cell is
    non-degenerate by assumption the norm of $J$ is nonzero). The factor $n_\psi$ in
    the definition of $\kappa(K)$ is there to ensure that $\kappa(K)\geq 1$.
    Indeed we have $\lVert J \rVert_{\infty}\leq C(\nabla \bm{\Phi})\lVert \mathcal{X}
    \rVert_{\infty}\leq n_\psi \lVert \mathcal{X} \rVert_{\infty}$.
\end{proof}

\begin{remark}
    The bound \eqref{eq:lemma_jacobian} can be generalized to higher-order meshes by using the
    analysis from Appendix \ref{appendix_sec:polyval} to account for derivatives with a
    higher polynomial degree, but we refrain from doing so here for the sake of
    simplicity.
\end{remark}


\subsubsection{Jacobian determinants and mass geometry.}
\label{subsubsec:jacobian_det}
The errors in the evaluation of $\text{det}(J)$ are bounded by the following
result:
\begin{theorem}
    \label{th:detJ}
    Let the assumptions of Lemma \ref{lemma:jacobian_eval} hold and consider the
    evaluation of $\textnormal{det}(J)$ via LU factorization as in Theorem \ref{th:LU_det}.
    Then, provided $u$ is sufficiently small, the following forward error bound holds
    \begin{align}
        \label{eq:detJ}
        \widehat{\textnormal{det}(J)} = \textnormal{det}(J) +
        \Delta_{\textnormal{det}},\quad\text{where}\quad
        |\Delta_{\textnormal{det}}|\leq c(d)\kappa_2(J)\kappa(K)u|\textnormal{det}(J)|.
    \end{align}
    Here $c(d)>1$ is a constant which grows with $d$. The same bound holds after
    replacing the determinants with their absolute values:
    \begin{align}
        \label{eq:detJabs}
        |\widehat{\textnormal{det}(J)}| = |\textnormal{det}(J)| +
        \Delta_{\textnormal{absdet}},\quad\text{where}\quad
        |\Delta_{\textnormal{absdet}}|\leq |\Delta_{\textnormal{det}}|.
    \end{align}
\end{theorem}

\begin{proof}
    Throughout the proof we ignore $d$-dependent constants for simplicity and
    simply write $a \lesssim b$ to indicate that there exists $c(d)\geq 1$ such
    that $a \leq c(d)b$. For simplicity of notation, let $Y=\hat{J}$ where
    $\hat{J}$ satisfies equation \eqref{eq:lemma_jacobian}. If we perform an LU
    factorization of $Y$ we obtain the perturbed LU factors $\hat{L}$ and
    $\hat{U}$ which satisfy the bound in Theorem \ref{th:LU_det}. Therefore,
    \begin{align}
        \hat{L}\hat{U} &= Y + \Delta Y = J + \Delta J + \Delta Y = J +
        E,\\
        \text{where}\quad \lVert \Delta Y \rVert_2 &\lesssim u \lVert
        \ |\hat{L}||\hat{U}|\ \rVert_2 \lesssim u \lVert Y \rVert_2 \lesssim u
        (\lVert J \rVert_2 + \lVert \Delta J \rVert_2),\\
        \lVert \Delta J \rVert_2 &\lesssim u\kappa(K)\lVert J
    \rVert_{\infty} \lesssim u\kappa(K)\lVert J
    \rVert_{2},
    \end{align}
    where $E=\Delta J + \Delta Y$. Here we have used Lemma
    \ref{lemma:jacobian_eval}, the triangle inequality, and the bound $\lVert\
    |\hat{L}| |\hat{U}|\ \rVert_2 \lesssim u\lVert Y \rVert_2$ which is a
    well-known result applying to LU factorizations
    (see \cite[Lemma~9.6 and Section~9.4]{higham2002accuracy}).
    Thus, ignoring higher-order terms in $u$, we obtain
    \begin{align}
        \lVert E \rVert_2 &\leq \lVert \Delta J \rVert_2 + \lVert \Delta Y
        \rVert_2 \lesssim u(1+\kappa(K))\lVert J \rVert_2 \lesssim
        u\kappa(K)\lVert J \rVert_2,
    \end{align}
    where we have absorbed the $1$ factor into the constant.  Applying Theorem
    \ref{th:ipsen_coroll} to $\hat{L}\hat{U}=J+E$ and plugging in the newly
    found bound for $\lVert E \rVert_2$, we obtain
    \begin{align}
        \label{eq:_proof_det_bound_full_1}
        |\textnormal{det}(\hat{U})-\textnormal{det}(J)|=|\textnormal{det}(J+E)-\textnormal{det}(J)|\lesssim \lVert J^{-1}
        \rVert_2\lVert E \rVert_2|\textnormal{det}(J)| \lesssim
        u\kappa_2(J)\kappa(K)|\textnormal{det}(J)|,
    \end{align}
    where we have used the fact that
    $\textnormal{det}(\hat{U})=\textnormal{det}(\hat{L}\hat{U})$. Note that
    there is no need to account for the $(1+\theta_n)$ term appearing in
    equation \eqref{eq:_LU_det_theta_n} (cf.~Theorem \ref{th:LU_det}), since
    this term only adds an $O(u|\textnormal{det}(J)|)$ error which can be
    harmlessly absorbed into the implicit constant factor in
    \eqref{eq:_proof_det_bound_full_1}. Hence, we can
    replace the left side of equation \eqref{eq:_proof_det_bound_full_1} with
    $|\Delta_{\text{det}}|$ which in turn yields \eqref{eq:detJ}.

    Note that replacing the determinant with its absolute values yields the same
    bound since $\widehat{|a|}=|\hat{a}|$ and thus, by the reverse triangle
    inequality,
    \begin{align}
        |\Delta_{\text{absdet}}| = |\ |\widehat{\det(J)}| - |\det(J)|\
        | \leq | \widehat{\det(J)} - \det(J) | = | \Delta_{\text{det}} |.
    \end{align}
    Therefore \eqref{eq:detJabs} also holds and the proof is concluded.
\end{proof}

Theorem \ref{th:detJ} automatically yields a bound for the mass geometry tensor
(which is simply $|\textnormal{det}(J)|$). For the Poisson geometry the process
is more complicated and we first need an intermediate result: a rounding error
bound for the computation of $J^{-1}$.

\subsubsection{Jacobian inverse.}
\label{subsubsec:jacobian_inv}
The bound for the inverse Jacobian is established
by the following lemma:
\begin{lemma}
    \label{lemma:jacobian_inverse}
    Let Assumption \ref{ass:matrix_inverse_algorithm} and the assumptions of
    Lemma \ref{lemma:jacobian_eval} hold. Then, for $u$ sufficiently small, it
    holds
    \begin{align}
        |\widehat{J^{-1}}-J^{-1}|\leq c(d)\kappa_2(J)\kappa(K)u|J^{-1}|,
    \end{align}
    where $c(d)>1$ is a constant growing with $d$.
\end{lemma}

\begin{proof}
    Throughout the proof we ignore $d$-dependent constants for simplicity and
    simply write $a \lesssim b$ to indicate that there exists $c(d)\geq 1$ such
    that $a \leq c(d)b$. By the triangle inequality
    \begin{align}
        |\widehat{J^{-1}} - J^{-1}| \leq | \widehat{J^{-1}} - A^{-1}| + |A^{-1}
        - J^{-1}|.
    \end{align}
    Setting $A=\hat{J}=J+\Delta J$ and replacing $Y$ with $\widehat{J^{-1}}$ in
    \eqref{eq:matrix_inv} we can bound the first term on the right and obtain
    \begin{align}
        \label{eq:_lemma_jacobian_inverse_1}
        |\widehat{J^{-1}} - J^{-1}| \lesssim u|A^{-1}||A||A^{-1}| + |A^{-1} -
        J^{-1}|.
    \end{align}
    To leading order in $u$ we have that
    \begin{align}
        A^{-1} = (J+\Delta J)^{-1} = A^{-1} - A^{-1}\Delta J A^{-1} + O(u^2),
    \end{align}
    and therefore,
    \begin{align}
        |A^{-1}||A||A^{-1}| = |J^{-1}||J||J^{-1}| + O(u),\quad\text{and}\quad
        |A^{-1} - J^{-1}| = |J^{-1}||\Delta J||J^{-1}| + O(u^2).
    \end{align}
    Replacing the above into \eqref{eq:_lemma_jacobian_inverse_1}, ignoring
    $O(u^2)$ terms and invoking Lemma \ref{lemma:jacobian_eval} we obtain
    \begin{align}
        |\widehat{J^{-1}} - J^{-1}| \lesssim u|J^{-1}||J||J^{-1}| +
        |J^{-1}||\Delta J||J^{-1}| \lesssim
        \kappa_2(J)\kappa(K)u|J^{-1}|,
    \end{align}
    which is the thesis. Note that we have used the fact that $|J||J^{-1}| \leq
    \lVert\ |J||J^{-1}|\ \rVert_{\max} \leq \lVert J \rVert_2 \lVert J^{-1}
    \rVert_2 = \kappa_{2}(J)$.
\end{proof}

\subsubsection{Poisson geometry.}
\label{subsubsec:poisson_geom}
We can now combine the above bounds to obtain a rounding error estimate for the
Poisson geometry tensor.
\begin{theorem}
    \label{th:poisson_geometry}
    Let the assumptions of Theorem \ref{th:detJ} and Lemma
    \ref{lemma:jacobian_inverse} hold and let $G\in\mathbb{R}^{d\times d}$ be
    the Poisson geometry tensor. Its evaluation in precision $u$ yields
    $\hat{G}$ satisfying
    \begin{align}
        \label{eq:thm_geometry_1}
        \hat{G} = G + \Delta G,\quad\text{where}\quad &|\Delta G| \leq
        c_1(d)u\kappa(K)\kappa_2(J)\tilde{G}(\X),\\[0.5em]
        \text{and}\quad &\lVert \Delta G \rVert_{\max} \leq c_2(d)u\kappa(K)\kappa_2(J)\lVert
        G \rVert_{\max},
        \label{eq:thm_geometry_2}
    \end{align}
    where $\tilde{G}(\X)=|\textnormal{det}(J)(\X)||J^{-1}(\X)||J^{-T}(\X)|$
    and $c_1(d),c_2(d)\geq 1$ are constants growing with $d$.
\end{theorem}

\begin{proof}
    Again, we ignore $d$-dependent constants in the proof and
    simply write $a \lesssim b$ to indicate that there exists $c(d)\geq 1$ such
    that $a \leq c(d)b$. Before we proceed, we write $G$ as
    \begin{align}
        G_{ij}(\X) = |\text{det}(J)|\sum_{k=1}^d(J^{-1}_{ik}J^{-1}_{jk}) =
        |\text{det}(J)|(\bm{a}_i^T\bm{a}_j),
    \end{align}
    where we have indicated with $\bm{a}_i^T$ the $i$-th row of $J^{-1}$.
    Computing $G_{ij}$ in finite precision then yields
    \begin{align}
        \hat{G}_{ij} =
        |\widehat{\text{det}(J)}|\widehat{(\hat{\bm{a}}_i^T\hat{\bm{a}}_j)}(1+\delta).
    \end{align}
    In what follows we ignore the $(1+\delta)$ term for simplicity it does not
    affect the final bound. To begin with, we require an error bound on the
evaluation of $\bm{a}_i^T$, the rows of $J^{-1}$. Applying Lemma
    \ref{lemma:jacobian_inverse} yields
    \begin{align}
        \hat{\bm{a}}_i = \bm{a}_i + \Delta \bm{a}_i,\quad\text{where}\quad |\Delta
        \bm{a}_i|\leq c(d)\kappa_{2}(J)\kappa(K)u|\bm{a}_i|.
    \end{align}
    Invoking the inner product bound \eqref{eq:inner_prods} and Theorem
    \ref{th:detJ} gives
    \begin{align}
        \hat{G}_{ij} =
        |\widehat{\text{det}(J)}|\widehat{(\hat{\bm{a}}_i^T\hat{\bm{a}}_j)}=(|\text{det}(J)|
        + \Delta_{\text{absdet}})\left((\hat{\bm{a}}_i + \Delta
        \hat{\bm{a}}_i)^T(\hat{\bm{a}}_j)\right),
    \end{align}
    where
    \begin{align}
        |\Delta \hat{\bm{a}}| &\lesssim u|\hat{\bm{a}}_i|\lesssim
        u\left(1+u\kappa_{2}(J)\kappa(K)\right)|\bm{a}_i|,\\[0.5em]
        |\Delta_{\text{absdet}}| &\lesssim \kappa_2(J)
        \kappa(K)u|\textnormal{det}(J)|.
    \end{align}
    Pulling together all bounds for all perturbation terms we obtain that, to
    leading order in $u$,
    \begin{align}
        |\Delta G_{ij}|=|\hat{G}_{ij} - G_{ij}| \lesssim&\ |\det(J)|\left(
        |\bm{a}_i|^T|\Delta\bm{a}_j| + |\Delta \bm{a}_i|^T|\bm{a}_j| + |\Delta
    \hat{\bm{a}}_i|^T|\bm{a}_j|\right) +
    |\Delta_{\text{absdet}}||\bm{a}_i|^T|\bm{a}_j|\\[0.5em]
            \lesssim&\ (1 +
            \kappa_2(J)\kappa(K))u|\textnormal{det}(J)||\bm{a}_i|^T|\bm{a}_j|\lesssim
            \kappa_2(J)\kappa(K)u|\textnormal{det}(J)||\bm{a}_i|^T|\bm{a}_j|.
    \end{align}
    The above yields \eqref{eq:thm_geometry_1} since $|\bm{a}_i|^T|\bm{a}_j| =
    (|J^{-1}||J^{-T}|)_{ij}$. Furthermore, the same result holds if we replace
    the absolute value with the $2$-norm, which gives
    \begin{align}
        \label{eq:_thm_geometry_1}
        \lVert\Delta
        G\rVert_2\lesssim\kappa_2(J)\kappa(K)u|\textnormal{det}(J)|\lVert\
        |J^{-1}||J^{-T}|\ \rVert_2.
    \end{align}
    To conclude the proof, we note that
    \begin{align}
        |\textnormal{det}(J)|^{-1}\lVert \tilde{G} \rVert_2 = \lVert\ |J^{-1}||J^{-T}|\ \rVert_2 &\leq \left( \lVert\ |J^{-1}||J^{-T}|\
        \rVert_1\ \  \lVert\ |J^{-1}||J^{-T}|\
    \rVert_{\infty}\right)^{1/2} \leq \lVert J^{-1}\rVert_1\lVert J^{-1}
    \rVert_{\infty}\\[0.5em]
    &\leq d \lVert J^{-1} \rVert_2^2 = d\lVert J^{-1} J^{-T}
    \rVert_2 = d\lVert G \rVert_2 |\textnormal{det}(J)|^{-1}.
    \end{align}
    Plugging the above into \eqref{eq:_thm_geometry_1} yields
    \begin{align}
        \lVert \Delta G \rVert_{\max} \leq \lVert \Delta G \rVert_2
        \lesssim u\kappa(K)\kappa_2(J)\lVert G \rVert_{2} \lesssim u\kappa(K)\kappa_2(J)\lVert G \rVert_{\max},
    \end{align}
    which is \eqref{eq:thm_geometry_2} and the thesis.
\end{proof}

\begin{remark}
    The analysis developed here should easily generalize to different forms and
    elements since geometry tensors often include products of entries of $J$
    and/or $J^{-1}$ as well as powers of $|\textnormal{det}(J)|$
    (see e.g., \cite{kirby2007efficient,rognes2010efficient}).
\end{remark}

\subsection{Mixed-precision construction of local matrices and vectors}
\label{subsec:MPmatrix_construction}

We now focus on the evaluation of the matrices and vectors appearing in
Algorithms 1 and 2. Here, we consider a mixed-precision implementation in which
different precisions are used for different computations. In particular, we employ
the precisions in the following list.

\paragraph{List 1. Precisions used in the analysis:}
\begin{itemize}
    \item Precision $u_p$ to tabulate and evaluate basis functions and functions
        in FE approximation subspaces.
    \item Precision $u_g$ to compute the geometry tensor (including the
        polynomial evaluations used to compute $J$).
    \item Precision $u_q$ to perform and accumulate the matrix-matrix and
        matrix-vector products into $A$ and $\bm{v}$ respectively.
    \item Precision $u_s$ with $\max(u_p,u_g,u_q)\leq u_s$ to store $B$, $\bm{r}_s$, and
        $C$, and to compute $H$.
\end{itemize}
Including a storage precision $u_s$ is useful as it describes reduced-precision
storage strategies and the possible decoupling of memory and computation
floating-point formats. Indeed, reducing the storage precision may lead to
memory savings as well as better cache use. Furthermore, AMX and tensor core
accelerators do perform mixed-precision computations in which the arrays are
stored at a lower precision than the arithmetic precision.

We remark that the use of multiple precisions is only a tool for deriving a
finer-grained analysis and that we do not necessarily advocate using four
different floating-point formats. For instance, it is sufficient to set
$u_p=u_g=u_q=u_s$ to obtain a single (in the sense of one) working precision
result.  In Section \ref{subsec:MP_strategy}, we will show how these precision
can be chosen to tailor our analysis to AMX- and tensor-core-accelerated
kernels.

\paragraph{Construction of $C$.}
We first focus on the construction of the fourth-order tensor
$C\in\mathbb{R}^{n_d\times n_d \times n_q \times n_q}$, which we recall is given
by
\begin{align}
    C_{st}=\textnormal{diag}_{q=1}^{n_q}(\omega_qG_{st}(\X_q)\check{z}(\X_q))\in\mathbb{R}^{n_q\times n_q},
\end{align}
where $s,t\in \{1,\dots,n_d\}$. A forward error bound is given by the following
theorem:

\begin{theorem}
    \label{th:C_MP_eval}
    Let the assumptions of Lemma \ref{lemma:polyval_lagrange_specific} hold for
    the Lagrange basis and let the assumptions of Theorems \ref{th:detJ} and
    \ref{th:poisson_geometry} hold for the mass and Poisson form respectively.
    Let $\tilde{G}$ be as in Theorem \ref{th:poisson_geometry} for Poisson and
    $\tilde{G}\equiv G$ for the mass form. Let $G(\X)$ be evaluated in precision
    $u_g$, $\check{z}(\X)=\bm{z}^T\bm{\Phi}(\X)$ be evaluated with precision $u_p$ and then cast to
    precision $u_g$, and $C$ be computed using precision $u_g$ and then rounded
    to precision $u_s$. Then, if $\max(u_g,u_p)\leq u_s$, the evaluation of $C$
    yields instead $\hat{C}$ satisfying
    \begin{align}
        \hat{C} = C + \Delta C,\quad\text{where}\quad\ |\Delta C| \lesssim u_s|C| + u_p\Theta^C + c(d)u_g \Gamma^C + O(u_s^2 + (u_g + u_p)^2),
    \end{align}
    Where $\Theta^C,\Gamma^C$ are fourth order tensors of the same sizes and
    sparsity pattern of $C$. They are given by
    \begin{align}
        \Theta^C_{st} = p^d\kappa(V_{\K})\lVert \bm{z}
        \rVert_{\infty}\textnormal{diag}_{q=1}^{n_q}\left(|\omega_qG_{st}(\X_q)|\right),\quad
        \Gamma^C_{st}=\kappa(K)\ \textnormal{diag}_{q=1}^{n_q}\left(\kappa_2(J(\X_q))|\omega_q\tilde{G}_{st}(\X_q)\check{z}(\X_q)|\right).
    \end{align}
    If $\check{z}(\X)$ is constant, the bound holds with $\Theta^C = 0$.
\end{theorem}

\begin{proof}
    The evaluation in finite precision of $C_{stqq}$ yields
    \begin{align}
        \omega_q(G_{st}(\X_q) + \Delta G_{st}(\X_q))(\check{z}(\X_q) + \Delta \check{z}(\X_q))(1+\theta^g_4)(1+\delta^s),
    \end{align}
    where $|\theta^g_4|\leq cu_g$ for some constant $c\geq4$ and
    $|\delta^s|<u_s$, $\Delta G_{st}=\hat{G}_{st}-G_{st}$, and $\Delta
    \check{z}(\X_q)=\hat{\check{z}}(\X_q)-\check{z}(\X_q)$. Here $\theta_4^g$ represents the error due to
    casting $\check{z}(\X_q)$ and $\omega_q$ to precision $u_g$ as well as the error in the
    multiplications forming $C$. The quantity $\delta^s$ accounts for the final
    rounding of $C$ to precision $u_s$. The evaluation error is thus given by
    \begin{align}
        \label{eq:_theorem_C_1}
        |\hat{C}_{stq}-C_{stq}| \leq |C|(|\theta^g_4| + |\delta_s|) +
        \bar{c}|\Delta \check{z}(\X_q)||\omega_qG_{st}(\X_q)| +
        \bar{c}|\omega_q\check{z}(\X_q)||\Delta G_{st}| + O(u_s^2 + (u_g + u_p)^2),
    \end{align}
    where $\bar{c}\geq 1$ is a constant such that we can upper bound
    $|(1+\theta^g_4)(1+\delta_s)|\leq \bar{c}$. We now have that $|\theta_4^g| +
    |\delta_s| \leq \tilde{c}u_s$ for some $\tilde{c}\geq 1$ since $u_g\leq
    u_s$ and we can now invoke our previous results to bound all the other
    quantities appearing in the bounds. Indeed, the quantity $|\Delta G_{st}|$ is
    bounded as described by Theorem \ref{th:poisson_geometry} (for Poisson) and
    by Lemma \ref{th:detJ} (for the mass form since
    $G_{st}=|\textnormal{det}(J)|$ in this case), while $|\Delta \check{z}(\X_q)|$
    satisfies the same bound as in equation
    \eqref{eq:coroll_polyval_lagrange_specific_3}. Therefore, it holds that
    \begin{align}
        |\Delta \check{z}(\X_q)|\leq \bar{c}_1p^du_p\kappa(V_{\K})\lVert \bm{z}
        \rVert_{\infty},\quad |\Delta G_{st}| \leq
        \bar{c}_2(d)u_g\kappa(K)\kappa_2(J(\X_q))|\tilde{G}_{st}(\X_q)|,
    \end{align}
    for some generic constants $\bar{c}_1,\bar{c}_2(d)\geq 1$. Plugging the
    above into \eqref{eq:_theorem_C_1} yields the thesis. Note that if $\check{z}(\X)$
    is constant, then $\Delta \check{z}(\X) = (1+\delta_p)\check{z}(\X)$ for all $X$ where
    $|\delta_p|\leq u_p \leq u_s$ so we can absorb the error in the evaluation
    of $\check{z}$ into the $O(u_s)$ term in \eqref{eq:_theorem_C_1} and set
    $\Theta^C=0$.
\end{proof}

\paragraph{Construction of $\bm{r}_s$.}
A similar result holds for actions, but first let us recall the expression for
$\bm{r}_s$:
\begin{align}
    \bm{r}_s &= \textnormal{vstack}_{q=1}^{n_q}\left(\omega_q\check{z}(\X_q)G(\X_q)\check{w}(\X_q)\right),&&\hspace{-3cm}\text{(mass form)},\\
    \bm{r}_s &= \textnormal{vstack}_{q=1}^{n_q}\left(\omega_q\check{z}(\X_q)(G(\X_q)\nabla\check{w}(\X_q))\right),&&\hspace{-3cm}\text{(Poisson form)},
\end{align}
where $s=1,\dots,n_d$. For easiness of notation, we set
\begin{align}
    R = [\bm{r}_1 | \dots | \bm{r}_{n_d}].
\end{align}
An error bound for the mixed-precision computation of
$\bm{r}_s$ is given by the following theorem:
\begin{theorem}
    \label{th:s_MP_eval}
    Let the assumptions of Lemma \ref{lemma:polyval_lagrange_specific} hold for
    the Lagrange basis and let the assumptions of Theorems \ref{th:detJ} and
    \ref{th:poisson_geometry} hold for the mass and Poisson form respectively.
    Let $\tilde{G}$ be as in Theorem \ref{th:poisson_geometry}. Let $G(\X)$ be
    evaluated in precision $u_g$, $\check{z}(\X)=\bm{z}^T\bm{\Phi}(\X)$,
    $\check{w}(\X)=\bm{w}^T\bm{\Phi}(\X)$ and $\nabla \check{w}(\X)$ be
    evaluated with precision $u_p$ and then cast to precision $u_g$, and
    $\bm{r}_{s}$ be computed using precision $u_g$ and then rounded to precision
    $u_s$. Then, if $\max(u_g,u_p)\leq u_s$, the evaluation of $R$ yields
    instead $\hat{R}$ satisfying
    \begin{align}
        \hat{R} = R + \Delta
        R,\quad\text{where}\quad\ |\Delta R| \lesssim
        u_s|R| + u_p\Theta^{R} + c(d)u_g \Gamma^{R} + O(u_s^2 + (u_g + u_p)^2).
    \end{align}
    If $\check{z}(\X)$ is constant, then the same bound holds after replacing
    $\lVert \bm{z} \rVert_{\infty}$ with zero.  Here
    $\Theta^R,\Gamma^R\in\mathbb{R}^{n_q\times n_d}$. For the mass form, these
    are given entrywise by
    \begin{align}
        \Theta^R_{qs} = p^d\kappa(V_{\K})\left(|\check{w}(\X_q)|\lVert \bm{z}
        \rVert_{\infty} + |\check{z}(\X_q)|\lVert \bm{w} \rVert_{\infty}\right)|\omega_qG(\X_q)|,\quad\quad
        \Gamma^R_{qs} = \kappa(K)\kappa_2(J(\X_q))|R_{sq}|,
    \end{align}
    while for the Poisson form it holds 
    \begin{align}
        \Theta^R_{qs} &= p^d|\omega_q|
        \left(\kappa(V_{\K})\lVert \bm{z}
        \rVert_{\infty}(|G(\X_q)\nabla\check{w}(\X_q)|) +
    |\check{z}(\X_q)|\lVert \bm{w}
\rVert_{\infty}\left(|G(\X_q)|\left[\kappa(\partial_{k}\bm{\Phi})\kappa(\partial_{k}V_{\K})\right]_{k=1}^{d}\right)
\right)_s,\\
        \Gamma^R_{qs} &=
        \kappa(K)\kappa_2(J(\X_q))|\omega_q\check{z}(\X_q)|(|\tilde{G}(\X_q)||\nabla
        \check{w}(\X_q)|)_s.
    \end{align}
\end{theorem}

\begin{proof}
    The proof for the mass form is essentially the same as the one of Theorem
    \ref{th:C_MP_eval} and we thus omit it. The only difference is the extra
    $\check{w}(\X_q)$ term which leads to the $(|\check{w}(\X_q)|\lVert \bm{z}
    \rVert_{\infty} + |\check{z}(\X_q)|\lVert \bm{w} \rVert_{\infty})$ term in
    the bound for $\Theta^R$.

    We now consider the Poisson form. We start by analyzing the errors arsing
    from $G\nabla\check{w}$. Owing to the error bound for matrix-vector products
    (cf.~\eqref{eq:matvec_higham}) it holds that
    \begin{align}
        (\widehat{G\nabla \check{w}})_s = ((\hat{G} + \Delta \hat{G})\widehat{\nabla
        \check{w}})_s = \left((G + \Delta G + \Delta \hat{G})(\nabla \check{w} +
        (\widehat{\nabla \check{w}} - \nabla \check{w}))\right)_s(1+\delta_g),
    \end{align}
    where $(1+\delta_g)$ accounts for the rounding of $\nabla \check{w}$ to
    precision $u_g$ and the term $|\widehat{\nabla \check{w}} - \nabla
    \check{w}|$ accounts for the error of evaluating $\nabla \check{w}$ at
    precision $u_p$, which can be bound using equation
    \eqref{eq:coroll_polyval_lagrange_specific_4} of Lemma
    \ref{lemma:polyval_lagrange_specific}. Here the quantity $\Delta G$
    satisfies the bound of Theorem \ref{th:poisson_geometry},
    $|\Delta\hat{G}|\lesssim u_g|G|$ due to equation \eqref{eq:matvec_higham},
    and $|\delta_g|<u_g$. Pulling all these results together, we obtain
    \begin{align}
        \label{eq:_theorem_s_1}
        |\widehat{G\nabla \check{w}} - G\nabla \check{w}| &\leq
        |G||\widehat{\nabla \check{w}} - \nabla \check{w}| + |\Delta G||\nabla
        \check{w}| + c(d)u_g|G\nabla \check{w}| + O((u_g+u_p)^2)\\ &\lesssim
        p^du_p|G|\left[
        \kappa(\partial_{k}\bm{\Phi})\kappa(\partial_{k}V_{\K})\right]_{k=1}^d\lVert
        \bm{w} \rVert_{\infty} + c(d)u_g
        \kappa(K)\kappa_2(J)|\tilde{G}||\nabla \check{w}| + O((u_g+u_p)^2),
        \notag 
    \end{align}
    where $c(d)\geq 1$ and we have omitted the $\X_q$-dependency for easiness of
    notation. In the above, the first term on the right-hand side bounds
    $|G||\widehat{\nabla \check{w}} - \nabla \check{w}|$ and is a consequence of
    taking the infinity norm of \eqref{eq:coroll_polyval_lagrange_specific_4}.
    The second term instead bounds $|\Delta G||\nabla \check{w}|$ and is a
    consequence of Theorem \ref{th:poisson_geometry}. In \eqref{eq:_theorem_s_1}
    the term $c(d)u_g|G||\nabla \check{w}|$ was absorbed into $c(d)$ since
    $|G|\leq |\tilde{G}|$ by Theorem \ref{th:poisson_geometry}.
    
    Turning to the evaluation of $R$, we have that
    \begin{align}
        \hat{R}_{qs} = \omega_q(\check{z}(\X_q)+\Delta \check{z}(\X_q))(G\nabla \check{w} + (\widehat{G\nabla \check{w}} - G\nabla \check{w}))(1+\theta^g_4)(1+\delta_s),
    \end{align}
    where $\Delta \check{z}(\X_q)$ is the error in the evaluation of
    $\check{z}(\X_q)$ at precision $u_p$ and can be bound using Lemma
    \ref{lemma:polyval_lagrange_specific}. Note that the expression above is
    in the same form as equation \eqref{eq:_theorem_C_1} and it can thus be
    bound using the same procedure and by applying \eqref{eq:_theorem_s_1}. The
    result reads
    \begin{align}
        \notag
        |\hat{R}_{qs} - R_{qs}|\lesssim{}& u_s|R_{qs}| \\
        \notag
        &+ p^du_p|\omega_q| \left(
        |\check{z}(\X_q)|\lVert
\bm{w}\rVert_{\infty}|G(\X_q)|\left[
        \kappa(\partial_{k}\bm{\Phi})\kappa(\partial_{k}V_{\K})\right]_{k=1}^d + \kappa(V_{\K})\lVert \bm{z}
\rVert_{\infty}|G(\X_q)\nabla\check{w}(\X_q)|\right)\\
        &+ c(d)u_g
        \kappa(K)\kappa_2(J(\X_q))|\omega_q\check{z}(\X_q)|(|\tilde{G}(\X_q)||\nabla
        \check{w}(\X_q)|)_s + O(u_s^2 + (u_g+u_p)^2),
    \end{align}
    which is the thesis.
\end{proof}

\begin{remark}
    In Theorems \ref{th:C_MP_eval} and \ref{th:s_MP_eval} we have assumed that
    all functions belong to the same finite element approximation space for
    simplicity. The same results as in Theorems \ref{th:C_MP_eval} and
    \ref{th:s_MP_eval} hold, albeit with slightly different values for
    $\Theta^C$ and $\Theta^R$, if different FE approximation spaces are used in
    the same form.
\end{remark}

\paragraph{Construction of $B$ and $H$.}
The entries of $B$ are simply the values of the basis functions (for the mass
form) or their derivatives (for Poisson) at the quadrature nodes. Thus,
Lemma \ref{lemma:polyval_lagrange_specific} directly applies\footnote{For
Poisson, we take the max of $|\partial_s\phi_k(\X)|$ with respect to the
quadrature nodes rather than over the whole reference cell. This is a minimal
variation of the result in Lemma \ref{lemma:polyval_lagrange_specific} and does
not affect its proof.} and the
evaluation of $B$ at precision $u_p\leq u_s$ yields instead $\tilde{B}$ which
satisfies, ignoring constants,
\begin{align}
    |\tilde{B} - B| &\lesssim dpu_p|B| = u_p\Theta^B, &&\text{(mass form)},\\ 
    |\tilde{B}_{skq} - B_{skq}| &\lesssim
    dpu_p\kappa(\partial_{s}\phi_k)\max_i|B_{ski}|=u_p\Theta^B_{skq}, &&\text{(Poisson form)}, 
\end{align}
where $\Theta^B$ is a third-order tensor of the same shape as $B$.
Rounding the result to precision $u_s$ yields $\hat{B}$ satisfying
\begin{align}
    \label{eq:bounds_for_B}
    \hat{B} = B + \Delta B,\quad\text{with}\quad|\Delta B| \lesssim u_s|B| + u_p\Theta^B + O(u_s^2).
\end{align}

As far as
$H_{st}=C_{st}B_t^T$ is concerned, its entries are given by
\begin{align}
    H_{stqk} = C_{stq}B_{tkq},
\end{align}
where $s,l\in \{1,\dots,n_d\}$. Thus,
invoking Theorem \ref{th:C_MP_eval}, we obtain for the computation of $H_{stqk}$
at precision $u_s$,
\begin{align}
    \label{eq:bound_for_E}
    \hat{H} = H + \Delta H,\quad\text{where}\quad\ |\Delta
    H| \lesssim u_s|H| + u_p\Theta^H + c(d)u_g \Gamma^H + O(u_s^2 + (u_g + u_p)^2),
\end{align}
where $\Theta^H_{stkq} = \Theta^C_{stq} + \Theta^B_{tkq}$ and $\Gamma^H_{stkq} =
\Gamma^C_{stq}$ for all $s,t\in \{1,\dots, n_d\}$, $k=1,\dots,n_\phi$,
$q=1,\dots,n_q$.

\subsection{Mixed-precision accumulation of products into \texorpdfstring{$A$}{A} and \texorpdfstring{$v$}{v} and global assembly}
\label{subsec:final_matmats_sum}

\subsubsection{Full local kernel error bounds}
We are now ready for the last step in our error analysis: the accumulation of
matrix and vector products into the local matrix and vector $A$ and $\bm{v}$. We
present both results at the same time since their proofs are similar:

\begin{theorem}
    \label{th:A_MP_eval}
    Let the assumptions of Theorem \ref{th:C_MP_eval} hold and assume that the
    matrix-matrix products and their accumulation into $A$ are performed at
    precision $u_q\leq u_s$. Then, the mixed-precision computation of $A$ using
    the precisions described in List 1 yields instead $\hat{A}$ satisfying 
    \begin{align}
        |\hat{A} - A| \lesssim (u_s + (n_d^2 +
        n_q)u_q)\sum_{s=1}^{n_d}\sum_{t=1}^{n_d}|B_{s}||H_{st}| + u_p\Theta^A + c(d)u_g\Gamma^A + O(u_s^2 + (u_q + u_g + u_p)^2),
    \end{align}
    where $\Theta^A,\Gamma^A\in\mathbb{R}^{n_\phi\times n_\phi}$ are given by
    \begin{align}
        \Theta^A = \sum_{s=1}^{n_d}\sum_{t=1}^{n_d}(|B_{s}|\Theta^H_{st} +
        \Theta^B_{s}|H_{st}|),\quad
        \Gamma^A = \sum_{s=1}^{n_d}\sum_{t=1}^{n_d}|B_s|\Gamma^H_{st}.
    \end{align}
\end{theorem}

\begin{theorem}
    \label{th:v_MP_eval}
    Let the assumptions of Theorem \ref{th:s_MP_eval} hold and assume that the
    matrix-vector products and their accumulation into $\bm{v}$ are performed at
    precision $u_q\leq u_s$. Then, the mixed-precision computation of $\bm{v}$
    using the precisions described in List 1 yields instead $\hat{\bm{v}}$ satisfying
    \begin{align}
        |\hat{\bm{v}} - \bm{v}| \lesssim (u_s + (n_d +
        n_q)u_q)\sum_{s=1}^{n_d}|B_{s}||\bm{r}_s| + u_p\Theta^v + c(d)u_g\Gamma^v + O(u_s^2 + (u_q + u_g + u_p)^2),
    \end{align}
    where $\Theta^v,\Gamma^v\in\mathbb{R}^{n_\phi}$ are vectors given by
    \begin{align}
        \Theta^v = \sum_{s=1}^{n_d}(|B_{s}|\Theta^R_{s} + \Theta^B_{s}|\bm{r}_s|),\quad
        \Gamma^v = \sum_{s=1}^{n_d}|B_s|\Gamma^R_{s}.
    \end{align}
    Here $\Theta^R_s$ and $\Gamma^R_s$ denote the $s$-th column of $\Theta^R$
    and $\Gamma^R$ respectively.
\end{theorem}

\begin{proof}
    The proof of both theorems is essentially the same since the forward bound for matrix-matrix products in \eqref{eq:matmat_higham} also holds for matrix-vector products. Therefore, we only prove Theorem \ref{th:A_MP_eval}. We have that
    \begin{align}
        \hat{A} = \sum_{s=1}^{n_d}\sum_{t=1}^{n_d}\left((B_s + \Delta B_s)(H_{st} + \Delta H_{st}) + \Delta A_{st}\right)\left(1 + \theta_{n_d^2}^q\right),
    \end{align}
    where $|\Delta B_s|$ satisfies \eqref{eq:bounds_for_B} and $|\Delta H_{st}|$
    satisfies \eqref{eq:bound_for_E}. Here, $\theta_{n_d^2}^q$ is the error due
    to the summations, bounded according to $|\theta_{n_d^2}^q|\lesssim n_d^2
    u_q$, and $\Delta A_{st}$ is the error from the matrix-matrix product which
    satisfies the bound (cf.~equation \eqref{eq:matmat_higham})
    \begin{align}
        |\Delta A_{st}|\lesssim n_qu_q|B_s + \Delta B_s||H_{st} + \Delta H_{st}| = n_qu_q|B_s||H_{st}| + O(u_q(u_s + u_p + u_g)).
    \end{align}
    Therefore
    \begin{align}
        |\hat{A} - A| \lesssim \sum_{s=1}^{n_d}\sum_{t=1}^{n_d}(|B_s||\Delta H_{st}| + |\Delta B_s||H_{st}| + (n_d^2u_q + n_qu_q)|B_s||H_{st}|) + O(u_s^2 + (u_q + u_g + u_p)^2).
    \end{align}
    Replacing $|\Delta B_s|$ and $|\Delta H_{st}|$ with their respective upper
    bounds (cf.~\eqref{eq:bounds_for_B} and \eqref{eq:bound_for_E}) yields the
    thesis. To prove Theorem \ref{th:v_MP_eval}, the same proof follows with the
    following modifications: Replace $n_d^2$ with $n_d$ (since there are only
    $n_d$ terms in the sum), $H_{st}$ with $\bm{r}_s$, and invoke Theorem
    \ref{th:s_MP_eval} in replacement of equation \eqref{eq:bound_for_E}.
\end{proof}

Since it may be difficult to keep track of all error terms appearing in Theorems
\ref{th:A_MP_eval} and \ref{th:v_MP_eval}, for the sake of clarity we condense and simplify the main results in
Corollary \ref{coroll:simplified_full_local_bound}. We first define the
following condition numbers:

\begin{align}
    \label{eq:_coroll_simplebounds_conds_q}
    \kappa^A_q &=  \dfrac{\left\lVert
    \sum_{s=1}^{n_d}\sum_{t=1}^{n_d}|B_{s}||H_{st}|\right\rVert_{\max}}{\lVert
    A \rVert_{\max}},\quad\ \ 
    \kappa^{\bm{v}}_q=\dfrac{\left\lVert
    \sum_{s=1}^{n_d}|B_{s}||\bm{r}_s|\right\rVert_{\infty}}{\lVert
\bm{v}\rVert_{\infty}},\\[0.5em]
            \kappa^A_p &= \frac{\lVert \Theta^A \rVert_{\max}}{\lVert A
    \rVert_{\max}},\quad \kappa^v_p = \frac{\lVert \Theta^v
    \rVert_{\infty}}{\lVert \bm{v}
    \rVert_{\infty}},\quad\quad
    \kappa^A_g = \frac{\lVert \Gamma^A \rVert_{\max}}{\lVert A
    \rVert_{\max}},\quad \kappa^v_g = \frac{\lVert \Gamma^v
    \rVert_{\infty}}{\lVert \bm{v}
    \rVert_{\infty}},
\end{align}
whose meaning is described as follows:
\begin{itemize}

    \item $\kappa^A_q$ and $\kappa^{\bm{v}}_q$, are the condition
        numbers in the max and infinity norm of the sum and accumulation of the
        products into $A$ and $\bm{v}$ respectively.

    \item $\kappa^A_p$ and $\kappa^{\bm{v}}_p$ represent the
        conditioning of the kernels with respect to the evaluation of the finite
        element (basis) functions and their derivatives.

    \item $\kappa^A_g$ and $\kappa^{\bm{v}}_g$ represent the conditioning of the
        kernels with respect to geometry computations.

\end{itemize}
We can now state the corollary:
\begin{corollary}
    \label{coroll:simplified_full_local_bound}
    Let the assumptions of Theorems \ref{th:A_MP_eval} and \ref{th:v_MP_eval} be
    satisfied. The mass and Poisson kernels evaluated in mixed-precision according to the
    precisions in List 1 yield the local matrix and vector
    $\hat{A}$ and $\hat{\bm{v}}$ satisfying the following error bounds:
    \begin{align}
        \label{eq:coroll_simplebounds}
        \lVert\hat{A} - A\rVert_{\max} &\lesssim \left((u_s + n_q
        u_q)\kappa_{q}^A + u_p\kappa_{p}^A + u_g\kappa_{g}^A\right)\lVert A \rVert_{\max},\\
        \lVert\hat{\bm{v}} - \bm{v}\rVert_{\infty} &\lesssim \left((u_s + n_q
        u_q)\kappa_{q}^{\bm{v}} + u_p\kappa_p^{\bm{v}} +
    u_g\kappa_g^{\bm{v}}\right)\lVert \bm{v} \rVert_{\infty}.
    \end{align}
\end{corollary}

\begin{proof}
    Apply the max and infinity norm to the bounds in Theorems \ref{th:A_MP_eval}
    and \ref{th:v_MP_eval} and multiply and divide by $\lVert A \rVert_{\max}$
    and $\lVert \bm{v} \rVert_{\infty}$ respectively.
\end{proof}

\begin{remark}
    \label{rem:simplified_full_local_bound_1}
    Recalling previous results from this section (Theorems \ref{th:C_MP_eval},
    \ref{th:s_MP_eval}, \ref{th:A_MP_eval}, and \ref{th:v_MP_eval}), we note
    that the condition numbers $\kappa^A_g$ and $\kappa^{\bm{v}}_g$ are both
    proportional to the condition number of the Jacobian of the reference map:
    \begin{align}
        \kappa^A_g, \kappa^{\bm{v}}_g \propto \max\limits_{q}\kappa_2(J(\X_q)).
    \end{align}
    Thus, ill-conditioned Jacobians, e.g., due to knife elements, may lead to
    loss of accuracy.
\end{remark}

\begin{remark}
    \label{rem:simplified_full_local_bound_2}
    Again, backtracking the values of the condition numbers $\kappa^A_p$ and
    $\kappa^{\bm{v}}_p$ through Theorems \ref{th:C_MP_eval}, \ref{th:s_MP_eval},
    \ref{th:A_MP_eval}, and \ref{th:v_MP_eval}, we note that $\kappa^A_p$ and
    $\kappa^{\bm{v}}_p$ are proportional to $p^d\kappa(V_{\K})$ for both forms.
    For the Poisson form it also holds that
    \begin{align}
        \kappa^A_p &\propto p^d\kappa(V_{\K}) +
        dp\max_k\kappa(\partial_{k}\bm{\Phi}),\\
        \kappa^{\bm{v}}_p &\propto p^d\kappa(V_{\K}) +
        dp\max_k\kappa(\partial_{k}\bm{\Phi}) +
        \max_k\left(\kappa(\partial_{k}\bm{\Phi})
        \kappa(\partial_{k}V_{\K})\right).
    \end{align}
    Poorly conditioned reference basis functions and high-degree polynomials may
    thus lead to loss of accuracy.
\end{remark}

\subsubsection{Global assembly error bounds}
\label{subsubsec:assembly_error_bounds}

We now extend the error bounds in Theorems \ref{th:A_MP_eval} and
\ref{th:v_MP_eval} to a rounding error bound for global assembly:
\begin{corollary}
    \label{coroll:global_assembly}
    Let the assumptions of Theorems \ref{th:A_MP_eval} and \ref{th:v_MP_eval} be
    satisfied for all cells in $D_h$ for local cell matrices and vectors
    respectively. Let $D$ and $\bm{b}$ (defined in
    \eqref{eq:assembly_matrixvec_def}) be computed in mixed-precision according
    to the precisions listed in List 1. Let all cell contributions be assembled
    into global tensors $A^g=LDL^T$, $\bm{v}^g=Lb$ at precision $u_q$. If $L$,
    defined in \eqref{eq:assembly_matrixvec_def}, has at most $m_K$ nonzero
    entries\footnote{Another equivalent definition of $m_k$: it is the maximum
    number of local degrees of freedom that correspond to the same global degree
of freedom.} per row, it holds that
    \begin{align}
        |\hat{A}^g - A^g| &\leq L|\hat{D} - D|L^T + 2cm_Ku_q L|D|L^T + O(u_q(u_s + u_q + u_g + u_p)),\\
        |\hat{\bm{v}}^g - \bm{v}^g| &\leq L|\hat{\bm{b}} - \bm{b}| + cm_Ku_qL|\bm{b}|+ O(u_q(u_s + u_q + u_g + u_p)).
    \end{align}
\end{corollary}
\begin{proof}
    We only prove the result for $A^g$ since the proof for $\bm{v}^g$ is
    essentially the same. To leading order, we can write
    \begin{align}
        \hat{A}^g = L(\hat{D} - D)L^T + \widehat{LDL^T},
    \end{align}
    where the last term on the right-hand side only accounts for the error in
    the matrix-matrix products and not in the computation of $D$. We can then
    invoke equation \eqref{eq:matmat_higham} twice to bound the two
    matrix-matrix products. Doing so yields
    \begin{align}
        \widehat{LDL^T} = LDL^T + \Delta(LDL^T),\quad\text{where}\quad|\Delta(LDL^T)|\leq 2\gamma_{m_K}L|D|L^T + O(u_q^2).
    \end{align}
    The above bound is due to the fact that $L$ is boolean and sparse with at
    most $m_K$ ones per row and thus: 1) $L=|L|$, 2) $L$ is representable
    exactly, 3) The action of $L$ from the left and of $L^T$ from the right
    against $D$ yields a matrix whose entries are the result of a sum of no
    more than $m_K$ terms. We thus have, to leading order,
    \begin{align}
        |\hat{A}^g - A^g| &\leq L|\hat{D} - D|L^T + 2cm_Ku_qL|D|L^T,
    \end{align}
    which is the thesis. The result for $\bm{v}^g$ follows the same exact
    strategy, but comes without the factor $2$ since no multiplication by $L^T$
    is performed.
\end{proof}
The above corollary implies that the error in the global assembly is roughly
proportional to the sum of the errors of the local cell tensors. We make this
proportionality clearer with an analogous of Corollary
\ref{coroll:simplified_full_local_bound}, which shows that the same bound as for
the local cell tensors holds, albeit with $m_K^2$ and $m_K$ times larger
factors for bilinear and linear forms respectively:
\begin{corollary}
    \label{coroll:assembly_simplified_full_local_bound}
    Let the assumptions of Corollary \ref{coroll:global_assembly} hold.
    The global assembly of mass and Poisson kernels evaluated in mixed-precision according to the
    precisions in List 1 yield the global matrix and vector
    $\hat{A}^g$ and $\hat{\bm{v}}^g$ satisfying:
    \begin{align}
        \label{eq:coroll_simplebounds_global}
        \lVert\hat{A}^g - A^g\rVert_{\max} &\lesssim m_K^2\left((u_s + n_q
        u_q)\bar{\kappa}_{q}^A + u_p\bar{\kappa}_{p}^A + u_g\bar{\kappa}_{g}^A\right)\lVert A^g \rVert_{\max},\\
        \lVert\hat{\bm{v}}^g - \bm{v}^g\rVert_{\infty} &\lesssim m_K\left((u_s + n_q
            u_q)\bar{\kappa}_{q}^{\bm{v}} + u_p\bar{\kappa}_p^{\bm{v}} +
        u_g\bar{\kappa}_g^{\bm{v}}\right)\lVert \bm{v}^g \rVert_{\infty}.
    \end{align}
    Here $\bar{\kappa}^A_t\lVert A^g \rVert_{\max}=\max_k(\kappa^{A}_t\lVert A^k
    \rVert_{\max})$ and $\bar{\kappa}^v_t\lVert \bm{v}^g
    \rVert_{\infty}=\max_k(\kappa^{v}_t\lVert \bm{v}^k
    \rVert_{\infty})$ for $t\in\{q,p,g\}$.
\end{corollary}

\begin{proof}
    We ignore the $O(u_q)$ terms in Corollary \ref{coroll:global_assembly} since
    they are of negligible size. We thus bound, ignoring higher-order terms,
    \begin{align}
        \lVert\hat{A}^g - A^g\rVert_{\max} &\leq \lVert L|\hat{D}-D|L^T
        \rVert_{\max} \leq m_K^2\lVert \hat{D} - D \rVert_{\max}\\
        \lVert\hat{\bm{v}}^g - \bm{v}^g\rVert_{\infty} &\leq \lVert L|\hat{\bm{b}} -
        \bm{b}| \rVert_{\infty} \leq m_K\lVert \hat{\bm{b}} - \bm{b} \rVert_{\infty}.
    \end{align}
    Since each entry in $L|\hat{D}-D|L^T$ is a sum of at most $m_K^2$ entries of
    $|\hat{D}-D|$ and each entry of $L|\hat{\bm{b}}-\bm{b}|$ is a sum of at most
    $m_K$ entries of $|\hat{\bm{b}}-\bm{b}|$. The thesis then readily follows
    by applying Corollary \ref{coroll:simplified_full_local_bound} to the
    above, taking the max across all cells, and multiplying and dividing by
    $\lVert A^g \rVert_{\max}$ and $\lVert \bm{v}^g \rVert_{\infty}$
    respectively.
\end{proof}

\subsection{Mixed-precision strategy and hardware accelerators}
\label{subsec:MP_strategy}

We have derived a fine-grained error bound for mixed-precision kernel and
assembly computations which allows for different precision choices. In this
subsection we make our exposition simpler and more specific by restricting the
precision choice to two formats: a low-precision format (e.g., single or half
precision) $u$ and a higher-precision format (e.g., single or double) $u^2$.
The objective here is to choose precisions so as to obtain accurate enough
computations while keeping costs contained. 

\begin{remark}
    From now on, we focus on 3D high-polynomial degree implementations ($p\geq
    3$). The reason is that in low dimensions and for low-degree polynomials all
    kernel computations have roughly the same cost and the different error terms
    are more or less of comparable size: a mixed-precision implementation of
    Algorithms 1 and 2 would thus lead to little practical benefits.
    Furthermore, low-dimension/low-degree kernels are extremely cheap already.
    We leave the investigation of performant mixed-precision implementations of
    low-dimension/low-degree kernels to future work.
\end{remark}

\noindent The first decision we make is the following:

\begin{itemize}
    \item \textbf{Store tensors in reduced precision.} Set $u_s=u$.  \textit{Reason:}
    the error term associated with storage has the smallest constant. Furthermore,
    reduced-precision storage leads to memory savings and better cache usage. 
\end{itemize}

\begin{figure}[!htbp]
    \centering
    \includegraphics[width=0.8\textwidth]{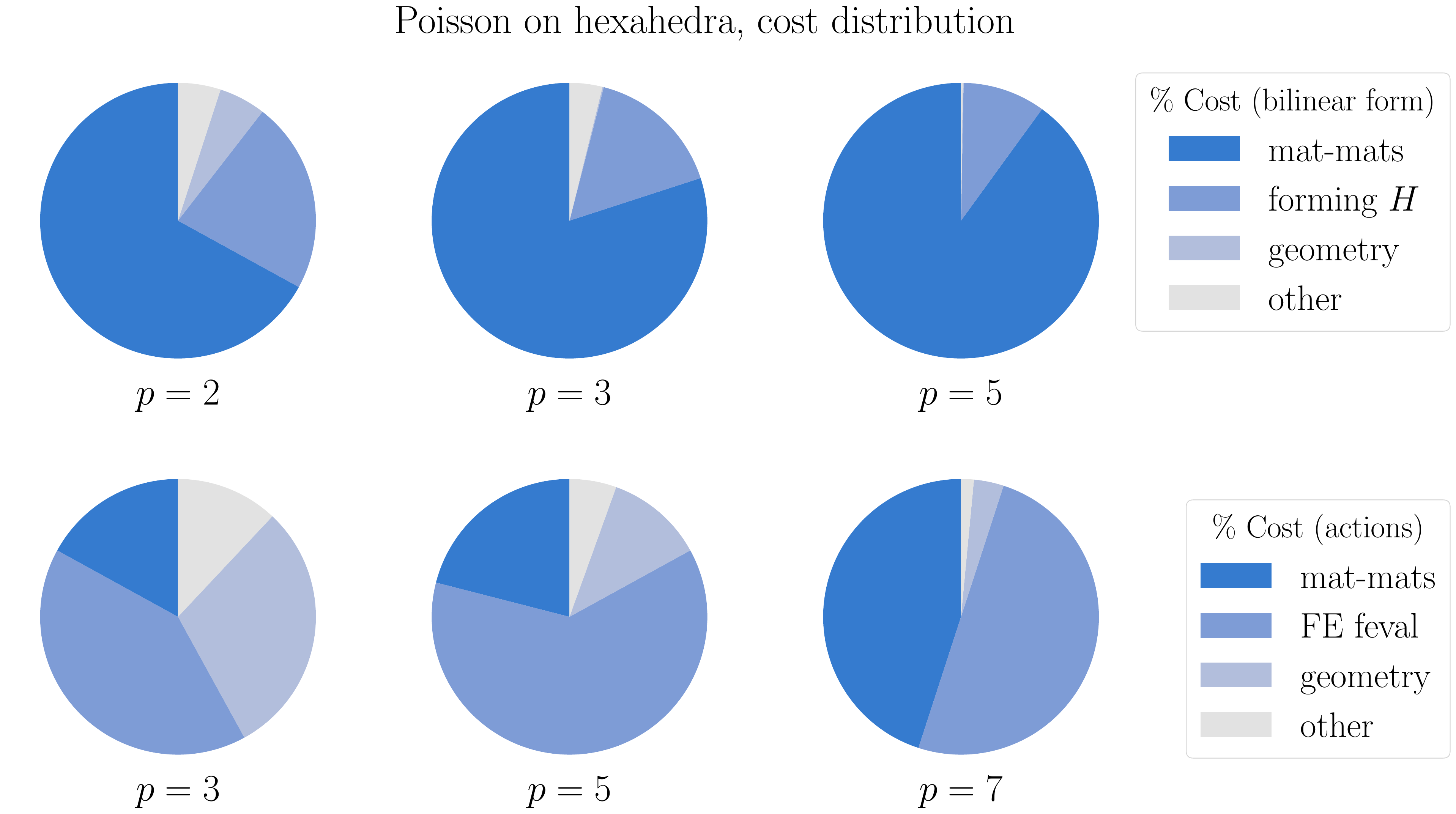}
    \caption{\textit{Pie charts of the computational cost distribution of
            different Poisson kernel subroutines on hexahedra. The results,
            entirely run in double precision, are shown for different polynomial
            degrees $p$. The top and bottom rows correspond to the Poisson
 bilinear form and action respectively. By ``mat-mats'' we indicate the cost
of accumulating the matrix-matrix products into $A$ and $\bm{v}$ and by ``FE
feval'' we denote the evaluation of $\nabla \check{w}(\X)$.}}
    \label{fig:pie_charts}
\end{figure}

To guide our next precision choices, we first run the kernels in double
precision and time each computational component. Taking the Poisson kernels on
hexahedra ($z(\bm{x})\equiv 1$, cell batch size of $16$ for actions) as a
representative example, we show these results in Figure \ref{fig:pie_charts}.
The cost distribution for the mass kernel and tetrahedra is similar (not shown).
For bilinear forms the bottleneck is always the accumulation of matrix-matrix
products into $A$. On the other hand, the evaluation of FE functions accounts
for the majority of the total cost of action kernels for low to medium $p$,
while for large $p$ the matrix-matrix product accumulation and the FE function
evaluations are both bottlenecks. 

A sensible criterion for designing a mixed-precision implementation is to employ
the high-precision format for all computations that are either 1) of negligible
cost or that 2) lead to the largest rounding errors. These considerations lead
to the following decisions:

\begin{itemize}
    \item \textbf{Use high-precision for geometry computations.} Set $u_g=u^2$.
        \textit{Reason:} Geometry computations are of negligible cost for
        moderate-to-high degree elements (cf.~Figure \ref{fig:pie_charts}), so a
        higher-precision implementation does not affect runtime. Furthermore,
        computing geometry tensors more accurately safeguards kernels from knife
        elements and ill-conditioned Jacobians.

    \item \textbf{Use high-precision for polynomial evaluations.} Use precision
        $u^2$ for tabulation. \textit{Reason:} FE basis functions are tabulated only
        once and for all in Algorithms 1 and 2 and thus the cost of this
        operation is typically negligible. Furthermore, the memory access costs
        of tabulating high-degree polynomials are minimized thanks to the
        reduced-precision storage.
\end{itemize}

There are still two precisions to be chosen: The precision of FE function
evaluations and the precision used for performing and accumulating the
matrix-matrix products into the local tensors $A$ and $\bm{v}$. The issue is
that these operations are at the same tame the computational bottlenecks of the
kernels (cf.~Figure \ref{fig:pie_charts}) and possibly the dominating sources of
rounding errors. Indeed, if we recall Corollary
\ref{coroll:simplified_full_local_bound} and Remarks
\ref{rem:simplified_full_local_bound_1} and
\ref{rem:simplified_full_local_bound_2}, these terms lead to rounding errors
growing like $p^du_p$ and $n_qu_q$ respectively and thus rapidly grow with the
polynomial degree. Fortunately, there is a way of obtaining high-accuracy
computations at a lower cost:

\begin{itemize}
    \item \textbf{Use mixed-precision accelerators for FE function evaluations
        and matrix-matrix products.} Set $u_p=u^2$ and $u_q=u^2$.
        \textit{Reason:} Mixed-precision vector and tensor accelerators perform
        high-precision vector and tensor operations between operands stored at
        low precisions. The resulting computational cost of these operations is
        typically lower than the corresponding fully high-precision
        computations.
\end{itemize}

\begin{remark}
    In some architectures the most performant (if not the only)
    tensor cores available may be mixed-precision cores \cite{fasi2021numerical}
    and thus the strategy of storing tensors in low precision and multiplying
    them in high (or mixed) precision may be the only option for exploiting
    these accelerators.
\end{remark}

The above precision selections lead to a mixed-precision kernel implementation
with mixed-precision accelerators that satisfies an error bound in the form
(cf.~Corollary \ref{coroll:simplified_full_local_bound}):
\begin{align}
    \label{eq:MP_simplified_final_local_bound}
    \text{relative error} \lesssim u\kappa_q + O(u^2),
\end{align}
where $\kappa_q$ stands for either $\kappa_q^A$ or $\kappa_q^v$. To leading
order, the above bound is thus independent from the conditioning of polynomial
evaluations and geometry computations. Furthermore, the bound is also
independent\footnote{Excluding the condition number $\kappa_q$ which may or may
    not hide some $p$ or $n_q$ dependency. In our numerical experiments
(cf.~Section \ref{sec:numerical_results}) it seems like it does not, at least
for the mass and Poisson forms.} from the polynomial degree and the number of
quadrature nodes. The relative error bound for the assembly is $m_K^2$ times
larger for bilinear forms and $m_K$ times larger for linear forms (actions),
cf.~\ref{coroll:assembly_simplified_full_local_bound}.

\begin{remark}
    We remark that using mixed-precision accelerators for FE function
    evaluations leads to a slight departure from our theory. Indeed, for the
    sake of simplicity we have not accounted for the error arising from rounding
    the function coefficients and the tabulated basis function values to a lower
    precision before performing the evaluation. This casting results in an
    additional error term in the bounds of Theorems \ref{th:A_MP_eval} and
    \ref{th:v_MP_eval} which is proportional to the conditioning of the
    polynomial basis:
    \begin{align}
        \propto u_s\kappa(V_{\K}),\quad&\text{for bilinear forms with coefficients and
        mass actions},\\
        \propto u_s(\kappa(V_{\K}) + \max_s\kappa(\partial_sV_{\K})),\quad&\text{for
        Poisson actions}.
    \end{align}
    The constants in these error terms are much smaller than in the polynomial
    evaluation terms in Theorems \ref{th:A_MP_eval} and \ref{th:v_MP_eval} since
    $p^d$ does not appear as a multiplicative factor.  Nevertheless, one must
    double check that the conditioning of the FE basis does not affect results:
    while $\kappa(V_{\K})$ is typically small provided that the basis is
    constructed accordingly \cite{isaac2020recursive},
    $\max_s\kappa(\partial_sV_{\K})$ could potentially be large, especially on
    tetrahedra. Nevertheless, we have not found this to be an issue in our
    numerical experiments (cf.~Figure \ref{fig:figure_errors_Poisson_both},
    bottom). In any case, the bounds in Corollaries
    \ref{coroll:simplified_full_local_bound} and
    \ref{coroll:assembly_simplified_full_local_bound} and in equation
    \eqref{eq:MP_simplified_final_local_bound} still hold as-is albeit with a
    larger condition number multiplying the $u_s$ term.
\end{remark}

\begin{remark}
    The multiplication of large matrices and vectors with tensor and vector
    accelerators must be blocked, i.e., decomposed into smaller matrix products
    so that each submatrix and subvector can fit within the accelerator
    registers. Blocking is beneficial from a rounding error viewpoint as it
    leads to smaller error constants
    \cite{blanchard2020class,blanchard2020mixed}. Nevertheless, we have ignored
    blocking effects in our analysis since these only affect the constants of
    the negligible $O(u^2)$ terms.
\end{remark}

\section{Numerical results}
\label{sec:numerical_results}

In this section we validate the rounding error analysis of Section
\ref{sec:rounding_error_analysis} and we test the efficiency and accuracy of
hardware accelerated mixed-precision kernel implementations.  The aim is to
establish the computational superiority of kernels exploiting matrix product
accelerators and to compare their accuracy with kernels implemented in different
precisions.  Specifically, we will show that AMX-accelerated kernels are up to
$60$ times faster than their double precision equivalents and up to hundreds of
times more accurate than fully half-precision kernels.

Here we do not investigate the rounding error behaviour of polynomial and FE
function evaluations for the sake of brevity: The evaluation of polynomials is a
classical subject in rounding error analysis and its rounding error behaviour
has been thoroughly documented (see e.g., Section 5 in \cite{higham2002accuracy} and
\cite{higham2004numerical,pena2000multivariate}). Similarly, the evaluation of
FE functions consists of an inner product (between coefficients and
basis functions) whose rounding error analysis is also classical,
see e.g., Section 3.1 in \cite{higham2002accuracy}.

\subsection{Experimental setup.}
\label{subsec:experimental_setup}

All experiments were performed on a single core of an Intel Xeon Platinum 8480+
processor (code name ``Sapphire Rapids'') which supports mixed-precision
single/bf16 precision vector instructions (AVX512-bf16) as well as matrix
multiplication instructions (AMX-bf16).  However, we found the overall bf16
intrinsics and compiler support to be limited. Circumventing these limitations
warranted specific implementation choices which we describe in Appendix
\ref{appendix_sec:implementation_details}.

We employed the FEniCSx
open-source FE software library \cite{BarattaEtal2023} (Basix in particular
\cite{BasixJoss,ScroggsEtal2022}) to tabulate the FE basis functions and to
compute the quadrature rules offline and in double precision. All kernels were
implemented in C++ in the C++23 standard and compiled using \texttt{g++} (GCC
version 14.2.0) with the compiler flags
\begin{center}
    \texttt{-std=c++23 -march=sapphirerapids -O3 -mprefer-vector-width=512}.
\end{center}
All kernels were validated via comparison with the FEniCSx built-in double
precision kernels which was also used to compute rounding errors. The
open-source code implementing the experiments is freely available on
GitHub\footnote{See
\url{https://github.com/croci/mpfem-paper-experiments-2024}.}.  For our
numerical computations we exclusively work in three dimensions and employ an
unstructured hexahedral mesh of a tetrahedron of coordinates
\begin{align*}
    \{(-5, 0, -5),\ (5, 0, -5), (0,-5,5), (0,5,5)\},
\end{align*}
and an unstructured tetrahedral mesh of the cube $[0,5]^d$. Both meshes were
generated using the open source mesh-generation software Gmsh
\cite{geuzaine2009gmsh} and have $n_K=55296$ and $n_K=63408$ hexahedral and
tetrahedral cells respectively. In order to keep the overall computational costs
contained even at high-degree polynomial degrees we only run the kernels on the
first $\tilde{n}_K$ cells, where $\tilde{n}_K$ is given by the formula
\begin{align*}
    \tilde{n}_K = n_{\text{batch}}\lfloor \min(2\times 10^6/n_\phi, n_K)/n_{\text{batch}}
    \rfloor.
\end{align*}
Here, we recall, $n_{K}$ is the total number of cells in the mesh, $n_\phi$ is
the number of cell degrees of freedom and $n_\text{batch}$ is the number of
cells batched together in action kernels. The number $2\times 10^6$ is a number
large enough to ensure cache spillover to the RAM. This choice prevents the CPU
from keeping all data in cache and thus ensures a fair comparison among kernels.
The division, floor, and multiplication by $n_{\text{batch}}$ is made for
convenience to ensure that the number of cells is a multiple of $n_{\text{batch}}$.
In our computations we use $n_{\text{batch}}=64$ which we found gives good
performance. With this choice, the above formula yields $\tilde{n}_K = n_K$ only
for $p=2$ for the hexahedral mesh and for $p\in\{2,3\}$ for the tetrahedral
mesh.  In our experiments we only consider polynomial degrees larger than $1$
and up to degree $10$ for tetrahedra and up to degree $7$ for hexahedra. These
upper bounds were arbitrarily chosen to keep the overall computational
time of our single-core experiments contained. As reference cells we take the
unit cube $[0,1]^3$ and the tetrahedron with vertices
\begin{align}
    (0,0,0),\quad (1,0,0),\quad (0,1,0),\quad (0,0,1).
\end{align}
We set $\check{\bm{z}}(\x)\equiv 1$ for all experiments since the effect of FE
function evaluations can already be observed in the action kernels.

\subsection{Geometry errors.}

We begin by analyzing the rounding error behaviour of geometry computations.
Our aim to establish the linear growth of the error with respect to
$\kappa_2(J)$ and thus validate the bounds of Theorem
\ref{th:poisson_geometry}. We only consider the Poisson geometry for the
sake of brevity and since results are similar. We start from a tetrahedron of coordinates
\begin{align}
    (1, 1, 0),\quad (1, 0, 1),\quad (\epsilon, 1, -1),\quad (0, 0, 0),
\end{align}
where $\epsilon>0$ is a constant used to tweak the Jacobian condition number. In
fact, the resulting Jacobian is 
\begin{align}
    \left[
    \begin{array}{ccr}
        1 & 1 & 0\\
        1 & 0 & 1\\
        \epsilon & 1 & -1
    \end{array}\right],
\end{align}
which is singular for $\epsilon=0$, has $\textnormal{det}(J)=\epsilon$ and
$\kappa_{\infty}(J)=(2+\epsilon)(3+\epsilon)/\epsilon$ (and therefore
$\kappa_2(J)\propto \frac{1}{\epsilon}$). We then scale the coordinates by
$3^{-1/2}$ and translate them by the translation vector $(1/3, 2/3, 4/3)$ so
that the transformed cell coordinates and Jacobian stop being exactly representable (this
transformation leaves the condition number unchanged). We monitor the 
relative rounding error in the max norm
\begin{align}
    \dfrac{\lVert G - \hat{G}\rVert_{\max}}{u \lVert G \rVert_{\max}},
\end{align}
as a function of the condition number $\kappa_2(J)$, where $\hat{G}$ is computed
using exclusively fp32 single precision arithmetic. Results are shown in Figure
\ref{fig:geometry_errors} which match what predicted by the theory. We remark
that the powers of ten on the $y$ axis indicate the number of digits lost in the
computation. For single precision we have roughly $7.5$ digits of accuracy and
the proportionality constant seems to be around $0.1$ for this cell so geometry
computations are safe in this case up to condition numbers of roughly
$10^{8.5}$. While we expect similar results to hold for fp16 half precision
arithmetic we report that the computation underflows for this precision (results
not shown).  We nevertheless predict that fp16 would lead to complete loss of accuracy much
earlier for condition numbers of only $10^{4.5}$. Mitigation of range issues via
scaling is likely possible, but outside the scope of this work (if the CPU
supported it here we would have just used bf16 which has a wider range).
Nevertheless these results show how using higher precision for geometry computations might be
beneficial, especially when the floating-point format has a narrow range or
a large unit roundoff.

\begin{figure}[!htbp]
    \centering
    \includegraphics[width=0.5\textwidth]{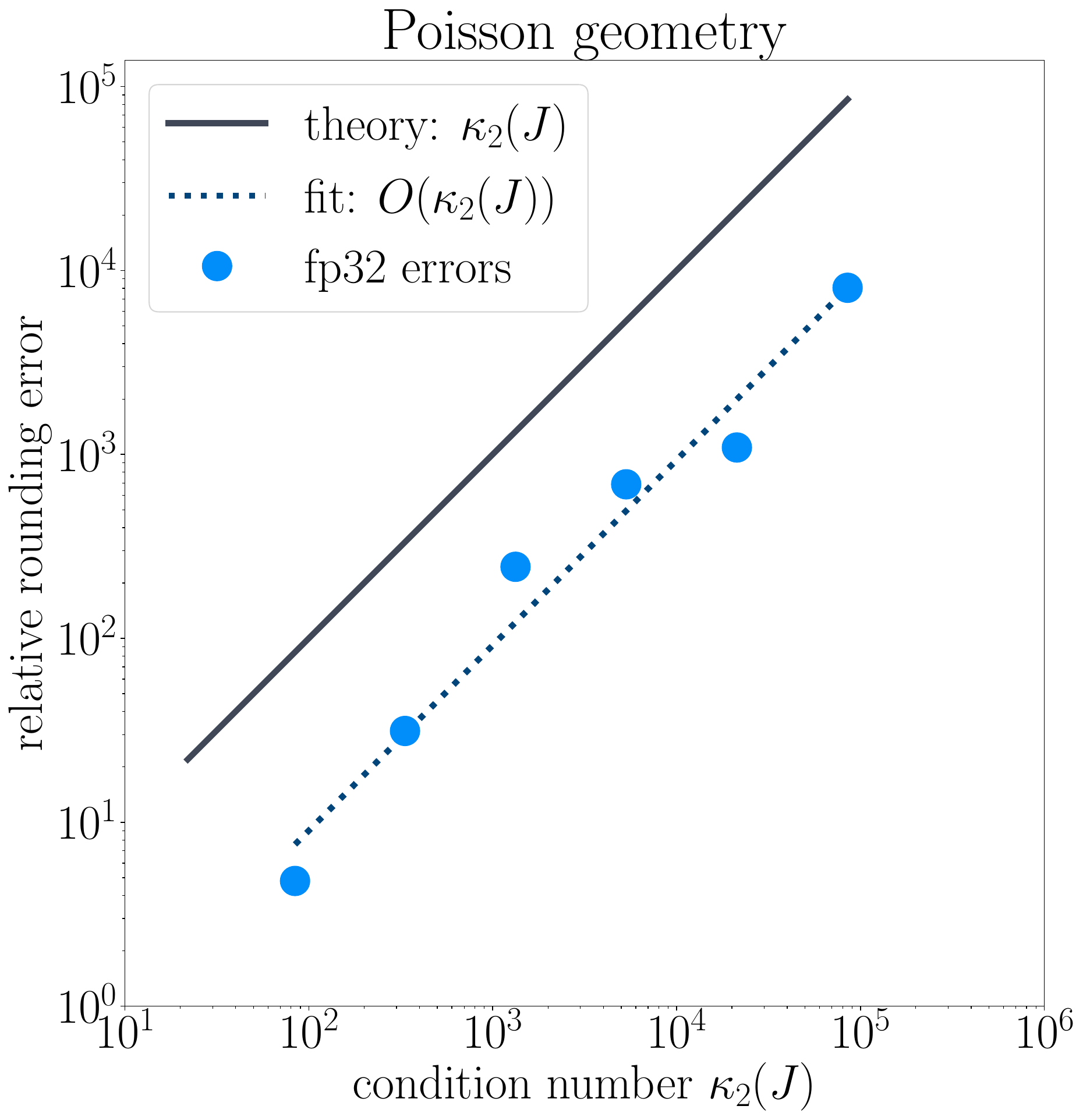}
    \caption{\textit{Relative rounding errors in the fp32 evaluation of the
    Poisson geometry as a function of the condition number of the reference map
    Jacobian. Fitting a line to the errors in the loglog plot yields the
    linear proportionality expected from Theorem \ref{th:poisson_geometry}.}}
    \label{fig:geometry_errors}
\end{figure}

\begin{remark}
    We remark that severely ill-conditioned Jacobians are not necessarily common in
    FE computations and only arise when cells are badly scaled and/or near
    degenerate. Nevertheless, such cells do appear in graded meshes, when boundary
    layers need to be resolved, and in adaptive refinement. When reduced-precision
    computations need to be performed on such meshes we suggest a simple
    mixed-precision strategy: tailor the geometry precision to the cell conditioning
    and selectively employ higher-precision on low-quality cells.
\end{remark}

\begin{figure}[!htbp]
    \centering
    \includegraphics[width=\textwidth]{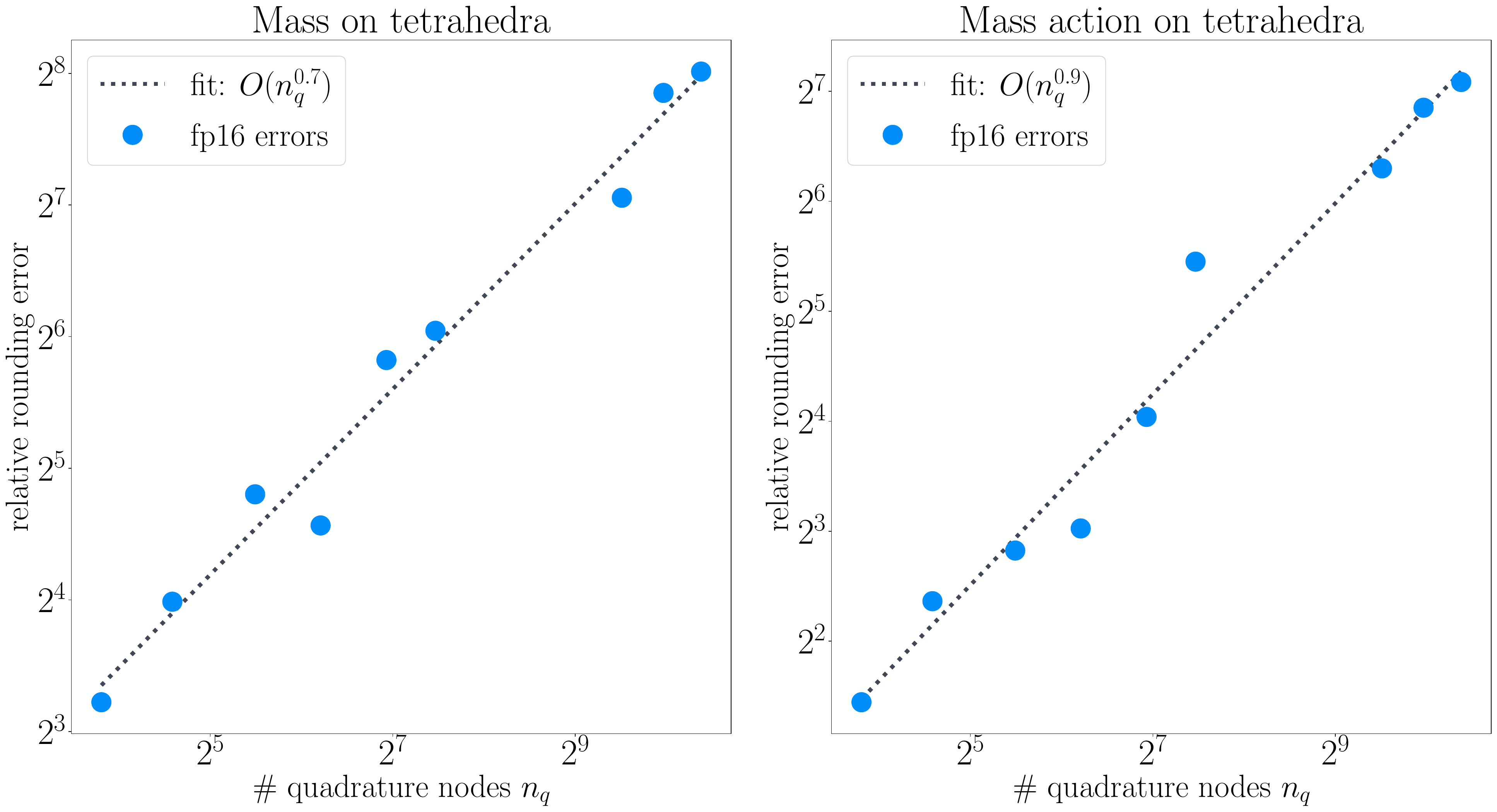}
    \caption{\textit{Relative rounding errors in mass kernel evaluations on
    tetrahedra versus the number of quadrature nodes $n_q$. Here the fp16 format
    is used in place of the precision $u_q$ and single precision is used in
    everything else.}}
    \label{fig:n_q_errors}
\end{figure}

\subsection{Errors in the accumulation of matrix-matrix products into the
local tensors.}

We now focus on the $O(u_q)$ error term in Theorems \ref{th:A_MP_eval} and
\ref{th:v_MP_eval}. Our objective here is to check whether we actually observe an
error growth rate of $O(n_q)$. The $O(n_q)$ error growth is a direct consequence
of \eqref{eq:matmat_higham} and is a worst-case error bound. However, it is
well-known that rounding errors may behave stochastically leading to error
cancellation and milder error growth rates \cite{higham2019new}. For instance,
for matrix-matrix products and rounding errors behaving as independent zero-mean
random variable the rate would be $O(n_q^{1/2})$ \cite{higham2019new}.
Here we plan to determine which rates are observed in practice when
running FE kernels.

For this purpose, we consider both mass kernels (bilinear form and action) on
the tetrahedral mesh (results for hexahedral meshes are similar and thus not
shown) and use single precision for storage, geometry, and polynomial/FE
function evaluations so that $u_s=u_g=u_p=2^{-24}$ and employ fp16 half
precision for the matrix-matrix products so that $u_q=2^{-11}$, cf.~Table
\ref{tab:precision}.  These precision choices lead to a total rounding error
dominated by the $O(u_q)$ term. We then monitor the relative rounding errors
\begin{align}
    \max_{k=1,\dots,\tilde{n}_K}\left(\dfrac{\lVert A^k - \hat{A}^k
    \rVert_{\max}}{u_q\lVert A^k \rVert_{\max}}\right),\quad
    \max_{k=1,\dots,\tilde{n}_K}\left(\dfrac{\lVert \bm{v}^k - \hat{\bm{v}}^k
    \rVert_{\infty}}{u_q\lVert \bm{v}^k \rVert_{\infty}}\right),
\end{align}
as a function of $n_q$. Results are shown in Figure \ref{fig:n_q_errors}. For
mass actions we do observe a rate close to $1$, indicating the expected error
growth. However, for bilinear forms the rate is lower ($0.7$) indicating that
some error cancellation is occurring. Checking the error behaviour on a
cell-by-cell basis leads to varying rates between $0.5$ and $1$ so we can only
conclude that the actual degree of error cancellation and growth rate with
respect to $n_q$ is cell- and form-dependent. This is not surprising, but it
means that the only option when implementing fully low-precision kernels is to
either determine experimentally the growth rate or to plan for the largest
possible error growth.

We remark that the base-ten logarithm of the error indicates the number of
digits lost in the computation. For fp16 this means that all precision is lost
when the error reaches $2^11$ so care must be taken for high polynomial degrees.
For instance, from the left figure in Figure \ref{fig:n_q_errors} we deduce that
at polynomial degrees $p=9,10$ we already have less than $1$ digit of accuracy.
These consideration further motivate the use of mixed-precision accelerators
that are able to efficiently compute the matrix-matrix products in higher precision.

\begin{figure}[!htbp]
    \centering
    \begin{subfigure}[h]{\textwidth}
        \centering
        \includegraphics[width=\textwidth]{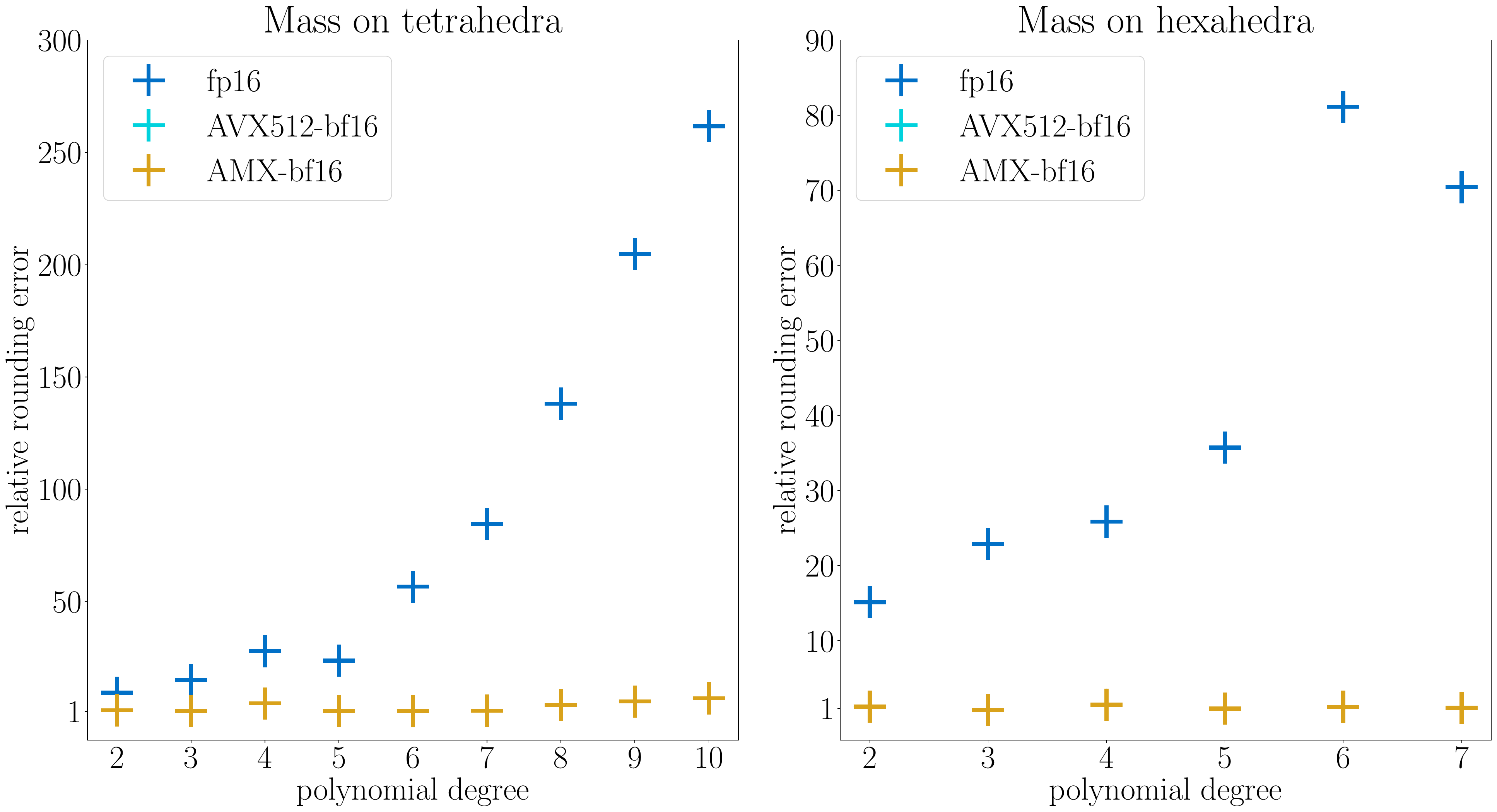}
    \end{subfigure}\\
    \vspace{6pt}
    \begin{subfigure}[h]{\textwidth}
        \centering
        \includegraphics[width=\textwidth]{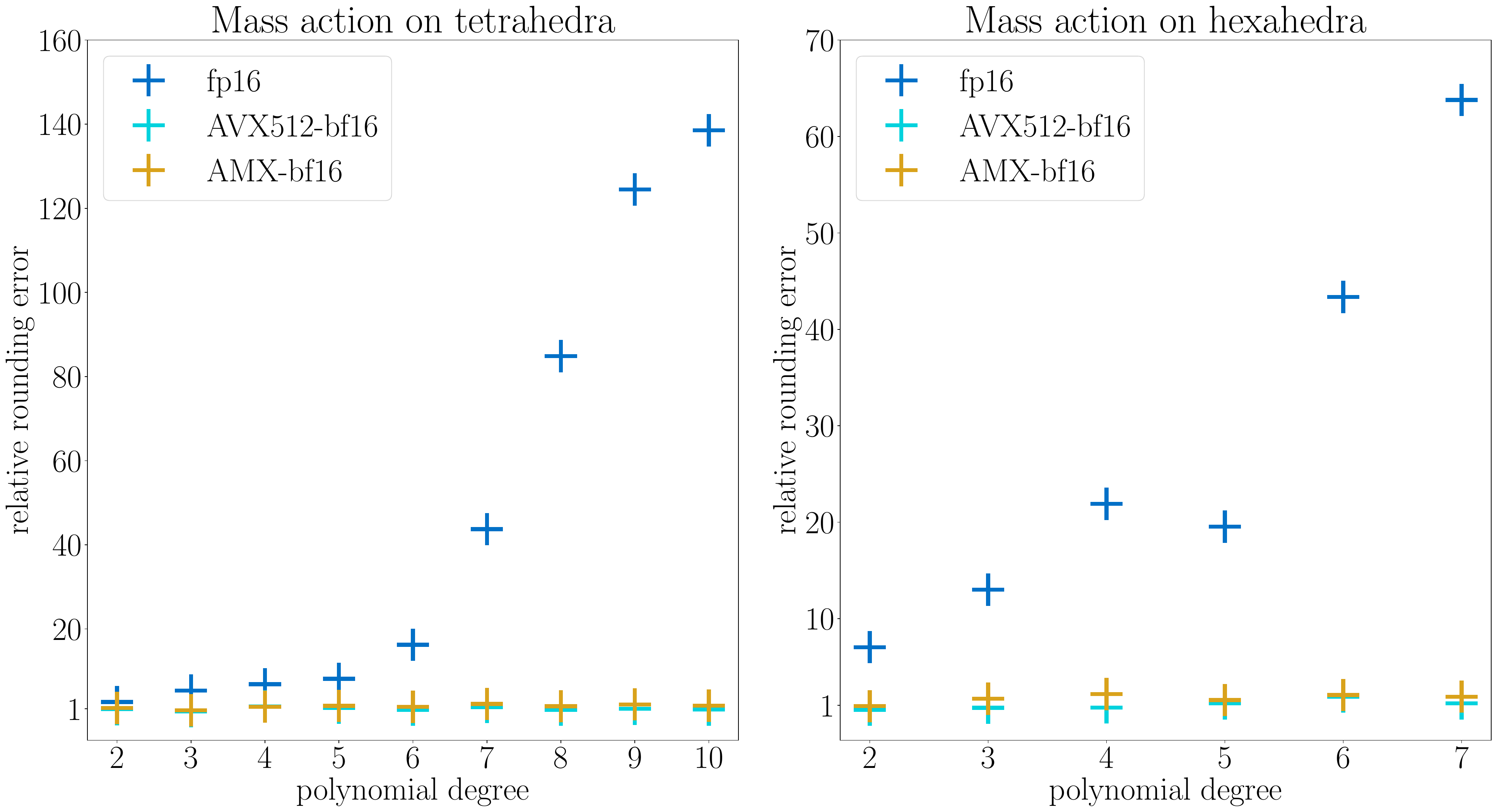}
    \end{subfigure}
    \caption{\textit{Relative rounding errors arising in the mass kernel
            evaluations via fp16 and fp32/bf16 mixed-precision plotted against
            the polynomial degree. The AVX512-bf16 and AMX-bf16 are almost
            overlapping. Note that the fp16 errors grow with $p$ while the
            mixed-precision errors are constant.}}
    \label{fig:figure_errors_mass_both}
\end{figure}

\begin{figure}[!htbp]
    \centering
    \begin{subfigure}[h]{\textwidth}
        \centering
        \includegraphics[width=\textwidth]{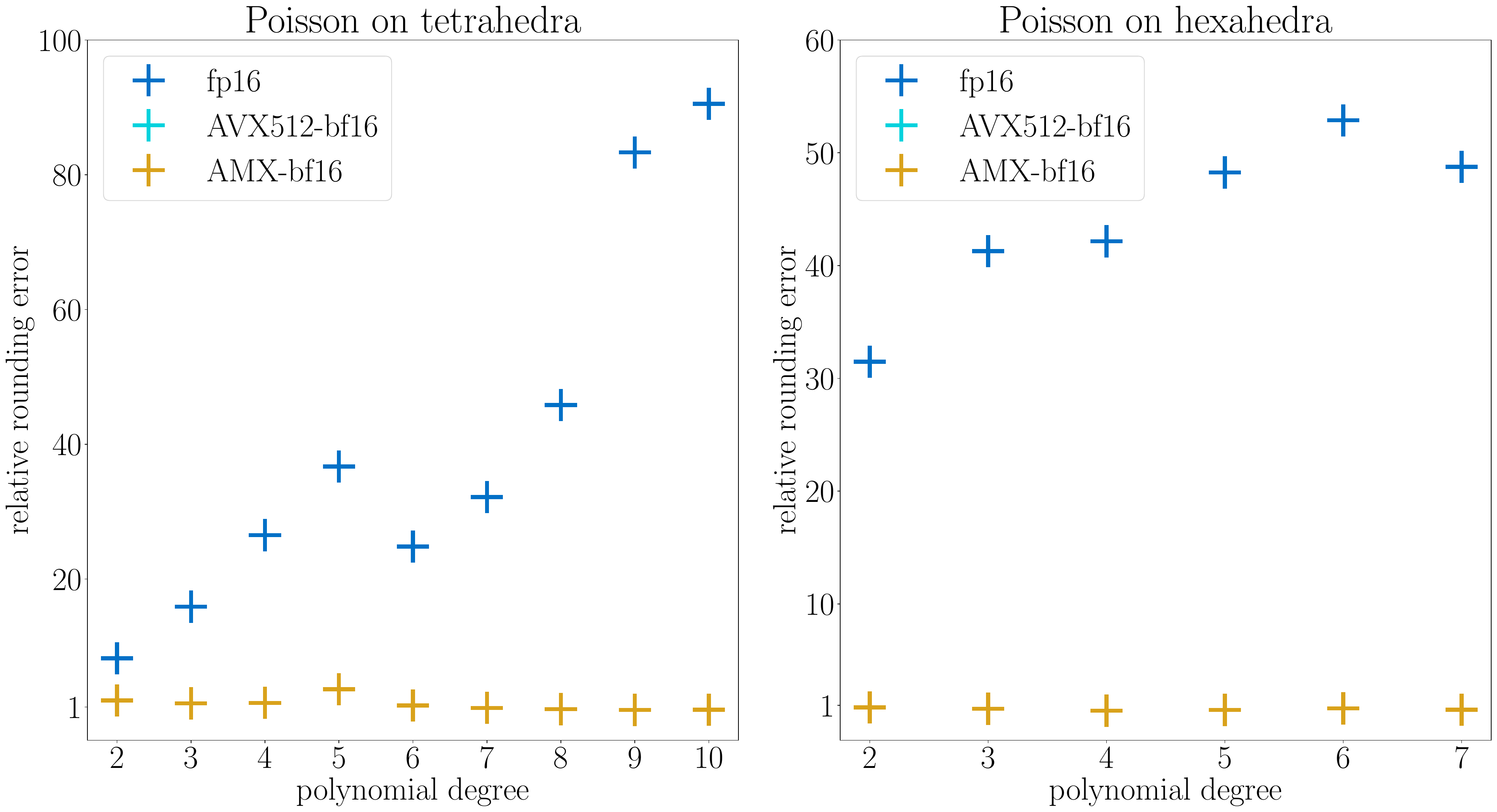}
    \end{subfigure}\\
    \vspace{6pt}
    \begin{subfigure}[h]{\textwidth}
        \centering
        \includegraphics[width=\textwidth]{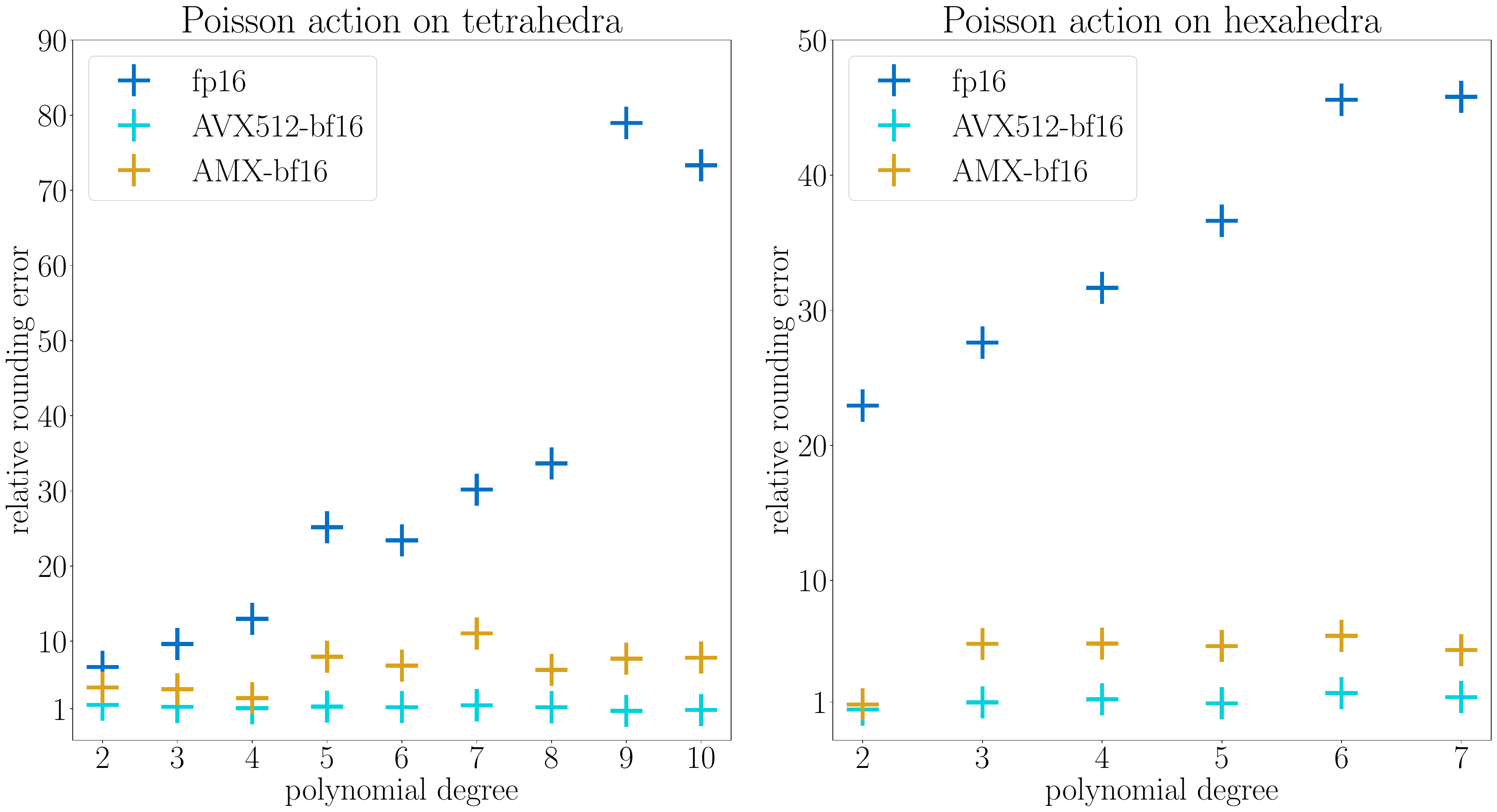}
    \end{subfigure}
    \caption{\textit{Relative rounding errors arising in the Poisson kernel
            evaluations via fp16 and fp32/bf16 mixed-precision plotted against
            the polynomial degree. The AVX512-bf16 and AMX-bf16 are almost
            overlapping for bilinear forms (top figures). The slight difference
            in the error constant of AVX512-bf16 and AMX-bf16 computations in
            the action kernels (bottom figures) is due to the different
            implementation: the AVX512-bf16 kernels are slightly more accurate
            since they evaluate FE functions using only single precision
            (cf.~Appendix \ref{appendix_sec:implementation_details}).}}
    \label{fig:figure_errors_Poisson_both}
\end{figure}

\subsection{Mixed-precision kernels: rounding errors.}

We now study the accuracy of the hardware-accelerated mixed-precision
kernels as described in Section \ref{subsec:MP_strategy}. Namely, we use single
precision for geometry, polynomial and FE function evaluations, and for the matrix-matrix
products and their accumulation ($u_p=u_g=u_q=2^{-24}$). For storage we use
instead bf16 ($u_s=2^{-8}$). Our aim is to validate the theory from Section
\ref{sec:rounding_error_analysis} as well as establishing that these
mixed-precision kernels yield indeed $O(u_s)$ errors where the hidden constant
is essentially independent from the conditioning of geometry computations and
polynomial and FE function evaluations as well as from the number of quadrature
nodes and the polynomial degree.

We run mixed-precision kernels on all meshes and for all forms using either
AVX512-bf16 or AMX-bf16 accelerators and we compare their error behaviour with a
kernel almost fully run in fp16 half precision (except for the geometry which we
have to run in fp32 to avoid underflow and much larger errors). We compute the
relative rounding errors as
\begin{align}
    \max_{k=1,\dots,\tilde{n}_K}\left(\dfrac{\lVert A^k - \hat{A}^k
    \rVert_{\max}}{u_s\lVert A^k \rVert_{\max}}\right),\quad
    \max_{k=1,\dots,\tilde{n}_K}\left(\dfrac{\lVert \bm{v}^k - \hat{\bm{v}}^k
    \rVert_{\infty}}{u_s\lVert \bm{v}^k \rVert_{\infty}}\right),
\end{align}
where $u_s=2^{-8}$ for the mixed-precision kernels using bf16 and
$u_s=2^{-11}$ for the fp16 kernels. Normalizing by the precision here is crucial
for a fair comparison as it sanitizes the results from the
choice of the half-precision format. Again, we would employ bf16 instead of
fp16, but pure bf16 operations are not supported in the CPU.
With this normalization the base-$10$ logarithm of the relative error is the
number of digits lost in the computation which helps with the interpretation of
the results.

Results are shown in Figure \ref{fig:figure_errors_mass_both} for the mass forms
and in Figure \ref{fig:figure_errors_Poisson_both} for the Poisson forms where
we plot the relative rounding errors versus the polynomial degree.  These
results match with what the theory predicts: while the fp16 errors
grow with the polynomial degree (and thus $n_q$ as well) and lose up to $2$
digits of accuracy (out of $3.5$), the relative rounding errors of the
mixed-precision kernels are essentially $O(1)$ and constant, which implies that
no digits have been lost. From Figure \ref{fig:geometry_errors} we also know that as
long as the condition number of the Jacobian is not larger than $10^5$ then the
errors arising from single-precision geometry computations are also negligible.
Therefore, these mixed-precision strategies are an effective way of computing
highly accurate kernels (given the precision) even at high polynomial degrees.

\begin{remark}
    In terms of absolute rounding errors the mixed-precision kernels are of the
    same order of the half-precision unit roundoff (roughly $2^{-8}\approx
    0.0039$) which is still orders of magnitude larger than the error that would
    be obtained with single- or double-precision-only kernels (cf.~Table
    \ref{tab:precision}). These considerations hold even after accounting for
    digit loss due to the larger error constants (see the fp16 results in
    Figures \ref{fig:figure_errors_mass_both} and
    \ref{fig:figure_errors_Poisson_both}).  Therefore, using the mixed-precision
    kernels is only suitable either when many digits of accuracy are not needed
    and/or after weighing this loss of accuracy against the speedups obtained
    (presented next).
\end{remark}

\subsection{Mixed-precision kernels: timings and speedup.}
\label{subsec:MP_kernels_speedups}

In this subsection we test the computational efficiency of the mixed-precision
kernels described in Section \ref{subsec:MP_strategy}. Our aim is to show that
these kernels are appreciably more efficient than their pure single or double
precision counterparts.

The kernels considered are the same as in the previous subsection, plus the
addition of single- and double-precision-only kernels. We measure the kernel
efficiency by the average per-cell CPU time taken by each kernel, excluding
tabulation and quadrature rule construction which are performed offline once and
for all for all cells. We then measure the speedup with respect to double
precision by taking the ratio between the timings of each kernel and the
corresponding double-precision timings.

Results are shown in Figure
\ref{fig:figure_timings_mass_both} for the mass forms and in Figure
\ref{fig:figure_timings_Poisson_both} for the Poisson forms. Firstly, we
note that the reduced-/mixed-precision implementation presented in this paper
is only beneficial for large enough polynomial degrees (larger than 4 or 5 for
tetrahedra and larger than 2 or 3 for hexahedra). This behaviour is
expected since our implementation is only really tailored to the higher
polynomial degrees for which the advantages of AVX512 and AMX accelerators are more
significant. Lower degrees imply a different cost and rounding error balance
between the kernel components and would require a separate, tailored
implementation which is outside the scope of this work. Secondly, we observe that
the speedups are clearly dependent on the precision and accelerator choices:
\begin{figure}[!htbp]
    \centering
    \begin{subfigure}[h]{\textwidth}
        \centering
        \includegraphics[width=\textwidth]{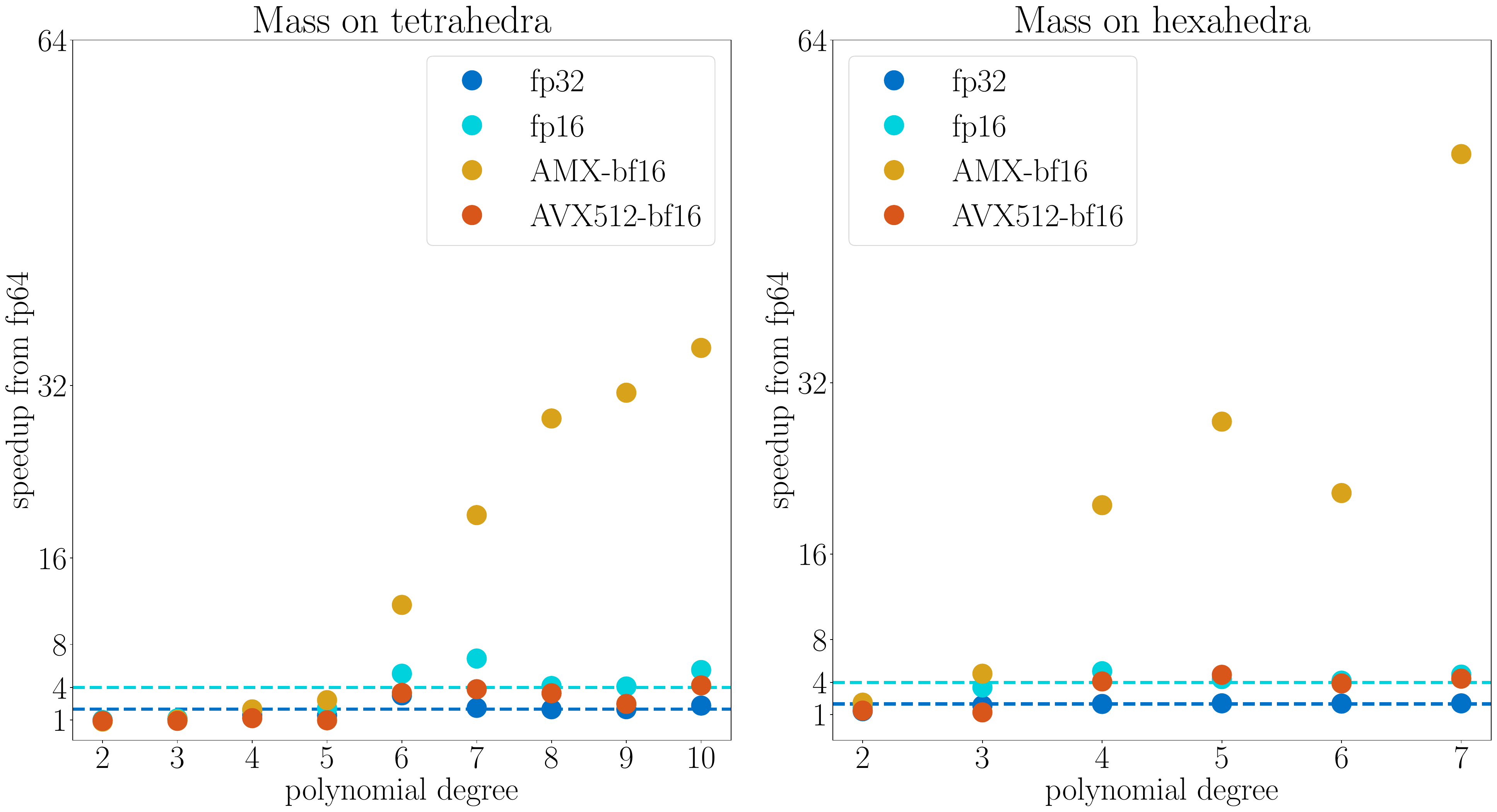}
    \end{subfigure}\\
    \vspace{6pt}
    \begin{subfigure}[h]{\textwidth}
        \centering
        \includegraphics[width=\textwidth]{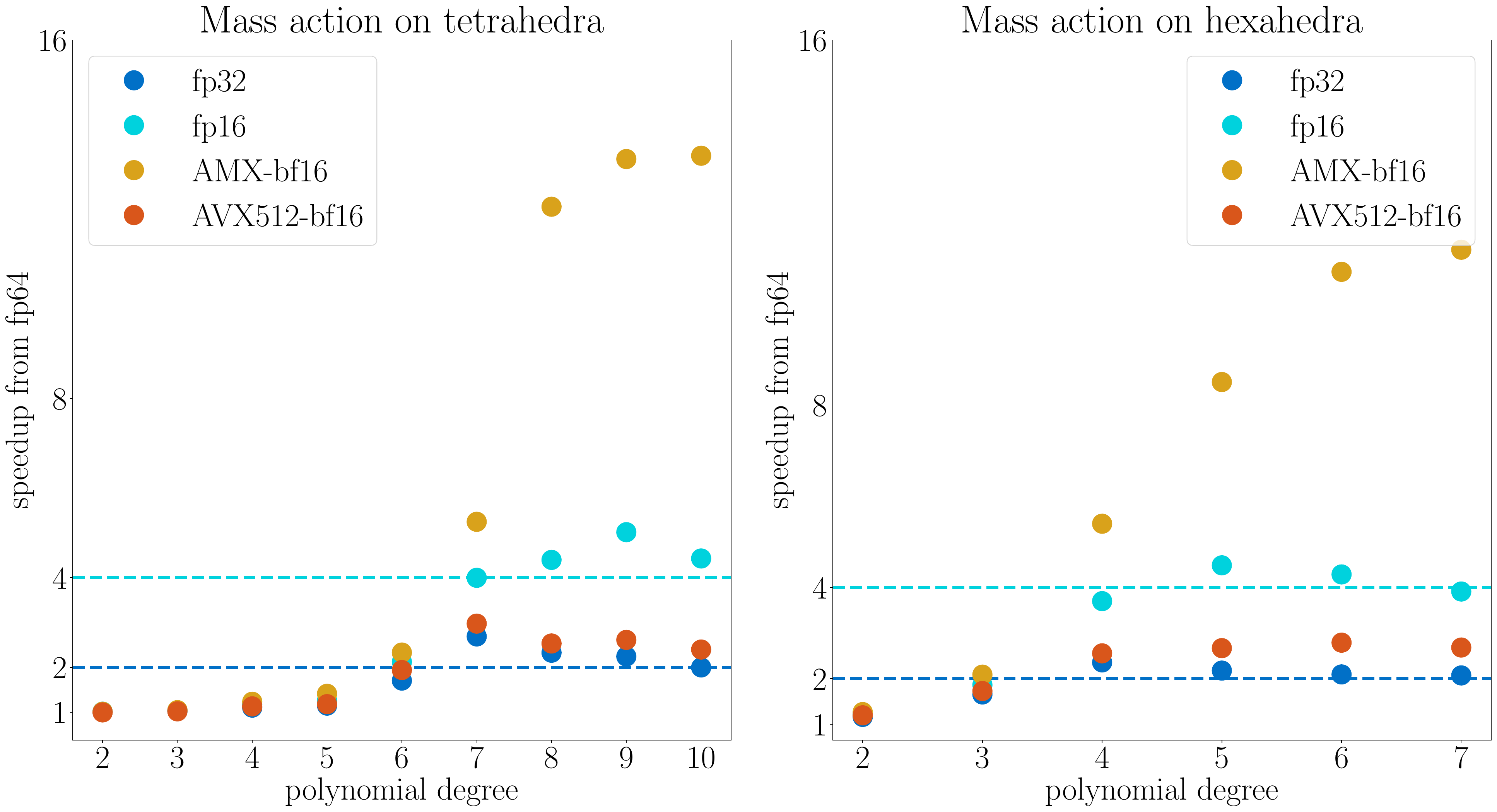}
    \end{subfigure}
    \caption{\textit{Computational speedup of the reduced- and mixed-precision
            mass kernels with respect to their double-precision equivalent
            versus the polynomial degree. The horizontal dashed lines correspond
            to the typical expected speedups obtainable with single and
            half-precision computations on CPUs ($2$ and $4$ times
            respectively). The largest speedups are obtained by the AMX-bf16
            kernels.}}
    \label{fig:figure_timings_mass_both}
    \vspace{-12pt}
\end{figure}
\begin{figure}[!htbp]
    \centering
    \begin{subfigure}[h]{\textwidth}
        \centering
        \includegraphics[width=\textwidth]{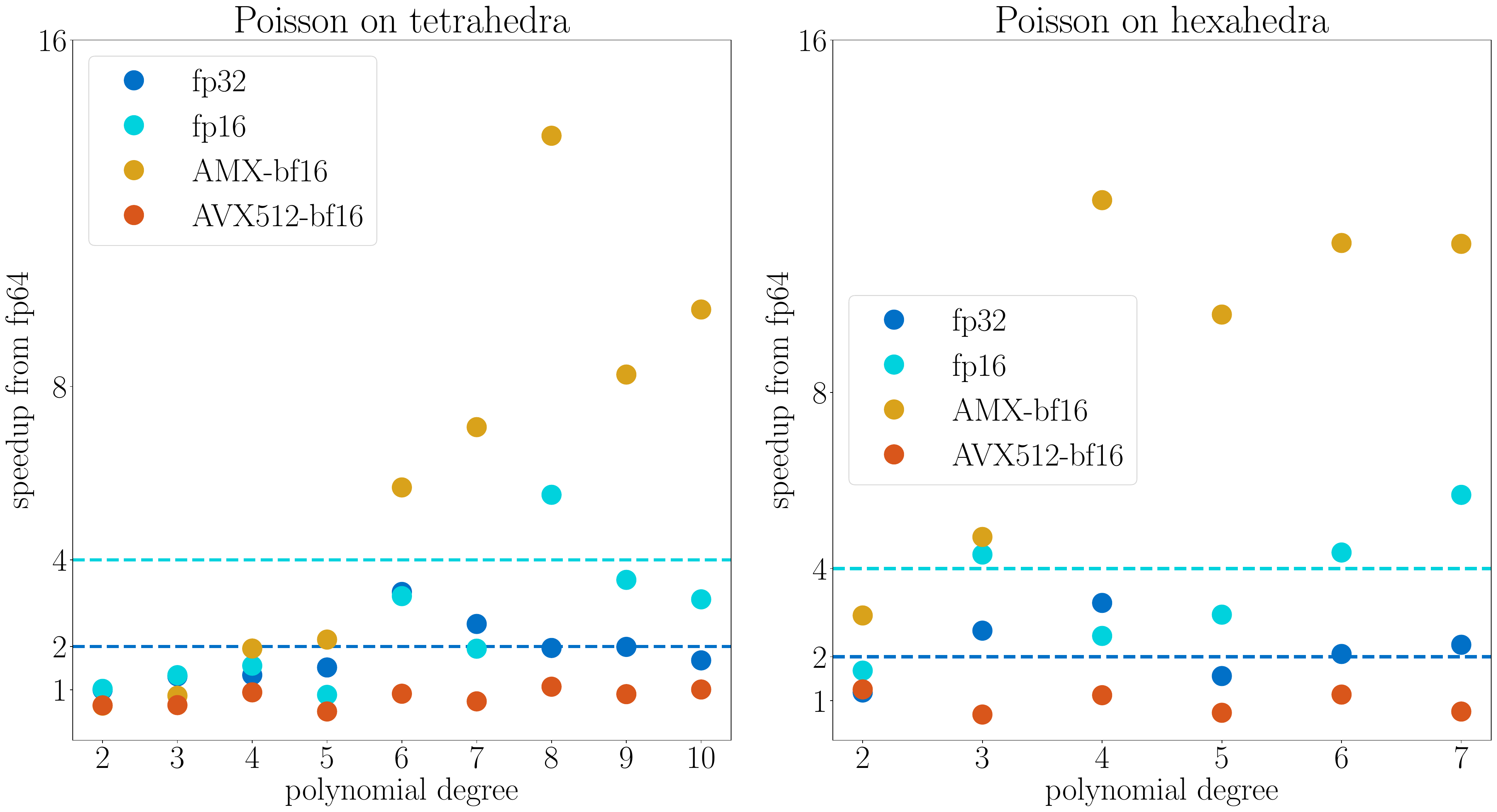}
    \end{subfigure}\\
    \vspace{6pt}
    \begin{subfigure}[h]{\textwidth}
        \centering
        \includegraphics[width=\textwidth]{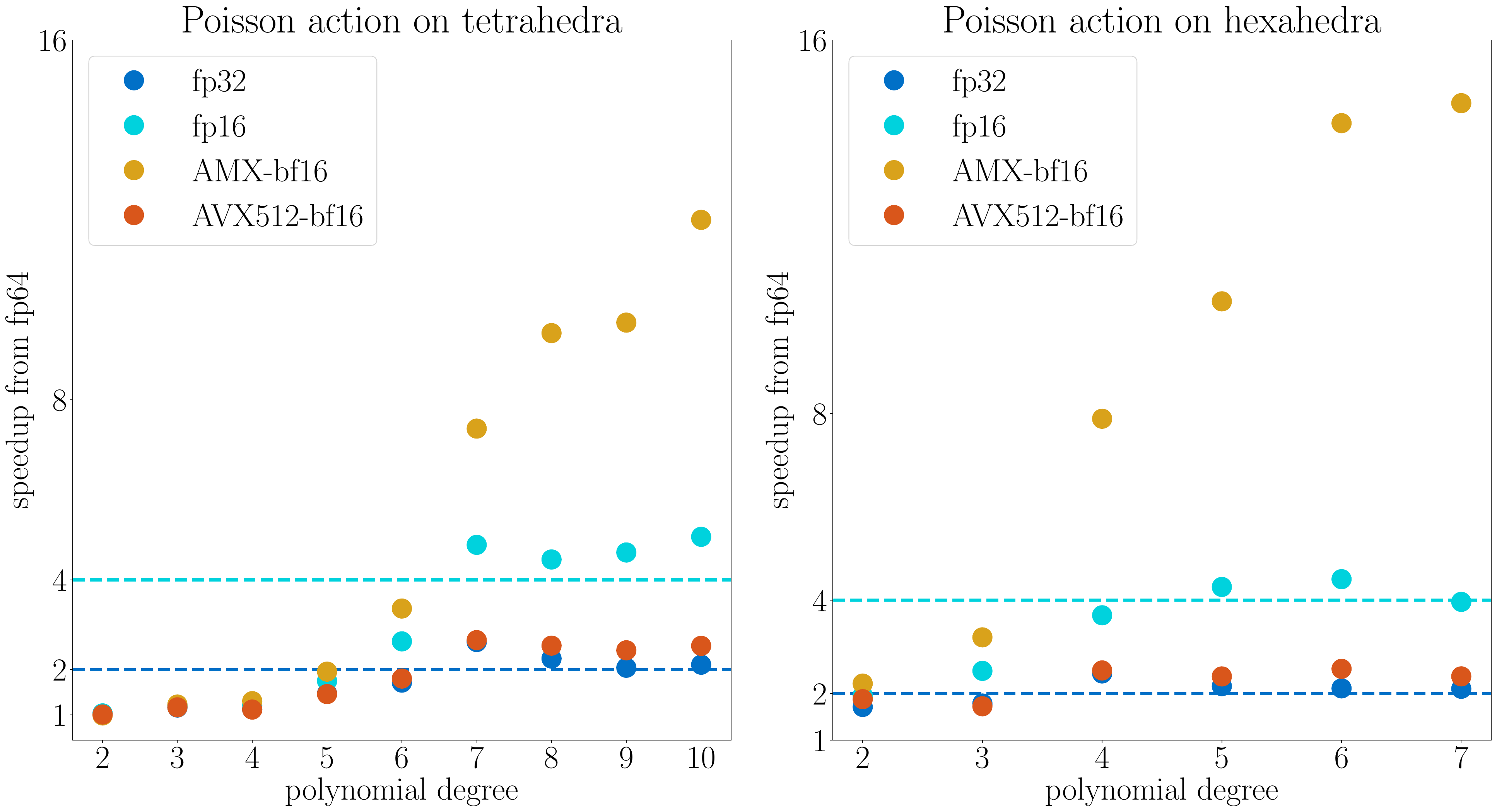}
    \end{subfigure}
    \caption{\textit{Computational speedup of the reduced- and mixed-precision
            Poisson kernels with respect to their double-precision equivalent
            versus the polynomial degree. The horizontal dashed lines correspond
            to the typical expected speedups obtainable with single and
            half-precision computations on CPUs ($2$ and $4$ times
            respectively). The largest speedups are obtained by the AMX-bf16
            kernels.}}
    \label{fig:figure_timings_Poisson_both}
    \vspace{-12pt}
\end{figure}
\begin{itemize}
    \item The (vectorized) single and half precision (dark and light blue
        respectively in the figures) are roughly up to twice and
        four times faster than double precision respectively. These are the
        typical improvements which we would expect to obtain by using
        lower-precision arithmetic (e.g., using half the digits results in twice
        the speed). We remark that these speedup factors do not
        occur in all reduced-precision applications so it is a positive result
        that they do occur for high-polynomial degree FE kernels.

    \item The vectorized mixed-precision AVX512-bf16 kernels (dark orange in the
        figures) yield varying results. For mass bilinear forms (Figure
        \ref{fig:figure_timings_mass_both}, top) they are as fast as
        half-precision with a factor-four speedup (while also being more
        accurate, recall Figure \ref{fig:figure_errors_mass_both}). This is also
        the maximum speedup that is theoretically achievable with AVX512-bf16
        instructions (recall Section \ref{subsec:background_on_accelerators}).
        For mass and Poisson actions (bottom of Figures
        \ref{fig:figure_timings_mass_both} and
        \ref{fig:figure_timings_Poisson_both}) they are only slightly faster
        than single precision, and for Poisson bilinear forms they are as slow
        as double precision. We believe this performance may be due to a
        lack of CPU and compiler features. In fact, mixed-precision and bf16
        instructions are limited if compared with typical AVX512 support and the
        compiler is unable to use these automatically. Even if we invoke these
        intrinsics by hand (which we do) this means that the full range of
        AVX512 kernel-wide compiler optimizations is unavailable here.
        Nevertheless, the four-times speedup obtained with mass bilinear forms
        is a promising result as it shows that these mixed-precision AVX512
        kernels can indeed be as fast as half precision while being more
        accurate.

    \item The AMX-bf16 kernels (ocra in the figures) outperform all other
        kernels and yield large speedups. For high-degree polynomials, these
        kernels are over $32$ times faster than double precision for mass
        bilinear forms on tetrahedra and up to $60$ times faster on hexahedra.
        For all other forms they are between $8$ and $16$ times faster than
        vectorized double precision kernels. These speedups are remarkable since
        they go well beyond the usual factor-four obtainable with half precision
        and, for the mass bilinear forms, approach the theoretical maximum
        speedup obtainable with AMX instructions (i.e., $64\times$, cf.~Section
        \ref{subsec:background_on_accelerators}). Furthermore, the AMX kernels
        are not only faster than fp16 computations, but also more accurate
        (recall Figures \ref{fig:figure_errors_mass_both} and
        \ref{fig:figure_errors_Poisson_both}). Additionally, we note that these
        kernels are not affected by the same slowdowns of the AVX512-bf16
        kernels, likely because AMX-bf16 accelerators have a much higher
        throughput which more than compensates for the limited CPU/compiler bf16
        support.
\end{itemize}

\begin{remark}
    \label{rem:no_numres_assembly}
    The results presented correspond as-is to the assembly of global tensors
    when discontinuous Lagrange basis functions are employed. We do not
    investigate numerically the assembly process of forms involving continuous
    Lagrange elements for the sake of brevity: continuous Lagrange elements only
    slightly affect timings and do not qualitatively affect rounding errors.
\end{remark}

\section{Conclusions}
\label{sec:conclusions}

Artificial intelligence applications have been driving the development of
computer chips that are optimized for parallel processing and reduced- and
mixed-precision computations.  While these hardware units are specifically
tailored to computer vision and machine learning workloads, the exploitation of
their parallel capabilities has also been shown to improve the efficiency of FE
computations.  However, the acceleration of the FE method via reduced- and
mixed-precision implementations has received less attention and the effect of
rounding errors in FE computations is underexplored in both theory and numerical
practice. Indeed, there is no comprehensive rounding error analysis of the FE
method that can guide reduced-precision implementations and ensure their
accuracy.

This paper is a first step in the direction of closing this gap in the
literature: We have derived the first rounding error analysis of FE kernels and
assembly which also accounts for mixed-precision and hardware accelerated
implementations. Furthermore, we have designed a mixed-precision implementation
strategy which exploits mixed-precision matrix units to obtain FE kernels and
assembly routines which are both performant and robust to half-precision
computations. Indeed, the resulting algorithms are provenly half-precision
accurate with an error constant that is independent from: the conditioning of FE
basis function evaluations, the ill-posedness of the cell, the polynomial
degree, and the number of quadrature nodes. Finally, we have presented the first
single-/half-precision FE kernel implementations exploiting Intel AMX units.
Numerical experimentation has shown that these kernels can be between 8 and 60
times faster than their double precision equivalents while being provenly more
accurate than their fully half precision counterparts.

Additional investigation is still required to derive a comprehensive rounding
error analysis and mixed-precision implementation of FE computations which also
includes timestepping, linear and nonlinear solvers, and preconditioning
strategies.  For instance, an open research question is whether it is possible
to design mixed-precision FE methods which benefit from the speedups resulting
from reduced-precision computations while retaining high-precision accuracy.

Further implementation work is also required. In particular, the design of
automatic architecture-specific code generation algorithms would be highly
beneficial as it would lighten the implementation burden on the user, accelerate
further research on the topic, and increase the adoption of reduced- and
mixed-precision FE methods by the scientific community.

\paragraph{Acknowledgements}
MC was supported by the European Union Horizon research and innovation
program under the Marie Sklodowska-Curie postdoctoral fellowship grant
101103593 (GEOLEARN). GNW was supported by Engineering and Physical
Sciences Research Council grants EP/S005072/1, EP/W00755X/1 and
EP/W026635/1.


\printbibliography

@article{amestoy2024five,
  title={Five-precision GMRES-based iterative refinement},
  author={Amestoy, Patrick and Buttari, Alfredo and Higham, Nicholas J and L’Excellent, Jean-Yves and Mary, Th{\'e}o and Vieubl{\'e}, Bastien},
  journal={SIAM Journal on Matrix Analysis and Applications},
  volume={45},
  number={1},
  pages={529--552},
  year={2024},
  publisher={SIAM}
}

@article{carson2023mixed,
  title={Mixed precision iterative refinement with sparse approximate inverse preconditioning},
  author={Carson, Erin and Khan, Noaman},
  journal={SIAM Journal on Scientific Computing},
  volume={45},
  number={3},
  pages={C131--C153},
  year={2023},
  publisher={SIAM}
}

@article{flegar2021adaptive,
  title={Adaptive precision block-Jacobi for high performance preconditioning in the Ginkgo linear algebra software},
  author={Flegar, Goran and Anzt, Hartwig and Cojean, Terry and Quintana-Orti, Enrique S},
  journal={ACM Transactions on Mathematical Software (TOMS)},
  volume={47},
  number={2},
  pages={1--28},
  year={2021},
  publisher={ACM New York, NY, USA}
}

@article{oktay2022multistage,
  title={Multistage mixed precision iterative refinement},
  author={Oktay, Eda and Carson, Erin},
  journal={Numerical Linear Algebra with Applications},
  volume={29},
  number={4},
  pages={e2434},
  year={2022},
  publisher={Wiley Online Library}
}

@article{mccormick2021algebraic,
  title={Algebraic error analysis for mixed-precision multigrid solvers},
  author={McCormick, Stephen F and Benzaken, Joseph and Tamstorf, Rasmus},
  journal={SIAM Journal on Scientific Computing},
  volume={43},
  number={5},
  pages={S392--S419},
  year={2021},
  publisher={SIAM}
}

@article{tamstorf2021discretization,
  title={Discretization-error-accurate mixed-precision multigrid solvers},
  author={Tamstorf, Rasmus and Benzaken, Joseph and McCormick, Stephen F},
  journal={SIAM Journal on Scientific Computing},
  volume={43},
  number={5},
  pages={S420--S447},
  year={2021},
  publisher={SIAM}
}

@inproceedings{georgiou2023mixed,
  title={A mixed precision randomized preconditioner for the LSQR solver on GPUs},
  author={Georgiou, Vasileios and Boutsikas, Christos and Drineas, Petros and Anzt, Hartwig},
  booktitle={International Conference on High Performance Computing},
  pages={164--181},
  year={2023},
  organization={Springer}
}

@book{higham2002accuracy,
  title={Accuracy and Stability of Numerical Algorithms},
  author={Higham, Nicholas J},
  year={2002},
  publisher={SIAM}
}

@article{pena2000multivariate,
  title={On the multivariate Horner scheme},
  author={Pe{\~n}a, Juan Manuel and Sauer, Thomas},
  journal={SIAM Journal on Numerical Analysis},
  volume={37},
  number={4},
  pages={1186--1197},
  year={2000},
  publisher={SIAM}
}

@article{higham2004numerical,
  title={The numerical stability of barycentric Lagrange interpolation},
  author={Higham, Nicholas J},
  journal={IMA Journal of Numerical Analysis},
  volume={24},
  number={4},
  pages={547--556},
  year={2004},
  publisher={Oxford University Press}
}

@book{ciarlet2002finite,
  title={The finite element method for elliptic problems},
  author={Ciarlet, Philippe G},
  year={2002},
  publisher={SIAM}
}

@article{blanchard2020class,
  title={A class of fast and accurate summation algorithms},
  author={Blanchard, Pierre and Higham, Nicholas J and Mary, Theo},
  journal={SIAM Journal on Scientific Computing},
  volume={42},
  number={3},
  pages={A1541--A1557},
  year={2020},
  publisher={SIAM}
}

@article{blanchard2020mixed,
  title={Mixed precision block fused multiply-add: Error analysis and application to GPU tensor cores},
  author={Blanchard, Pierre and Higham, Nicholas J and Lopez, Florent and Mary, Theo and Pranesh, Srikara},
  journal={SIAM Journal on Scientific Computing},
  volume={42},
  number={3},
  pages={C124--C141},
  year={2020},
  publisher={SIAM}
}

@article{fasi2021numerical,
  title={Numerical behavior of NVIDIA tensor cores},
  author={Fasi, Massimiliano and Higham, Nicholas J and Mikaitis, Mantas and Pranesh, Srikara},
  journal={PeerJ Computer Science},
  volume={7},
  pages={e330},
  year={2021},
  publisher={PeerJ Inc.}
}

@article{isaac2020recursive,
  title={Recursive, parameter-free, explicitly defined interpolation nodes for simplices},
  author={Isaac, Tobin},
  journal={SIAM Journal on Scientific Computing},
  volume={42},
  number={6},
  pages={A4046--A4062},
  year={2020},
  publisher={SIAM}
}

@article{kirby2007efficient,
  title={Efficient compilation of a class of variational forms},
  author={Kirby, Robert C and Logg, Anders},
  journal={ACM Transactions on Mathematical Software (TOMS)},
  volume={33},
  number={3},
  pages={17--es},
  year={2007},
  publisher={ACM New York, NY, USA}
}

@article{higham2019new,
  title={A new approach to probabilistic rounding error analysis},
  author={Higham, Nicholas J and Mary, Theo},
  journal={SIAM Journal on Scientific Computing},
  volume={41},
  number={5},
  pages={A2815--A2835},
  year={2019},
  publisher={SIAM}
}

@article{klower2022fluid,
  title={Fluid simulations accelerated with 16 bits: Approaching 4x speedup on A64FX by squeezing ShallowWaters. jl into Float16},
  author={Kl{\"o}wer, Milan and Hatfield, Sam and Croci, Matteo and D{\"u}ben, Peter D and Palmer, Tim N},
  journal={Journal of Advances in Modeling Earth Systems},
  volume={14},
  number={2},
  pages={e2021MS002684},
  year={2022},
  publisher={Wiley Online Library}
}

@article{croci2022mixed,
  title={Mixed-precision explicit stabilized Runge--Kutta methods for single-and multi-scale differential equations},
  author={Croci, Matteo and de Souza, Giacomo Rosilho},
  journal={Journal of Computational Physics},
  volume={464},
  pages={111349},
  year={2022},
  publisher={Elsevier}
}

@article{croci2022stochastic,
  title={Stochastic rounding: implementation, error analysis and applications},
  author={Croci, Matteo and Fasi, Massimiliano and Higham, Nicholas J and Mary, Theo and Mikaitis, Mantas},
  journal={Royal Society Open Science},
  volume={9},
  number={3},
  pages={211631},
  year={2022},
  publisher={The Royal Society}
}

@inproceedings{balos2023leveraging,
  title={Leveraging mixed precision in exponential time integration methods},
  author={Balos, Cody J and Roberts, Steven and Gardner, David J},
  booktitle={2023 IEEE High Performance Extreme Computing Conference (HPEC)},
  pages={1--8},
  year={2023},
  organization={IEEE}
}

@article{burnett2024stability,
  title={Stability analysis and performance evaluation of additive mixed-precision Runge-Kutta methods},
  author={Burnett, Ben and Gottlieb, Sigal and Grant, Zachary J},
  journal={Communications on Applied Mathematics and Computation},
  volume={6},
  number={1},
  pages={705--738},
  year={2024},
  publisher={Springer}
}

@article{grant2022perturbed,
  title={Perturbed Runge--Kutta methods for mixed precision applications},
  author={Grant, Zachary J},
  journal={Journal of Scientific Computing},
  volume={92},
  number={1},
  pages={6},
  year={2022},
  publisher={Springer}
}

@article{tisseur2001newton,
  title={Newton's method in floating point arithmetic and iterative refinement of generalized eigenvalue problems},
  author={Tisseur, Fran{\c{c}}oise},
  journal={SIAM Journal on Matrix Analysis and Applications},
  volume={22},
  number={4},
  pages={1038--1057},
  year={2001},
  publisher={SIAM}
}

@article{kelley2022newton,
  title={Newton's method in mixed precision},
  author={Kelley, CT},
  journal={SIAM Review},
  volume={64},
  number={1},
  pages={191--211},
  year={2022},
  publisher={SIAM}
}

@article{croci2023effects,
  title={Effects of round-to-nearest and stochastic rounding in the numerical solution of the heat equation in low precision},
  author={Croci, Matteo and Giles, Michael B},
  journal={IMA Journal of Numerical Analysis},
  volume={43},
  number={3},
  pages={1358--1390},
  year={2023},
  publisher={Oxford University Press}
}

@article{rognes2010efficient,
  title={Efficient assembly of H(div) and H(curl) conforming finite elements},
  author={Rognes, Marie E and Kirby, Robert C and Logg, Anders},
  journal={SIAM Journal on Scientific Computing},
  volume={31},
  number={6},
  pages={4130--4151},
  year={2010},
  publisher={SIAM}
}

@article{nicolaides1972class,
  title={On a class of finite elements generated by Lagrange interpolation},
  author={Nicolaides, RA},
  journal={SIAM Journal on Numerical Analysis},
  volume={9},
  number={3},
  pages={435--445},
  year={1972},
  publisher={SIAM}
}

@article{wathen1987realistic,
  title={Realistic eigenvalue bounds for the Galerkin mass matrix},
  author={Wathen, Andrew J},
  journal={IMA Journal of Numerical Analysis},
  volume={7},
  number={4},
  pages={449--457},
  year={1987},
  publisher={Oxford University Press}
}

@article{ipsen2008perturbation,
  title={Perturbation bounds for determinants and characteristic polynomials},
  author={Ipsen, Ilse CF and Rehman, Rizwana},
  journal={SIAM Journal on Matrix Analysis and Applications},
  volume={30},
  number={2},
  pages={762--776},
  year={2008},
  publisher={SIAM}
}

@misc{BarattaEtal2023,
  title     = {{DOLFINx}: the next generation {FEniCS} problem solving environment},
  author    = {Baratta, Igor A. and Dean, Joseph P. and Dokken, J{\o}rgen S. and Habera, Michal and Hale, Jack S. and Richardson, Chris N. and Rognes, Marie E. and Scroggs, Matthew W. and Sime, Nathan and Wells, Garth N.},
  doi       = {10.5281/zenodo.10447666},
  year      = {2023},
  howpublished = {preprint}
}

@article{ScroggsEtal2022,
  title     = {Construction of arbitrary order finite element degree-of-freedom maps on polygonal and polyhedral cell meshes},
  author    = {Scroggs, Matthew W. and Dokken, J{\o}rgen S. and Richardson, Chris N. and Wells, Garth N.},
  journal   = {ACM Transactions on Mathematical Software},
  year      = {2022},
  volume    = {48},
  number    = {2},
  doi       = {10.1145/3524456},
  pages     = {{18:1--18:23}},
}

@article{BasixJoss,
  title     = {Basix: a runtime finite element basis evaluation library},
  author    = {Scroggs, Matthew W. and Baratta, Igor A. and Richardson, Chris N. and Wells, Garth N.},
  journal   = {Journal of Open Source Software},
  year      = {2022},
  volume    = {7},
  number    = {73},
  doi       = {10.21105/joss.03982},
  pages     = {3982}
}

@article{geuzaine2009gmsh,
  title={Gmsh: A 3-D finite element mesh generator with built-in pre-and post-processing facilities},
  author={Geuzaine, Christophe and Remacle, Jean-Fran{\c{c}}ois},
  journal={International Journal for Numerical Methods in Engineering},
  volume={79},
  number={11},
  pages={1309--1331},
  year={2009},
  publisher={Wiley Online Library}
}

@misc{intel-dev-manual,
    title = {{Intel\textsuperscript{\textregistered} 64 and IA-32 Architectures
             Software Developer Manuals, August 2024 version}},
    howpublished =
    {\url{https://www.intel.com/content/www/us/en/developer/articles/technical/intel-sdm.html}},
    note = {Accessed: 2024-07-24}
}

@misc{intel-intrinsics,
    title = {{Intel\textsuperscript{\textregistered} Intrinsics Guide version
             3.6.9}},
    howpublished =
    {\url{https://www.intel.com/content/www/us/en/docs/intrinsics-guide/index.html}},
    note = {Accessed: 2024-07-12}
}

@book{wilkinson2023rounding,
  title={Rounding Errors in Algebraic Processes},
  author={Wilkinson, James Hardy},
  year={2023},
  publisher={SIAM}
}

@article{knepley2013finite,
  title={Finite element integration on GPUs},
  author={Knepley, Matthew G and Terrel, Andy R},
  journal={ACM Transactions on Mathematical Software (TOMS)},
  volume={39},
  number={2},
  pages={1--13},
  year={2013},
  publisher={ACM New York, NY, USA}
}

@article{fu2014architecting,
  title={Architecting the finite element method pipeline for the GPU},
  author={Fu, Zhisong and Lewis, T James and Kirby, Robert M and Whitaker, Ross T},
  journal={Journal of Computational and Applied Mathematics},
  volume={257},
  pages={195--211},
  year={2014},
  publisher={Elsevier}
}

@article{reguly2015finite,
  title={Finite element algorithms and data structures on graphical processing units},
  author={Reguly, Istv{\'a}n Zolt{\'a}n and Giles, Michael B},
  journal={International Journal of Parallel Programming},
  volume={43},
  pages={203--239},
  year={2015},
  publisher={Springer}
}

@article{pazner2023end,
  title={End-to-end GPU acceleration of low-order-refined preconditioning for high-order finite element discretizations},
  author={Pazner, Will and Kolev, Tzanio and Camier, Jean-Sylvain},
  journal={The International Journal of High Performance Computing Applications},
  volume={37},
  number={5},
  pages={578--599},
  year={2023},
  publisher={SAGE Publications Sage UK: London, England}
}

@article{dziekonski2013generation,
  title={Generation of large finite-element matrices on multiple graphics processors},
  author={Dziekonski, A and Sypek, Piotr and Lamecki, A and Mrozowski, Micha{\l}},
  journal={International Journal for Numerical Methods in Engineering},
  volume={94},
  number={2},
  pages={204--220},
  year={2013},
  publisher={Wiley Online Library}
}

@inproceedings{beams2020high,
  title={High-order finite element method using standard and device-level batch GEMM on GPUs},
  author={Beams, Natalie and Abdelfattah, Ahmad and Tomov, Stan and Dongarra, Jack and Kolev, Tzanio and Dudouit, Yohann},
  booktitle={2020 IEEE/ACM 11th Workshop on Latest Advances in Scalable Algorithms for Large-Scale Systems (ScalA)},
  pages={53--60},
  year={2020},
  organization={IEEE}
}

@article{andrej2024high,
  title={High-performance finite elements with MFEM},
  author={Andrej, Julian and Atallah, Nabil and B{\"a}cker, Jan-Phillip and Camier, Jean-Sylvain and Copeland, Dylan and Dobrev, Veselin and Dudouit, Yohann and Duswald, Tobias and Keith, Brendan and Kim, Dohyun and others},
  journal={The International Journal of High Performance Computing Applications},
  pages={10943420241261981},
  year={2024},
  publisher={SAGE Publications Sage UK: London, England}
}

@article{trotter2023targeting,
  title={Targeting performance and user-friendliness: GPU-accelerated finite element computation with automated code generation in FEniCS},
  author={Trotter, James D and Langguth, Johannes and Cai, Xing},
  journal={Parallel Computing},
  volume={118},
  pages={103051},
  year={2023},
  publisher={Elsevier}
}

@article{banas2014numerical,
  title={Numerical integration on GPUs for higher order finite elements},
  author={Bana{\'s}, Krzysztof and P{\l}aszewski, Przemys{\l}aw and Macio{\l}, Pawe{\l}},
  journal={Computers \& Mathematics with Applications},
  volume={67},
  number={6},
  pages={1319--1344},
  year={2014},
  publisher={Elsevier}
}

@article{cecka2011assembly,
  title={Assembly of finite element methods on graphics processors},
  author={Cecka, Cris and Lew, Adrian J and Darve, Eric},
  journal={International Journal for Numerical Methods in Engineering},
  volume={85},
  number={5},
  pages={640--669},
  year={2011},
  publisher={Wiley Online Library}
}

@article{remacle2016gpu,
  title={GPU accelerated spectral finite elements on all-hex meshes},
  author={Remacle, J-F and Gandham, Rajesh and Warburton, Tim},
  journal={Journal of Computational Physics},
  volume={324},
  pages={246--257},
  year={2016},
  publisher={Elsevier}
}

@article{babuvska2018roundoff,
  title={On roundoff error growth in elliptic problems},
  author={Babu{\v{s}}ka, Ivo and S{\"o}derlind, Gustaf},
  journal={ACM Transactions on Mathematical Software (TOMS)},
  volume={44},
  number={3},
  pages={1--22},
  year={2018},
  publisher={ACM New York, NY, USA}
}

@article{alvarez2012round,
  title={On round-off error for adaptive finite element methods},
  author={Alvarez-Aramberri, Julen and Pardo, David and Paszynski, Maciej and Collier, Nathan and Dalcin, Lisandro and Calo, Victor M},
  journal={Procedia Computer Science},
  volume={9},
  pages={1474--1483},
  year={2012},
  publisher={Elsevier}
}

@article{melosh1973inherited,
  title={Inherited error in finite element analyses of structures},
  author={Melosh, Robert J},
  journal={Computers \& Structures},
  volume={3},
  number={5},
  pages={1205--1217},
  year={1973},
  publisher={Elsevier}
}

@article{utku1984solution,
  title={Solution errors in finite element analysis},
  author={Utku, Senol and Melosh, Robert J},
  journal={Computers \& Structures},
  volume={18},
  number={3},
  pages={379--393},
  year={1984},
  publisher={Elsevier}
}

@techreport{melosh1969manipulation,
  title={Manipulation errors in finite element analysis of structures},
  author={Melosh, Robert J and Palacol, E. L.},
  year={1969},
  institution={NASA}
}

@phdthesis{fried1971discretization,
  title={Discretization and round-off errors in the finite element analysis of elliptic boundary value problems and eigenvalue problems.},
  author={Fried, Isaac},
  year={1971},
  school={Massachusetts Institute of Technology}
}

@article{maryvska2000schur,
  title={Schur complement reduction in the mixed-hybrid approximation of Darcy's law: rounding error analysis},
  author={Mary{\v{s}}ka, Ji{\v{r}}{\'\i} and Rozlo{\v{z}}n{\i}k, M and Tuma, M},
  journal={Journal of Computational and Applied Mathematics},
  volume={117},
  number={2},
  pages={159--173},
  year={2000},
  publisher={Elsevier}
}

@article{lewis2022large,
  title={Large-scale distributed linear algebra with tensor processing units},
  author={Lewis, Adam GM and Beall, Jackson and Ganahl, Martin and Hauru, Markus and Mallick, Shrestha Basu and Vidal, Guifre},
  journal={Proceedings of the National Academy of Sciences},
  volume={119},
  number={33},
  pages={e2122762119},
  year={2022},
  publisher={National Acad Sciences}
}

@article{fasi2023matrix,
  title={Matrix multiplication in multiword arithmetic: Error analysis and application to GPU tensor cores},
  author={Fasi, Massimiliano and Higham, Nicholas J and Lopez, Florent and Mary, Theo and Mikaitis, Mantas},
  journal={SIAM Journal on Scientific Computing},
  volume={45},
  number={1},
  pages={C1--C19},
  year={2023},
  publisher={SIAM}
}

@article{lopez2023mixed,
  title={Mixed precision LU factorization on GPU tensor cores: reducing data movement and memory footprint},
  author={Lopez, Florent and Mary, Theo},
  journal={The International Journal of High Performance Computing Applications},
  volume={37},
  number={2},
  pages={165--179},
  year={2023},
  publisher={SAGE Publications Sage UK: London, England}
}

@article{haidar2020mixed,
  title={Mixed-precision iterative refinement using tensor cores on GPUs to accelerate solution of linear systems},
  author={Haidar, Azzam and Bayraktar, Harun and Tomov, Stanimire and Dongarra, Jack and Higham, Nicholas J},
  journal={Proceedings of the Royal Society A},
  volume={476},
  number={2243},
  pages={20200110},
  year={2020},
  publisher={The Royal Society Publishing}
}

@article{higham2022mixed,
  title={Mixed precision algorithms in numerical linear algebra},
  author={Higham, Nicholas J and Mary, Theo},
  journal={Acta Numerica},
  volume={31},
  pages={347--414},
  year={2022},
  publisher={Cambridge University Press}
}

@article{abdelfattah2021survey,
  title={A survey of numerical linear algebra methods utilizing mixed-precision arithmetic},
  author={Abdelfattah, Ahmad and Anzt, Hartwig and Boman, Erik G and Carson, Erin and Cojean, Terry and Dongarra, Jack and Fox, Alyson and Gates, Mark and Higham, Nicholas J and Li, Xiaoye S and others},
  journal={The International Journal of High Performance Computing Applications},
  volume={35},
  number={4},
  pages={344--369},
  year={2021},
  publisher={SAGE Publications Sage UK: London, England}
}

@book{karniadakis2005spectral,
  title={Spectral/$hp$ Element Methods for Computational Fluid Dynamics},
  author={Karniadakis, George and Sherwin, Spencer J},
  year={2005},
  publisher={Oxford University Press, USA}
}

\newpage

\appendix

\section{Rounding error analysis of finite element function evaluations.}
\label{appendix_sec:polyval}

The rounding error analysis of multivariate
polynomial evaluations via Horner's method is known: In
\cite{pena2000multivariate} the following result is proven:

\begin{theorem}[\cite{pena2000multivariate}]
    \label{th:generic_poly_eval}
    Let $\phi(\x)$ with $\x\in K\subset\mathbb{R}^d$ be a
    multivariate polynomial in $d$ dimensions of total degree $m$. Let $u$ be
    the working precision. There exist polynomial evaluation algorithms and an
    algorithm-dependent constant $c_{\phi}\geq 1$ which, assuming $(c_{\phi}m +
    d)u < 1$, evaluate $\phi(\x)$ in finite precision yielding instead the
    quantity $\hat{\phi}(\x)$ satisfying
    \begin{align}
        |\phi(\x) - \hat{\phi}(\x)| \leq \gamma_{c_{\phi}m + d}\ \kappa(K,m,\phi)|\phi(\x)|.
    \end{align}
    Here $\kappa(K,m,\phi)$ is the condition number of the evaluation problem
    which is algorithm-dependent. For the multivariate Horner scheme from
    \cite{pena2000multivariate} and the monomial basis, $c_{\phi}=2$.
\end{theorem}

The above result holds for all polynomials and can thus be applied to any
polynomial FE basis. However, for polynomials that are expressible
as a product of monomials, such as Lagrange polynomials on simplicies
(cf.~\cite{nicolaides1972class}) and quadrilaterals and hexahedra, the
bound can be tightened:

\begin{lemma}
    \label{lemma:lagrange_poly_eval}
    Let $\x\in K\subset\mathbb{R}^d$ and let $\phi(\x)$
    be a multivariate polynomial of total degree $m$ which can be written in the
    form,
    \begin{align}
        \label{eq:polyval_product_repr}
        \phi(\x) = w_{\phi}\prod_{i=1}^m \ell_i(\x),
    \end{align}
    where $w_\phi\in\mathbb{R}$ is known or pre-computed in exact arithmetic,
    and $\{\ell_i(\x)\}_{i=1}^m$ are first-degree monomials. Assume that the
    evaluation of the monomials in finite precision yields instead
    $\hat{\ell}_i(\x)=\ell_i(\x)(1 + \theta_r)$ with $|\theta_r|\leq \gamma_r$.
    Then, provided that $u(m(r+1)+1)<1$, evaluating $\phi(\x)$ as the product
    above yields instead $\hat{\phi}(\x)$ satisfying
    \begin{align}
        |\phi(\x) - \hat{\phi}(\x)| \leq \gamma_{m(r+1)+1}|\phi(\x)|.
    \end{align}
\end{lemma}

\begin{proof}
    Under the assumptions of the lemma it holds that
    \begin{align}
        \label{eq:_lemma_lagrange_poly_eval_1}
        \hat{\phi}(\x) = \prod_{i=1}^{m+1}(1+\delta_i)w_{\phi}\prod_{i=1}^m
        \hat{\ell}_i(\x) =
        (1+\theta_{m+1})(1+\theta_r)^m\left(w_{\phi}\prod_{i=1}^m
        \ell_i(\x)\right) =  (1+\theta_{m(r+1)+1})\phi(\x).
    \end{align}
    Therefore
    \begin{align}
        |\phi(\x) - \hat{\phi}(\x)| \leq \gamma_{m(r + 1)+1}|\phi(\x)|,
    \end{align}
    which is the thesis.
\end{proof}
Note that the result in Lemma \ref{lemma:lagrange_poly_eval} is stronger than
Theorem \ref{th:generic_poly_eval} since there is no condition number appearing
in the error bound, i.e., the evaluation problem in this case has perfect
conditioning.

What about derivatives? Since derivatives of polynomials are polynomials,
Theorem \ref{th:generic_poly_eval} still applies, although it requires
expressing the derivatives in terms of the multivariate monomial basis. We adopt
a different strategy: under the assumptions of Lemma
\ref{lemma:lagrange_poly_eval}, we can write the derivative of $\phi(\x)$ as
follows: for $s\in\{1,\dots,d\}$,
\begin{align}
    \label{eq:polyval_derivatives}
    \partial_{s} \phi(\x) = w_{\phi}\sum_{j=1}^m(\partial_{s}\ell_j)\prod_{i\neq
    j}^m\ell_i(\x).
\end{align}
Here $\partial_s$ denotes the partial derivative with respect to $\x_s$.
Note that since $\ell_i$ are first-degree monomials, their derivatives are
constants and can only attain the values $\pm 1$.

\begin{lemma}
    \label{lemma:polyval_derivatives}
    Let the assumptions of Lemma \ref{lemma:lagrange_poly_eval} hold. The
    evaluation of equation \eqref{eq:polyval_derivatives} yields
    $\widehat{\partial_{x_s} \phi(\x)}$ satisfying:
    \begin{align}
        \label{eq:polyval_derivatives_1}
        |\partial_{s}\phi(\x) - \widehat{\partial_{s}\phi}(\x)| &\leq \gamma_{(m-1)(r + 1)}\sum_{j=1}^m\left|\frac{\phi(\x)}{\ell_j(\x)}\right|,\\[0.25em]
        \label{eq:polyval_derivatives_2}
            \max_{\x\in K}|\partial_{s}\phi(\x) -
        \widehat{\partial_{s}\phi}(\x)| &\leq \gamma_{(m-1)(r +
            1)}\kappa(\partial_{s}\phi)\max_{\x\in
            K}|\partial_{s}\phi(\x)|,\\[0.5em]
            \text{where}\quad \kappa(\partial_{s}\phi) &= \dfrac{\max\limits_{\x\in
    K}\sum\limits_{j=1}^m\left|\phi(\x)/\ell_j(\x)\right|}{\max\limits_{\x\in
            K}|\partial_{s}\phi(\x)|}\geq 1.
    \end{align}
\end{lemma}
\begin{proof}
    Since $\partial_{s}\ell_j = \pm 1$, these terms are exactly representable
    and their product is always exact. Therefore, each term in the sum in
    \eqref{eq:polyval_derivatives} is of the same form as
    \eqref{eq:polyval_product_repr} with one less term in the product and
    without the final multiplication by $w_\phi$. It thus
    holds that
    \begin{align}
        \yhwidehat{(\partial_{s}\ell_j)\prod_{i\neq
        j}^m\ell_i(\x)} =
        (1+\theta_{(m-1)(r+1)-1})(\partial_{s}\ell_j)\prod_{i\neq
        j}^m\ell_i(\x).
    \end{align}
    Here we have used the same argument as for equation
    \eqref{eq:_lemma_lagrange_poly_eval_1}. Hence, we have that
    \begin{align}
        \widehat{\partial_{s}\phi}(\x) 
        = \left(1 +
        \theta_{(m-1)(r+1)}\right)w_{\phi}\sum_{j=1}^m(\partial_{s}\ell_j)\prod_{i\neq
            j}^m\ell_i(\x)(1+\theta_{m-1}^j),
    \end{align}
    where we write $(1+\theta_{m-1}^j)$ to indicate that the rounding errors in
    the sum are $j$-dependent. From the above and using $|\partial_{s}\ell_j|
    = 1$ we obtain the bound
    \begin{align}
        |\partial_{s}\phi(\x) - \widehat{\partial_{s}\phi}(\x)| \leq
    \gamma_{(m-1)(r + 1)}\sum_{j=1}^m\left|w_{\phi}\prod_{i\neq
        j}^m\ell_i(\x)\right|=\gamma_{(m-1)(r +
        1)}\sum_{j=1}^m\left|\frac{\phi(\x)}{\ell_j(\x)}\right|,
    \end{align}
    which is equation \eqref{eq:polyval_derivatives_1}. Taking the maximum over
    all $\x\in K$ (recall that $K$ is compact by assumption) and multiplying
    and dividing by $\max_{\x\in K}|\partial_{s}\phi(\x)|$ yields
    equation \eqref{eq:polyval_derivatives_2}. The bound
    $\kappa(\partial_{s}\phi)\geq 1$ is proven by applying the triangle
    inequality to equation \eqref{eq:polyval_derivatives} and using again
    $|\partial_{s}\ell_j| = 1$. This final step concludes the proof.
\end{proof}

It is straightforward to extend the error bounds for basis function evaluations
to bounds for finite element function evaluations:

\begin{lemma}
    \label{lemma:FEM_function_eval}
    Let $\K$ be a reference cell and let
    $V_{\K}=\textnormal{span}\left(\bm{\Phi}\right)$ where
    $\bm{\Phi}(\X)=[\phi_i(\X)]_{i=1}^{n_\phi}$ and its entries are polynomials
    which satisfy the assumptions of Lemmas \ref{lemma:lagrange_poly_eval} and
    \ref{lemma:polyval_derivatives}. For $\bm{z} = [z_i]_{i=1}^{n_\phi}$, let
    $z_h=\bm{z}^T\bm{\Phi}(\X)\in V_{\K}$ and assume that $u(m(r+2) + n_{\phi})<1$. Evaluating
    $z_h$ and $\partial_{s} z_h$ in precision $u$ yields $\hat{z}_h$
    and $\widehat{\partial_{s} z_h}$ respectively satisfying
    \begin{align}
        \max_{\X\in \K}|z_h - \hat{z}_h| &\leq \gamma_{n_{\phi} + m(r+1)+1}\kappa(V_{\K})\lVert \bm{z} \rVert_{\infty},\\[0.25em]
        \max_{\X\in \K}|\partial_{s}z_h - \widehat{\partial_{s}z_h}| &\leq
        \gamma_{n_{\phi} + (m-1)(r+1)}\kappa(\partial_{s}\bm{\Phi})\kappa(\partial_{s}V_{\K})\lVert \bm{z} \rVert_{\infty},\\[0.5em]
        \quad\textnormal{where}\quad
        \kappa(V_{\K}) :=
        \max_{\X\in \K}\
        \lVert\bm{\Phi}(\X)\rVert_1,\quad\kappa(\partial_{s}V_{\K}) :=& \max_{\X\in
        \K}\
        \lVert\partial_{s}\bm{\Phi}(\X)\rVert_1,\quad\textnormal{and}\quad
        \kappa(\partial_{s}\bm{\Phi}) :=
        \max_{i=1,\dots,n_{\phi}}\kappa(\partial_{s}\phi_i).
    \end{align}
\end{lemma}

\begin{proof}
    Since $z_h=\bm{z}^T\bm{\Phi}$ and
    $\partial_{s}z_h=\bm{z}^T(\partial_{s}\bm{\Phi})$, we can apply the
    error bound for inner products (cf.~equation \eqref{eq:inner_prods}) and
    Lemma \ref{lemma:lagrange_poly_eval} to obtain the first bound:
    \begin{align}
        \max_{\X\in \K}|z_h - \hat{z}_h| \leq (\gamma_{n_{\phi}} + \gamma_{m(r+1)+1} +
    \gamma_{n_{\phi}}\gamma_{m(r+1)+1})\max_{\X\in \K}|\bm{z}|^T|\bm{\Phi}| \leq
    \gamma_{n_{\phi} +
        m(r+1)+1}\lVert \bm{z} \rVert_{\infty} \max_{x\in \K}\lVert \bm{\Phi} \rVert_1.
    \end{align}
    Here we have used the relation $\gamma_a + \gamma_{b} + \gamma_a\gamma_{b}
    \leq \gamma_{a+b}$, cf.~Lemma \ref{lemma:higham_manipulation_theta_gamma}.
    Similarly, combining the error bound \eqref{eq:inner_prods} for inner
    products with Lemma \ref{lemma:polyval_derivatives} yields the second bound
    in the lemma and concludes the proof:
    \begin{align}
        \max_{\X\in \K}|\partial_{s}z_h - \widehat{\partial_{s}z_h}| &\leq
        \gamma_{n_{\phi} +
        (m-1)(r+1)} \kappa(\partial_{s}\bm{\Phi})\kappa(\partial_{s}V_{\K})\lVert \bm{z}
        \rVert_{\infty}.
    \end{align}
\end{proof}

We now specialize the above results to Lagrange elements over reference
simplices or boxed cells:
\begin{lemma}
    Let $\K$ be a $d$-dimensional reference simplex or boxed cell and let
    $V_{\K}=\textnormal{span}\left(\bm{\Phi}\right)$, where the entries of
    $\bm{\Phi}$ form the degree-$p$ nodal Lagrange reference basis over $\K$.
    Then, if $u$ is sufficiently small, the evaluation of $\phi_i(\X)$ and its
    derivatives in precision $u$ satisfy, for all $i$ and for all $\X\in \K$,
    the bounds
    \begin{align}
        |\phi_i(\X) - \hat{\phi}_i(\X)| &\lesssim dpu|\phi_i(\X)|,\\[0.25em]
        \max_{\X\in \K}|\partial_{s}\phi_i(\X) -
        \widehat{\partial_{s}\phi_i}(\X)| &\lesssim
        dpu\kappa(\partial_{s}\phi_i)\max_{\X\in
            \K}|\partial_{s}\phi_i(\X)|,
    \end{align}
    Furthermore, the evaluation of a function
    $z_h(\X)=\bm{z}^T\bm{\Phi}(\X)$ and its derivatives satisfy the
    bounds
    \begin{align}
        \max_{\X\in \K}|z_h - \hat{z}_h| &\lesssim p^du\kappa(V_{\K})\lVert
        \bm{z} \rVert_{\infty},\\[0.25em]
        \max_{\X\in \K}|\partial_{s}z_h - \widehat{\partial_{s}z_h}| &\lesssim
        p^du \kappa(\partial_{s}\bm{\Phi})\kappa(\partial_{s}V_{\K})\lVert
        \bm{z} \rVert_{\infty}.
    \end{align}
\end{lemma}

\begin{proof}
    Nodal Lagrange finite elements satisfy the assumptions of Lemmas
    \ref{lemma:lagrange_poly_eval}, \ref{lemma:polyval_derivatives}, and
    \ref{lemma:FEM_function_eval}, see, e.g., \cite[Section~2]{ciarlet2002finite}
    and \cite{nicolaides1972class}. The value of
    $n_{\phi}$ for these elements is well known: we have $n_{\phi}=(p+1)^d$ if
    $\K$ is a boxed cell and $n_{\phi}=\frac{1}{d!}\prod_{j=1}^d(p+j)$ if $\K$
    is a simplex. In both cases, $n_{\phi}=O(p^d)$. Similarly, it is well known
    that $m=pd$ for boxed cells and $m=p$ for simplices. For boxed cells we have
    $r=O(1)$ since each monomial is in the form $\ell_i(\X) =
    \check{x}_j-\check{y}_j$ for some $j\in\{1,\dots,d\}$ and some
    $\check{\bm{y}}\in \K$, and therefore no more than $O(1)$ roundoffs affect
    each monomial evaluation. For simplicies monomials include coordinates from
    all dimensions (cf.~\cite{nicolaides1972class}) so we have $r=O(d)$.
    Therefore, for both simplices and boxed cells we have $O(mr)=O(pd)$. We can
    consequently replace all $\gamma_{O(mr)}$ terms with $O(pdu)$ and all
    $\gamma_{O(n_{\phi}+mr)}$ terms with $O(p^du)$ and the corollary is proved.
\end{proof}

\section{Implementation details.}
\label{appendix_sec:implementation_details}

We found the overall bf16 intrinsics and compiler support to be
limited. To circumvent these limitations we made some implementation choices
which we highlight here for the sake of reproducibility:

\begin{itemize}
    \item We implement mixed-precision matrix and dot products by employing
        Intel intrinsics \cite{intel-intrinsics}. Indeed, compilers\footnote{Not
            just \texttt{g++} (GCC version 14.2.0). The same applies to all
            other compilers we tried: the \texttt{clang++} compiler (LLVM
            version 18.1.8) and the Intel \texttt{icpx} compiler (Intel OneAPI
        version 2024.2.0).} are unable to automatically use the mixed-precision
        fp32/bf16 AVX and AMX accelerators available in the processor.

    \item The processor and the Intel intrinsics \cite{intel-intrinsics} have
        limited vectorization support for bf16: the only operations available
        are casting to and from single precision and an fp32/bf16
        mixed-precision entrywise dot product. When a pure bf16 operation is
        needed we perform it by casting to and from fp16 which has full
        vectorization support.  The fp16 format has a higher precision, but a
        narrower range, therefore this casting increases the chances of
        underflow/overflow. We did not find this to be problematic in our
        experiments. 

    \item We use bit manipulations to cast fp16 to bf16 (with full subnormal
        support) and we employ Intel AVX512 intrinsics for casting from fp32 to
        bf16 \cite{intel-intrinsics}. The reason is that we found the compiled
        version of these operations to be either unnecessarily slow or bugged.

    \item In the AMX-bf16 and AVX512-bf16 kernels we use mixed-precision
        accelerators for the accumulation of matrix-matrix products into the
        local tensors. However, only the AMX kernels employ these accelerators
        for the FE function evaluations as well. In the AVX512-bf16 kernels
        these evaluation are performed entirely in (vectorized) single
        precision. The reason is that while AMX-bf16 accelerators are much
        faster than vectorized computations in any precision, the evaluations
        performed with AVX512-bf16 accelerators are sometimes slower than their
        fully single precision equivalents (for results and a discussion about
        this issue see Section \ref{subsec:MP_kernels_speedups}).

\end{itemize}

\end{document}